\documentclass[12pt]{amsart}
\usepackage{amsfonts}
\usepackage{bbm}
\usepackage{mathrsfs}
\usepackage{amssymb}
\usepackage{amsmath}
\usepackage{url}

\numberwithin{equation}{section}

\DeclareMathOperator{\supp}{supp}

\begin{document}

\baselineskip 16.1pt \hfuzz=6pt

\theoremstyle{plain}
\newtheorem{theorem}{Theorem}[section]
\newtheorem{prop}[theorem]{Proposition}
\newtheorem{lemma}[theorem]{Lemma}
\newtheorem{corollary}[theorem]{Corollary}
\newtheorem{example}[theorem]{Example}

\theoremstyle{definition}
\newtheorem{definition}[theorem]{Definition}
\newtheorem{remark}[theorem]{Remark}

\renewcommand{\theequation}
{\thesection.\arabic{equation}}

\allowdisplaybreaks

\newcommand{\XX}{X}
\newcommand{\GG}{\mathop G \limits^{    \circ}}
\newcommand{\GGs}{{\mathop G\limits^{\circ}}}
\newcommand{\GGtheta}{{\mathop G\limits^{\circ}}_{\theta}}
\newcommand{\xoneandxtwo}{X_1\times\mathcal X_2}

\newcommand{\GGp}{{\mathop G\limits^{\circ}}}

\newcommand{\GGpp}{{\mathop G\limits^{\circ}}_{\theta_1,\theta_2}}

\newcommand{\e}{\varepsilon}
\newcommand{\bmo}{{\rm BMO}}
\newcommand{\vmo}{{\rm VMO}}
\newcommand{\cmo}{{\rm CMO}}
\newcommand{\Z}{\mathbb{Z}}
\newcommand{\N}{\mathbb{N}}
\newcommand{\R}{\mathbb{R}}
\newcommand{\C}{\mathbb{C}}

\newcommand{\dxone}{d\mu(x_{1})}
\newcommand{\dxtwo}{d\mu(x_{2})}
\newcommand{\dyone}{d\mu(y_{1})}
\newcommand{\dytwo}{d\mu(y_{2})}
\newcommand{\dzone}{d\mu(z_{1})}
\newcommand{\dztwo}{d\mu(z_{2})}
\newcommand{\dyonep}{d\mu(y_{1}^{'})}
\newcommand{\dytwop}{d\mu(y_{2}^{'})}



\newcommand{\hardy}{H^{1}(\XX\times\XX)}

\newcommand{\poissonone}{\frac{\displaystyle 2^{-k_1\epsilon}}{\displaystyle (2^{-k_1}+\rho(x_1,y_1))^{1+\epsilon}}}
\newcommand{\poissontwo}{\frac{\displaystyle 2^{-k_2\epsilon}}{\displaystyle (2^{-k_2}+\rho(x_2,y_2))^{1+\epsilon}}}

\newcommand{\smoothxone}{\bigg(\frac{\displaystyle \rho(x_1,x_1^{'})}{\displaystyle (2^{-k_1}+\rho(x_1,y_1))}\bigg)^{\epsilon}}
\newcommand{\smoothxtwo}{\bigg(\frac{\displaystyle \rho(x_2,x_2^{'})}{\displaystyle (2^{-k_2}+\rho(x_2,y_2))}\bigg)^{\epsilon}}
\newcommand{\smoothyone}{\bigg(\frac{\displaystyle \rho(y_1,y_1^{'})}{\displaystyle (2^{-k_1}+\rho(x_1,y_1))}\bigg)^{\epsilon}}
\newcommand{\smoothytwo}{\bigg(\frac{\displaystyle \rho(y_2,y_2^{'})}{\displaystyle (2^{-k_2}+\rho(x_2,y_2))}\bigg)^{\epsilon}}

\newcommand{\conditionxone}{\rho(x_1,x_1^{'})\leq \frac{1}{\displaystyle 2A}(2^{-k_1}+\rho(x_1,y_1))}
\newcommand{\conditionxtwo}{\rho(x_2,x_2^{'})\leq \frac{1}{\displaystyle 2A}(2^{-k_1}+\rho(x_2,y_2))}
\newcommand{\conditionyone}{\rho(y_1,y_1^{'})\leq \frac{1}{\displaystyle 2A}(2^{-k_2}+\rho(x_1,y_1))}
\newcommand{\conditionytwo}{\rho(y_2,y_2^{'})\leq \frac{1}{\displaystyle 2A}(2^{-k_2}+\rho(x_2,y_2))}

\newcommand{\xy}{(x_1,x_2,y_1,y_2)}
\newcommand{\xyp}{(x_1,x_2,y_{\tau_1^{'}}^{k_1^{'},v_1^{'}},y_{\tau_2^{'}}^{k_2^{'},v_2^{'}})}



\newcommand{\orth}{2^{-|k_1-k_1^{'}|\epsilon^{'}}2^{-|k_2-k_2^{'}|\epsilon^{'}}}
\newcommand{\orthpone}{\frac{\displaystyle 2^{-(k_1\wedge k_1^{'})\epsilon}}{\displaystyle (2^{-(k_1\wedge k_1^{'})}+\rho(x_1,y_1))^{1+\epsilon}}}
\newcommand{\orthptwo}{\frac{\displaystyle 2^{-(k_2\wedge k_2^{'})\epsilon}}{\displaystyle (2^{-(k_2\wedge k_2^{'})}+\rho(x_2,y_2))^{1+\epsilon}}}
\newcommand{\orthR}{
\bigg( \frac{\displaystyle \mu(\Qone)}{\displaystyle\mu(\Qonep)}
\wedge \frac{\displaystyle \mu(\Qonep)}{\displaystyle\mu(\Qone)}
\bigg)^{\epsilon^{'}} \bigg( \frac{\displaystyle
\mu(\Qtwo)}{\displaystyle\mu(\Qtwop)} \wedge \frac{\displaystyle
\mu(\Qtwop)}{\displaystyle\mu(\Qtwo)} \bigg)^{\epsilon^{'}} }
\newcommand{\orthQone}{\bigg( \frac{\displaystyle \mu(\Qone)}{\displaystyle\mu(\Qonep)} \wedge \frac{\displaystyle \mu(\Qonep)}{\displaystyle\mu(\Qone)} \bigg)^{\epsilon^{'}}}
\newcommand{\orthQtwo}{\bigg( \frac{\displaystyle \mu(\Qtwo)}{\displaystyle\mu(\Qtwop)} \wedge \frac{\displaystyle \mu(\Qtwop)}{\displaystyle\mu(\Qtwo)} \bigg)^{\epsilon^{'}}}
\newcommand{\orthponeR}{\frac{\displaystyle (\muQone\vee\muQonep)^{\epsilon}}{\displaystyle (\muQone\vee\muQonep+dist(\Qone,\Qonep))^{1+\epsilon}}}
\newcommand{\orthptwoR}{\frac{\displaystyle (\muQtwo\vee\muQtwop)^{\epsilon}}{\displaystyle (\muQtwo\vee\muQtwop+dist(\Qtwo,\Qtwop))^{1+\epsilon}}}
\newcommand{\orthRfinal}{
\bigg(\frac{\displaystyle\mu(\Qone)}{\displaystyle\mu(\Qonep)}\wedge\frac{\displaystyle\mu(\Qonep)}{\displaystyle\mu(\Qone)}\bigg)^{1+\epsilon^{'}}
\bigg(\frac{\displaystyle\mu(\Qtwo)}{\displaystyle\mu(\Qtwop)}\wedge\frac{\displaystyle\mu(\Qtwop)}{\displaystyle\mu(\Qtwo)}\bigg)^{1+\epsilon^{'}}
}
\newcommand{\orthQonefinal}{\bigg( \frac{\displaystyle \mu(\Qone)}{\displaystyle\mu(\Qonep)} \wedge \frac{\displaystyle \mu(\Qonep)}{\displaystyle\mu(\Qone)} \bigg)^{1+\epsilon^{'}}}
\newcommand{\orthQtwofinal}{\bigg( \frac{\displaystyle \mu(\Qtwo)}{\displaystyle\mu(\Qtwop)} \wedge \frac{\displaystyle \mu(\Qtwop)}{\displaystyle\mu(\Qtwo)} \bigg)^{1+\epsilon^{'}}}
\newcommand{\orthponeRfinal}{\frac{\displaystyle (\muQone\vee\muQonep)^{1+\epsilon}}{\displaystyle (\muQone\vee\muQonep+dist(\Qone,\Qonep))^{1+\epsilon}}}
\newcommand{\orthptwoRfinal}{\frac{\displaystyle (\muQtwo\vee\muQtwop)^{1+\epsilon}}{\displaystyle (\muQtwo\vee\muQtwop+dist(\Qtwo,\Qtwop))^{1+\epsilon}}}



\newcommand{\yone}{y_{\tau_1}^{k_1}}
\newcommand{\ytwo}{y_{\tau_2}^{k_2}}

\newcommand{\qone}{Q_{\tau_1}^{k_1}}
\newcommand{\qtwo}{Q_{\tau_2}^{k_2}}

\newcommand{\mR}{m_{\qone\times\qtwo}}
\newcommand{\muqone}{\mu(\qone)}
\newcommand{\muqtwo}{\mu(\qtwo)}

\newcommand{\pone}{\frac{\displaystyle 2^{-k_1\epsilon}}{\displaystyle (2^{-k_1}+\rho(x_1,y_{\tau_1}^{k_1}))^{1+\epsilon}}}
\newcommand{\ptwo}{\frac{\displaystyle 2^{-k_2\epsilon}}{\displaystyle (2^{-k_2}+\rho(x_2,y_{\tau_2}^{k_2}))^{1+\epsilon}}}

\newcommand{\ponep}{\frac{\displaystyle 2^{-k_1\epsilon}}{\displaystyle (2^{-k_1}+\rho(x_1^{'},y_{\tau_1}^{k_1}))^{1+\epsilon}}}
\newcommand{\ptwop}{\frac{\displaystyle 2^{-k_2\epsilon}}{\displaystyle (2^{-k_2}+\rho(x_2^{'},y_{\tau_2}^{k_2}))^{1+\epsilon}}}

\newcommand{\poney}{\frac{\displaystyle 2^{-k_1\epsilon}}{\displaystyle (2^{-k_1}+\rho(y_1,y_{\tau_1}^{k_1}))^{1+\epsilon}}}
\newcommand{\ptwoy}{\frac{\displaystyle 2^{-k_2\epsilon}}{\displaystyle (2^{-k_2}+\rho(y_2,y_{\tau_2}^{k_2}))^{1+\epsilon}}}

\newcommand{\sone}{ \bigg(\frac{\displaystyle \rho(x_1,y_1)}{\displaystyle 2^{-k_1}+\rho(x_1,y_{\tau_1}^{k_1})} \bigg)^{\epsilon^{'}} }
\newcommand{\stwo}{ \bigg(\frac{\displaystyle \rho(x_2,y_2)}{\displaystyle 2^{-k_2}+\rho(x_2,y_{\tau_2}^{k_2})} \bigg)^{\epsilon^{'}} }

\newcommand{\pdone}{\frac{\displaystyle 2^{-j_1\epsilon}}{\displaystyle (2^{-j_1}+\rho(x_1,y_1))^{1+\epsilon}}}
\newcommand{\pdtwo}{\frac{\displaystyle 2^{-j_2\epsilon}}{\displaystyle (2^{-j_2}+\rho(x_2,y_2))^{1+\epsilon}}}

\newcommand{\pdonep}{\frac{\displaystyle 2^{-j_1\epsilon}}{\displaystyle (2^{-j_1}+\rho(x_1,\yone))^{1+\epsilon}}}
\newcommand{\pdtwop}{\frac{\displaystyle 2^{-j_2\epsilon}}{\displaystyle (2^{-j_2}+\rho(x_2,\ytwo))^{1+\epsilon}}}


\pagestyle{myheadings}\markboth{\rm\small Yongsheng Han, Ji Li
and Lesley A.~Ward}{\rm\small Product $H^p$, $\cmo^p$ and
$\vmo$ on spaces of homogeneous type}

\title[Hardy space theory on spaces of homogeneous type via
wavelet bases]{Hardy space theory on spaces of homogeneous type\\
via orthonormal wavelet bases}

\author{Yongsheng Han, Ji Li and Lesley A.~Ward}

\thanks{The second and third authors are supported by the
Australian Research Council under Grant No.~ARC-DP120100399.
The second author was also supported by the NNSF of China Grant
No.~11001275, by China Postdoctoral Science Foundation funded
project Grant No.~201104383, and by the Fundamental Research
Funds for the Central Universities, Grant No.~11lgpy56. Parts
of this paper were written while the second author was a member
of the Department of Mathematics, Sun Yat-sen University,
Guangzhou, 510275, P.R. China.}

\subjclass[2010]{Primary 42B35; Secondary 43A85, 42B25, 42B30}

\keywords{Spaces of homogeneous type, orthonormal basis, test
function space, distributions, Calder\'on reproducing formula,
wavelet expansion, product Hardy space, Carleson measure space,
BMO, VMO, duality.}

\begin{abstract}
    In this paper, using the remarkable orthonormal wavelet
    basis constructed recently by Auscher and Hyt\"onen, we
    establish the theory of product Hardy spaces on spaces
    ${\widetilde X} = X_1\times X_2\times\cdot \cdot\cdot\times
    X_n$, where each factor $X_i$ is a space of homogeneous
    type in the sense of Coifman and Weiss. The main tool we
    develop is the Littlewood--Paley theory on $\widetilde X$,
    which in turn is a consequence of a corresponding theory on
    each factor space. We define the square function for this
    theory in terms of the wavelet coefficients. The Hardy
    space theory developed in this paper includes
    product~$H^p$, the dual $\cmo^p$ of $H^p$ with the special
    case $\bmo = \cmo^1$, and the predual $\vmo$ of $H^1$. We
    also use the wavelet expansion to establish the
    Calder\'on--Zygmund decomposition for product $H^p$, and
    deduce an interpolation theorem. We make no additional
    assumptions on the quasi-metric or the doubling measure for
    each factor space, and thus we extend to the full generality
    of product spaces of homogeneous type the aspects of both
    one-parameter and multiparameter theory involving the
    Littlewood--Paley theory and function spaces. Moreover, our
    methods would be expected to be a powerful tool for
    developing wavelet analysis on spaces of homogeneous type.
\end{abstract}

\maketitle

\tableofcontents

\section{Introduction}\label{sec:introduction}
\setcounter{equation}{0} We work on wavelet analysis in the
setting of product spaces of homogeneous type in the sense of
Coifman and Weiss~\cite{CW1}, where each factor is of the form
$(X,d,\mu)$ with $d$ a quasi-metric and $\mu$ a doubling
measure. We make no additional assumptions on $d$ or $\mu$.
After recalling the systems of dyadic cubes of Hyt\"onen and
Kairema~\cite{HK} and the orthonormal wavelet basis of Auscher
and Hyt\"onen~\cite{AH}, we define an appropriate class of test
functions and the induced class of distributions on spaces of
homogeneous type. We prove that the Auscher--Hyt\"onen wavelets
are test functions, and that the Auscher--Hyt\"onen reproducing
formula for~$L^p$ also holds for our test functions and
distributions. We show that the kernels of certain wavelet
operators~$D_k$ defined in terms of these wavelets satisfy
decay and smoothness conditions similar to those of our test
functions. These facts play a crucial role in our development
of the Littlewood--Paley theory and function spaces, later in
our paper.

We define the discrete Littlewood--Paley square function via
the Auscher--Hyt\"onen wavelet coefficients. In order to
establish its $L^p$-boundedness, we also introduce a different,
\emph{continuous} Littlewood--Paley square function defined in
terms of the wavelet operators~$D_k$. We prove that the
discrete and continuous square functions have equivalent norms,
by first establishing some inequalities of Plancherel--P\'olya
type. We develop this Littlewood--Paley theory first in the
one-parameter setting, and then for product spaces.

For $p$ in a range that depends on the upper dimensions of the
spaces $X_1$ and $X_2$ and strictly includes the range $1 \leq
p < \infty$, we define the product Hardy space $H^p(X_1\times
X_2)$ as the class of distributions whose discrete
Littlewood--Paley square functions are in $L^p(X_1 \times
X_2)$. (Here we write only two factors, for simplicity, but our
results extend to $n$ factors.) For $p$ in this range with $p
\leq 1$, we define the Carleson measure space $\cmo^p(X_1
\times X_2)$ via the Auscher--Hyt\"onen wavelet coefficients,
as a subset of our space of distributions, and prove the
duality $(H^p(X_1\times X_2))' = \cmo^p(X_1 \times X_2)$ by
means of sequence spaces that form discrete analogues of these
spaces. This duality result includes the special case
$(H^1(X_1\times X_2))' = \bmo(X_1 \times X_2)$. We define the
space $\vmo(X_1 \times X_2)$ of functions of vanishing mean
oscillation, also in terms of the Auscher--Hyt\"onen wavelet
coefficients, and prove the duality $(\vmo(X_1\times X_2))' =
H^1(X_1 \times X_2)$ by adapting an argument of
Lacey--Terwilleger--Wick~\cite{LTW}. Using the wavelet
expansion, we establish the Calder\'on--Zygmund decomposition
for functions in our Hardy spaces $H^p(X_1 \times X_2)$, again
for a suitable range of~$p$ that strictly includes $1 \leq p <
\infty$. As a consequence, we deduce an interpolation theorem
for linear operators from these product Hardy spaces to
Lebesgue spaces on $X_1 \times X_2$.

We now set our work in context. As Meyer remarked in his
preface to~\cite{DH}, \emph{``One is amazed by the dramatic
changes that occurred in analysis during the twentieth century.
In the 1930s complex methods and Fourier series played a
seminal role. After many improvements, mostly achieved by the
Calder\'on--Zygmund school, the action takes place today on
spaces of homogeneous type. No group structure is available,
the Fourier transform is missing, but a version of harmonic
analysis is still present. Indeed the geometry is conducting
the analysis.''} Spaces of homogeneous type were introduced by
Coifman and Weiss in the early 1970s, in~\cite{CW1}. We say
that $(\XX,d,\mu)$ is a {\it space of homogeneous type} in the
sense of Coifman and Weiss if $d$ is a quasi-metric on~$\XX$
and $\mu$ is a nonzero measure satisfying the doubling
condition. A \emph{quasi-metric}~$d$ on a set~$\XX$ is a
function $d:\XX\times\XX\longrightarrow[0,\infty)$ satisfying
(i) $d(x,y) = d(y,x) \geq 0$ for all $x$, $y\in\XX$; (ii)
$d(x,y) = 0$ if and only if $x = y$; and (iii) the
\emph{quasi-triangle inequality}: there is a constant $A_0\in
[1,\infty)$ such that for all $x$, $y$, $z\in\XX$, 
\begin{eqnarray}\label{eqn:quasitriangleineq}
    d(x,y)
    \leq A_0 [d(x,z) + d(z,y)].
\end{eqnarray}

We define the quasi-metric ball by $B(x,r) := \{y\in X: d(x,y)
< r\}$ for $x\in X$ and $r > 0$. Note that the quasi-metric, in
contrast to a metric, may not be H\"older regular and
quasi-metric balls may not be open.
We say that a nonzero measure $\mu$ satisfies the
\emph{doubling condition} if there is a constant $C_\mu$ such
that for all $x\in\XX$ and $r > 0$,
\begin{eqnarray}\label{doubling condition}
   \mu(B(x,2r))
   \leq C_\mu \mu(B(x,r))
   < \infty.
\end{eqnarray}
We point out that the doubling condition (\ref{doubling
condition}) implies that there exist positive constants
$\omega$ (the \emph{upper dimension} of~$\mu$) and $C$ such
that for all $x\in X$, $\lambda\geq 1$ and $r > 0$,
\begin{eqnarray}\label{upper dimension}
    \mu(B(x, \lambda r))
    \leq C\lambda^{\omega} \mu(B(x,r)).
\end{eqnarray}

Spaces of homogeneous type include many special spaces in
analysis and have many applications in the theory of singular
integrals and function spaces; see~\cite{CW2,NS1,NS2} for more
detail. For instance, Coifman and Weiss introduced the atomic
Hardy space on $(\XX,d,\mu)$ and proved that if $T$ is a
Calder\'on--Zygmund singular integral operator and is bounded
on $L^2(X)$, then $T$ is bounded from $H^p(X)$ to $L^p(X)$ for
some $p\leq 1$.

However, for some applications, additional assumptions were
imposed on these general spaces of homogeneous type, because as
noted above the original quasi-metric~$d$ may have no
regularity and quasi-metric balls, even Borel sets, may not be
open. For example, to establish the maximal function
characterization of the Hardy space introduced by Coifman and
Weiss, Mac\'ias and Segovia proved in~\cite{MS} that one can
replace the quasi-metric~$d$ by another quasi-metric $d'$
on~$X$ such that the topologies induced on~$\XX$ by $d$
and~$d'$ coincide, and $d'$ has the following regularity
property:
\begin{eqnarray}\label{smooth metric}
    |d'(x,y) - d'(x',y)|
    \le C_0 \, d'(x,x')^\theta \,
        [d'(x,y) + d'(x',y)]^{1 - \theta}
\end{eqnarray}
for some constant~$C_0,$ some regularity exponent
$\theta\in(0,1)$, and for all $x$, $x'$, $y\in X$. Moreover, if
quasi-metric balls are defined by this new quasi-metric $d'$,
that is, $B'(x,r) := \{y\in X: d'(x,y) < r\}$ for $r > 0$, then
the measure $\mu$ satisfies the following property:
\begin{eqnarray}\label{regular}
    \mu(B'(x,r))\sim r.
\end{eqnarray}
Note that property~\eqref{regular} is much stronger than the
doubling condition. Mac\'{i}as and Segovia established the
maximal function characterization for Hardy spaces $H^p(\XX)$
with $(1 + \theta)^{-1} < p \leq 1$, on spaces of homogeneous
type~$(\XX,d',\mu)$ that satisfy the regularity
condition~\eqref{smooth metric} on the metric~$d'$ and
property~\eqref{regular} on the measure~$\mu$; see~\cite{MS}.

A fundamental result for these spaces $(\XX, d', \mu)$ is the
$T(b)$ theorem of David--Journ\'{e}--Semmes~\cite{DJS}. The
crucial tool in the proof of the $T(b)$ theorem is the
existence of a suitable approximation to the identity. The
construction of such an approximation to the identity is due to
Coifman. More precisely, take a smooth function $h$ defined on
$[0, \infty)$, equal to~1 on $[1, 2]$, and equal to 0 on $[0,
1/2]$ and on $[4, \infty)$. Let $T_k$ be the operator with
kernel $2^kh(2^k d'(x,y))$. The property~\eqref{regular} of the
measure~$\mu$ implies that $C^{-1} \leq T_k(1) \leq C$ for some
$C$ with $0 < C < \infty$. Let $M_k$ and $W_k$ be the operators
of multiplication by $1/T_k(1)$ and $\{T_k[1/T_k(1)]\}^{-1}$,
respectively, and let $S_k := M_kT_kW_kT_kM_k$. It is clear
that the regularity property~\eqref{smooth metric} on the
metric~$d'$ and property~\eqref{regular} on the measure $\mu$
imply that the kernel $S_k(x,y)$ of~$S_k$ satisfies the
following conditions: for some constants $C > 0$ and $\e > 0$,
\begin{eqnarray*}
    &\textup{(i)}& S_k(x,y) = 0 {\rm\ for\ } d'(x,y) \geq C2^{-k},
        {\rm\ and\ } \| S_k\|_{\infty} \leq C2^k,\\
    &\textup{(ii)}& |S_k(x,y)-S_k(x',y)|
        \leq C2^{k(1+\e)}d'(x,x')^{\e},\\
    &\textup{(iii)}& |S_k(x,y)-S_k(x,y')|
        \leq C2^{k(1+\e)}d'(y,y')^{\e}, \quad\text{and}\\
    &\textup{(iv)}& \int_{X}S_k(x,y) \, d\mu(y)
        = 1
        = \int_{X}S_k(x,y) \, d\mu(x).
\end{eqnarray*}

Let $D_k := S_{k+1} - S_k$. In \cite{DJS}, the
Littlewood--Paley theory for $L^p(X)$, $1 < p < \infty$, was
established; namely, if $\mu(X) = \infty$ and $\mu(B(x,r))>0$
for all $x\in X$ and $r>0,$ then for each $p$ with $1 < p <
\infty$ there exists a positive constant $C_p$ such that
\[
    C_p^{-1}\|f\|_p
    \leq \big\|\big\{\sum_{k}|D_k(f)|^2\big\}^{\frac{1}{2}}\big\|_p
    \leq C_p\| f\|_p.
\]
The above estimates were the key tool for proving the
$T(1)$~theorem on $(\XX, d', \mu)$; see \cite{DJS} for more
detail. Later, the Calder\'on reproducing formula, test
function spaces and distributions, the Littlewood--Paley
theory, and function spaces on $(X, d', \mu)$ were developed in
\cite{H1}, \cite{HS} and~\cite{H2}. However, in those works
wavelet bases were replaced by frames, which in many
applications offer the same service; see \cite{DH} for more
details.

In \cite{NS1}, Nagel and Stein developed the product $L^p$ $(1
< p < \infty)$ theory in the setting of the
Carnot--Carath\'eodory spaces formed by vector fields
satisfying H\"{o}rmander's finite rank condition. The
Carnot--Carath\'eodory spaces studied in~\cite{NS1} are spaces
of homogeneous type with a smooth quasi-metric $d$ and a
measure~$\mu$ satisfying the conditions $\mu(B(x, sr)) \sim
s^{m+2} \mu(B(x,r))$ for $s\geq 1$ and $\mu(B(x, sr)) \sim
s^4\mu(B(x,r))$ for $s\leq 1.$ These conditions on the measure
are weaker than property~\eqref{regular} but are still stronger
than the original doubling condition~\eqref{doubling
condition}. In~\cite{HMY}, motivated by the work of Nagel and
Stein, Hardy spaces were developed on spaces of homogeneous
type with a regular quasi-metric and a measure satisfying the
above conditions. Recently, in~\cite{HLL2}, it was observed
that Coifman's construction of an approximation to the identity
still works on spaces of homogeneous type $(X, d, \mu)$ where
the quasi-metric $d$ satisfies the H\"older regularity
property~\eqref{smooth metric} but the measure $\mu$ only needs
to be doubling. Specifically, the kernel $S_k(x,y)$ of the
approximation to the identity $S_k$ satisfies the following
conditions: there exist constants $C > 0$ and $\theta > 0$ such
that for all $k\in\mathbb{Z}$ and all $x$, $x'$, $y$, $y'\in
X$,
\begin{eqnarray*}
    &\textup{(i)}& S_k(x,y)=0\ {\rm for}\ d(x,y)\geq C2^{-k},\ {\rm and}\
        |S_k(x,y)| \leq C \, \frac{1}{V_{2^{-k}}(x)+V_{2^{-k}}(y)},\\
    &\textup{(ii)}& |S_k(x,y)-S_k(x',y)| \leq
        C2^{k\theta}d(x,x')^{\theta} \,
        \frac{1}{V_{2^{-k}}(x)+V_{2^{-k}}(y)},\\
    &\textup{(iii)}& {\rm property}\ \textup{(ii)}\
        {\rm also\ holds\ with\ } x\
        {\rm and}\ y\ {\rm interchanged},\\
    &\textup{(iv)}& \big|[S_k(x,y)-S_k(x,y')] - [S_k(x',y)-S_k(x',y')]\big|\\
        &&\hskip1cm\leq
        C2^{2k\theta}d(x,x')^{\theta}d(y,y')^{\theta} \,
        \frac{1}{V_{2^{-k}}(x)+V_{2^{-k}}(y)}, \quad\text{and}\nonumber\\
    &\textup{(v)}& \int_X S_k(x,y) \, d\mu(y)
        = 1
        = \int_X S_k(x,y) \, d\mu(x),
\end{eqnarray*}
where $V_r(x) := \mu(B(x,r))$.

For spaces of homogeneous type $(X, d, \mu)$ with some
additional assumptions, the one-parameter and product Hardy
spaces were developed in~\cite{HMY} and ~\cite{HLL2},
respectively.

A natural question arises: can one develop the theory of the
spaces $H^p$ and $\bmo$ on spaces of homogeneous type in the
sense of Coifman and Weiss, with only the original
quasi-metric~$d$ and a doubling measure~$\mu$?

Recently, Auscher and Hyt\"onen constructed an orthonormal
wavelet basis with H\"older regularity and exponential decay
for spaces of homogeneous type~\cite{AH}. This result is
remarkable since there are no additional assumptions other than
those defining spaces of homogeneous type in the sense of
Coifman and Weiss.

Auscher and Hyt\"onen's orthonormal wavelet bases open the door
for developing wavelet analysis on spaces of homogeneous type
in the sense of Coifman and Weiss. Motivated by Auscher and
Hyt\"onen's work, the purpose of the current paper is to answer
the above question. Specifically, we will employ a unified
approach to establish a product Hardy space theory on
${\widetilde X} = X_1 \times X_2\times\cdots\times X_n$, where
each factor is a space of homogeneous type in the sense of
Coifman and Weiss. It was well known that any analysis of the
product Hardy space on ${\widetilde X} = X_1\times
X_2\times\cdots\times X_n$ must be based, to start with, on a
formulation on each factor~$X_j.$ The Hardy space on $X_j$ is
then defined by developing the Littlewood--Paley theory
on~$X_j$. Our approach includes the following five steps.

\textbf{1. Introduce the spaces of test functions and
distributions.} In the classical case, the relevant spaces of
test functions and distributions are just Schwartz test
functions and the class of tempered distributions. In order to
study the Calder\'on reproducing formula associated with the
$T(b)$ theorem, the new test function and distribution spaces
were first introduced on Euclidean spaces in~\cite{H1}, and on
spaces of homogeneous type, where the quasi-metric~$d$
satisfies the H\"older regularity condition~\eqref{smooth
metric} and the measure $\mu$ satisfies
condition~\eqref{regular}, in~\cite{HS}. See \cite{HMY}
and~\cite{HLL2}, respectively, for spaces of test functions and
distributions on spaces of homogeneous type with additional
assumptions. In this paper, we will introduce test functions
and distributions on spaces of homogeneous type in the sense of
Coifman and Weiss. These spaces include all those considered
previously.

\textbf{2. Establish the wavelet reproducing formula on test
functions and on distributions.} The classical Calder\'on
reproducing formula was first used by Calder\'on in \cite{C}.
Such a reproducing formula is a powerful tool, particularly in
the theory of wavelet analysis. See~\cite{M1}. Using Coifman's
decomposition of the identity operator, as mentioned above,
David, Journ\'e and Semmes \cite{DJS} gave a Calder\'on-type
reproducing formula which was a key tool in proving the $T(b)$
theorem on $\mathbb{R}^n$ and the $T(1)$ theorem on spaces of
homogeneous type with the conditions~\eqref{smooth metric}
and~\eqref{regular}. See \cite{HMY} and~\cite{HLL2} for the
continuous and discrete Calder\'on reproducing formulas on
spaces of homogeneous type with additional assumptions. As
mentioned above, Auscher and Hyt\"onen established a wavelet
expansion on $L^2(X)$ (and on $L^p(X)$, $1 < p < \infty$). For
our purposes, we will show that the wavelet expansion
constructed in~\cite{AH} also converges in both the test
function and distribution spaces.

As Meyer pointed out in~\cite{M1}, \emph{``The wavelet bases
are universally applicable: `everything that comes to hand',
whether function or distribution, is the sum of a wavelet
series and, contrary to what happens with Fourier series, the
coefficients of the wavelet series translate the properties of
the function or distribution simply, precisely and
faithfully.''} In particular, our results provide such wavelet
expansions for test functions and for distributions, and are
used below to introduce square functions and develop the
Littlewood--Paley theory.

\textbf{3. Develop the Littlewood--Paley theory.} Based on the
wavelet expansion provided in~\cite{AH}, one can formally
introduce two kinds of square functions, namely, the discrete
version defined in terms of wavelet coefficients and the
continuous version defined via wavelet operators~$D_k$
(different from the operators $D_k = S_{k+1} - S_k$ mentioned
above). To show that the $L^p$ norms of these square functions
are equivalent, for a suitable range of $p$, we need a
Plancherel--P\'olya inequality. The classical
Plancherel--P\'olya inequality says that the $L^p$ norm of a
function~$f$ whose Fourier transform has compact support is
equivalent to the $\ell^p$ norm of the restriction of $f$ to an
appropriate lattice. This kind of inequality was first proved
in \cite{H2} on spaces of homogeneous type with the
conditions~\eqref{smooth metric} and~\eqref{regular}, and in
\cite{HMY} and [HLL2], respectively, for the one-parameter and
multiparameter cases with some additional assumptions. As a
consequence of our Plancherel--P\'olya type inequalities, the
Hardy space on spaces of homogeneous type in the sense of
Coifman and Weiss is well defined. In particular, as in the
classical case, $H^p=L^p$ for $1<p<\infty.$

\textbf{4. Introduce the generalized Carleson measure space.}
It is well known that in the classical one-parameter case, the
space $\bmo$, as the dual of $H^1$, can be characterized by
Carleson measures. Moreover, in~\cite{CF1} Chang and Fefferman
proved that the dual of product $H^1$ is characterized by
product Carleson measures. The generalized Carleson measure
space $\cmo^p$, as the dual of the product $H^p$, was
introduced in \cite{HLL1} and \cite{HLL2} on spaces of
homogeneous type with some additional assumptions. In the
current paper, working in the setting of spaces of homogeneous
type in the sense of Coifman and Weiss with no additional
assumptions, we introduce $\cmo^p$ in terms of wavelet
coefficients, and prove that $\cmo^p$ is the dual of $H^p$. In
particular, $\cmo^1 = \bmo$ is the dual of $H^1$. Moreover, we
also introduce the space $\vmo$ and show that $\vmo$ is the
predual of $H^1$.

\textbf{5. Establish the Calder\'on--Zygmund decomposition.}
The classical Calder\'on--Zygmund decomposition played a
crucial role in developing Calder\'on--Zygmund operator theory.
This decomposition has many applications in harmonic analysis
and partial differential equations. Such a decomposition for
product Euclidean spaces was first provided by Chang and
Fefferman in~\cite{CF2}. The main tool used in \cite{CF2} is
the atomic decomposition. In the current  paper, applying the
wavelet expansion constructed in \cite{AH}, we establish the
Calder\'on--Zygmund decomposition on product $H^p$ on spaces of
homogeneous type with no additional assumptions. As a
consequence, we obtain the interpolation of operators that are
bounded from Hardy spaces to Lebesgue spaces, and of operators
that are bounded on Hardy spaces.

\medskip
We note that in the original work on extending the
Calder\'on--Zygmund theory to spaces of homogeneous
type~$(X,d',\mu)$, the philosophy was as follows: Coifman
constructed the approximations to the identity~$S_k$, which
were used in~\cite{DJS} to define the continuous square
function and to establish the Littlewood--Paley theory. Later
the discrete Calder\'on reproducing formula was introduced and
the Littlewood--Paley theory for the classical function spaces
were established in \cite{H1} and \cite{HS}, respectively. By
contrast, in our setting of $(X,d,\mu)$ with the original
quasi-metric~$d$, we begin with the discrete wavelet
reproducing formula (Theorem~\ref{thm reproducing formula test
function}) and define the discrete square function~$S(f)$ in
terms of wavelet coefficients
(Definition~\ref{def:discrete_square_function}). However, there
does not seem to be a direct proof of the Littlewood--Paley
theory for~$S(f)$. The question then is: how to find a
continuous version of the square function? We introduce a new
continuous square function~$S_c(f)$
(Definition~\ref{def:continuous_square_function}), via certain
wavelet operators~$D_k$ that are expressed in terms of the
Auscher--Hyt\"onen wavelets (Lemma~\ref{lemma another version
of Lemma 9.1}). We prove that $\|S_c(f)\|_p \sim \|f\|_p$ for
$1 < p < \infty$ (Theorem~\ref{theorem P P inequality
one-parameter}), and that $\|S(f)\|_p \sim \|S_c(f)\|_p$ both
for $1 < p < \infty$ and moreover for an additional range of $p
\leq 1$ depending on the upper dimensions~$\omega_i$ of the
factor spaces~$X_i$ and on the H\"older regularity
exponents~$\eta_i$ of the wavelets (Theorem~\ref{theorem
Littlewood Paley}).

We remark that in this paper we concentrate on the product
case. As Nagel and Stein observed in~\cite{NS1}, \emph{``Any
product theory tends to be burdened with notational
complexities.''} For notational simplicity, we have written our
results and proofs for the case of two parameters. However, our
methods also establish the corresponding results for the
product case with $k$ factors, for $k\in\mathbb{N}$. We also
point out that these results extend related previous results
from~\cite{DJS,H1,H2,HLL1,HLL2,HMY,HS} and the references
therein. In those papers either extra assumptions are made on
the quasi-metric and the measure, or the product case is not
covered, or both.

The paper is organized as follows. In
Section~\ref{sec:preliminaries} we briefly recall the systems
of dyadic cubes from~\cite{HK} and the orthonormal bases
from~\cite{AH} on spaces of homogeneous type in the sense of
Coifman and Weiss. In
Section~\ref{sec:testfunctionsdistributions} we introduce the
one-parameter and product test functions in
Definitions~\ref{def-of-test-func-space}
and~\ref{def-of-test-func-space-product}, respectively,
together with the induced classes of distributions. The main
result in this section is Theorem~\ref{thm reproducing formula
test function}, which gives the wavelet reproducing formula for
test functions. In Section~\ref{sec:squarefunctionPP}, the
Littlewood--Paley square functions in terms of the wavelet
coefficients and of the wavelet operators are given in
Definitions~\ref{def:discrete_square_function}
and~\ref{def:continuous_square_function}, respectively. The two
main results here are Theorems~\ref{theorem Littlewood Paley}
and~\ref{theorem P P inequality one-parameter}.
Theorem~\ref{theorem Littlewood Paley} gives the
Littlewood--Paley theory, including the norm equivalence of the
discrete and continuous Littlewood--Paley square functions.
Theorem~\ref{theorem P P inequality one-parameter} gives the
Plancherel--P\'olya inequalities, which are the main tool in
proving Theorem~\ref{theorem Littlewood Paley}. The product
$H^p$, $\cmo^p$, $\bmo$ and $\vmo$ spaces are defined in
Section~\ref{sec:functionspaces} via the orthonormal wavelet
basis. We use the Plancherel--P\'olya inequalities again to
show that these function spaces are well defined. The duality
results are given in Theorem~\ref{thm-duality} for $(H^p)' =
\cmo^p$ and in Theorem~\ref{thm-duality 2} for $(\vmo)' = H^1$.
Finally, in Section~\ref{sec:CZdecomposition} we prove the
Calder\'on--Zygmund decomposition and the interpolation theorem
for Hardy spaces in Theorems~\ref{theorem C-Z decomposition for
Hp} and~\ref{theorem interpolation Hp}, respectively.

\section{Preliminaries}\label{sec:preliminaries}
\setcounter{equation}{0}

We are interested in establishing the Hardy space theory on
spaces ${\widetilde X} = X_1\times X_2\times\cdot
\cdot\cdot\times X_n$. Each factor is a space of homogeneous
type in the sense of Coifman and Weiss. We will first need to
develop a Littlewood--Paley theory for each factor $X_i$,
$1\leq i\leq n$, and then pass to the corresponding product
theory. In this paper, we always assume that $\mu(X_i) =
\infty$ and that $\mu(B(x,r)) > 0$ for all $r > 0$ and $x\in
X_i$, for $1\leq i\leq n$. As usual, $C$ denotes a constant
that is independent of the essential variables, and that may
differ from line to line.

In this section we recall the systems of dyadic cubes, in a
geometrically doubling metric space, constructed by Hyt\"onen
and Kairema~\cite{HK}; and the orthonormal wavelet basis, on
spaces of homogeneous type, constructed by Auscher and
Hyt\"onen~\cite{AH,AH2}. See also \cite{HK}, \cite{AH} and the
references therein for the history and applications of various
versions of dyadic cubes.



\subsection{Systems of dyadic cubes in a geometrically doubling metric
space}\label{sec:dyadiccubes}
\medskip

Let $d$ be a quasi-metric (defined in the Introduction) on a
set~$X$. The quasi-metric space $(X,d)$ is assumed to have the
following {\it geometric doubling property}: there exists a
positive integer $A_1\in\mathbb{N}$ such that for each $x\in X$
and each $r > 0$, the ball $B(x,r) := \{y\in X : d(y,x) < r\}$
can be covered by at most $A_1$ balls $B(x_i,r/2)$. It is shown
in~\cite{CW1} that spaces of homogeneous type have the
geometric doubling property.

As usual, a set $\Omega\subset X$ is {\it open} if for every
$x\in\Omega$ there exists $\varepsilon > 0$ such that
$B(x,\varepsilon)\subset \Omega$, and a set is {\it closed} if
its complement is open.

\begin{theorem}[\cite{HK} Theorem 2.2]\label{theorem dyadic cubes}
    Suppose that constants $0 < c_0 \leq C_0 < \infty$ and
    $\delta\in(0,1)$ satisfy
    \begin{eqnarray}\label{test condition for cubes}
        12A_0^3C_0\delta
        \leq c_0.
    \end{eqnarray}
    Given a set of points $\{z_\alpha^k\}_{\alpha}$, $\alpha
    \in \mathscr{A}_k$, for every $k\in\mathbb{Z}$, with the
    properties that
    \begin{eqnarray}\label{sparse property}
        d(z_\alpha^k,z_\beta^k)
        \geq c_0\delta^k\
            (\alpha\not=\beta),\hskip1cm
        \min_\alpha d(x,z_\alpha^k)
        < C_0\delta^k,
            \qquad \text{for all $x\in X$},
    \end{eqnarray}
    we can construct families of sets $\widetilde{Q}_\alpha^k
    \subseteq Q_\alpha^k \subseteq
    \overline{Q}_\alpha^k$ (called open, half-open and closed
    \emph{dyadic cubes}), such that:
    \begin{eqnarray}
        && \widetilde{Q}_\alpha^k \mbox{ and } \overline{Q}_\alpha^k
            \mbox{ are the interior and closure of } Q_\alpha^k, \mbox{ respectively};\\
        &&\mbox{if } \ell\geq k, \mbox{ then either } Q_\beta^\ell\subseteq
            Q_\alpha^k \mbox{ or } Q_\alpha^k
            \cap Q_\beta^\ell=\emptyset;\label {DyadicP1}\\
        &&  X = \bigcup_\alpha Q_\alpha^k\ \ (\mbox{disjoint union})\qquad
            \text{for all $k\in\mathbb{Z}$};\label {DyadicP2} \\
        && B(z_\alpha^k,c_1\delta^k)\subseteq Q_\alpha^k\subseteq
            B(z_\alpha^k,C_1\delta^k),\ \mbox{where } c_1 := (3A_0^2)^{-1}c_0\
            \mbox{and}\  C_1 := 2A_0C_0;\label {prop_cube3}\\
        &&\mbox{if } \ell \geq k \mbox{ and } Q_\beta^\ell\subseteq Q_\alpha^k,
            \mbox{ then } B(z_\beta^\ell,C_1\delta^\ell)\subseteq
            B(z_\alpha^k,C_1\delta^k).\label {DyadicP4}
    \end{eqnarray}
    The open and closed cubes $\widetilde{Q}_\alpha^k$ and
    $\overline{Q}_\alpha^k$ depend only on the points
    $z_\beta^\ell$ for $\ell\geq k$. The half-open cubes
    $Q_\alpha^k$ depend on $z_\beta^\ell$ for $\ell\geq
    \min(k,k_0)$, where $k_0\in\mathbb{Z}$ is a preassigned
    number entering the construction.
\end{theorem}


\subsection{Orthonormal wavelet basis and wavelet
expansion}\label{sec:onbwavelets}

\medskip

In this subsection, we recall the orthonormal basis and wavelet
expansion in $L^2(X)$ which were recently constructed by
Auscher and Hyt\"onen~\cite{AH}. To state their result, we must
first recall the set $\{x_\alpha^k\}$ of {\it reference dyadic
points} as follows. Let $\delta$ be a fixed small positive
parameter (for example, as noted in Section~2.2 of \cite{AH},
it suffices to take $\delta\leq 10^{-3} A_0^{-10}$). For $k =
0$, let $\mathscr{X}^0 := \{x_\alpha^0\}_\alpha$ be a maximal
collection of 1-separated points in~$\XX$. Inductively, for
$k\in\mathbb{Z}_+$, let $\mathscr{X}^k := \{x_\alpha^k\}
\supseteq \mathscr{X}^{k-1}$ and $\mathscr{X}^{-k} :=
\{x_\alpha^{-k}\} \subseteq \mathscr{X}^{-(k-1)}$ be
$\delta^k$- and $\delta^{-k}$-separated collections in
$\mathscr{X}^{k-1}$ and $\mathscr{X}^{-(k-1)}$, respectively.

Lemma~2.1 in \cite{AH} shows that, for all $k\in\mathbb{Z}$ and
$x\in X$, the reference dyadic points satisfy
\begin{eqnarray}\label{delta sparse}
    d(x_\alpha^k,x_\beta^k)\geq\delta^k\ (\alpha\not=\beta),\hskip1cm
        d(x,\mathscr{X}^k)
    = \min_\alpha\,d(x,x_\alpha^k)
    < 2A_0\delta^k.
\end{eqnarray}
Also, taking $c_0 := 1$, $C_0 := 2A_0$ and $\delta \leq 10^{-3}
A_0^{-10}$, we see that $c_0$, $C_0$ and $\delta$ satisfy
\eqref{test condition for cubes} in Theorem~\ref{theorem dyadic
cubes}. Therefore we may apply Hyt\"onen and Kairema's
construction (Theorem~\ref{theorem dyadic cubes}), with the
reference dyadic points $\{x^k_\alpha\}_{k\in\Z,
\alpha\in\mathscr{X}^k}$ playing the role of the points
$\{z^k_\alpha\}_{k\in\Z, \alpha\in\mathscr{A}_k}$, to conclude
that there exists a set of half-open dyadic cubes
\[
    \{Q_\alpha^k\}_{k\in\mathbb{Z},\alpha\in\mathscr{X}^k}
\]
associated with the reference dyadic points
$\{x_\alpha^k\}_{k\in\mathbb{Z},\alpha\in\mathscr{X}^k}$. We
call the reference dyadic point~$x_\alpha^k$ the \emph{center}
of the dyadic cube~$Q_\alpha^k$. We also identify with
$\mathscr{X}^k$ the set of indices~$\alpha$ corresponding to
$x_\alpha^k \in \mathscr{X}^{k}$.

Note that $\mathscr{X}^{k}\subseteq \mathscr{X}^{k+1}$ for
$k\in\mathbb{Z}$, so that every $x_\alpha^k$ is also a point of
the form $x_\beta^{k+1}$. We denote
$\mathscr{Y}^{k}:=\mathscr{X}^{k+1}\backslash \mathscr{X}^{k}$,
and relabel the points $\{x_\alpha^k\}_\alpha$ that belong
to~$\mathscr{Y}^k$ as $\{y_\alpha^k\}_\alpha$.

We now recall the orthonormal wavelet basis of $L^2(X)$
constructed by Auscher and Hyt\"onen.

\begin{theorem}[\cite{AH} Theorem 7.1]\label{theorem AH orth basis}
    Let $(\XX,d,\mu)$ be a space of homogeneous type with
    quasi-triangle constant $A_0$, and let
    \begin{equation}\label{eqn:defn_of_a}
        a
        := (1+2\log_2A_0)^{-1}.
    \end{equation}
    There exists an orthonormal wavelet basis
    $\{\psi_\alpha^k\}$, $k\in\mathbb{Z}$, $y_\alpha^k\in
    \mathscr{Y}^k$, of $L^2(X)$, having exponential decay
    \begin{eqnarray}\label{exponential decay}
        |\psi_\alpha^k(x)|
        \leq {C\over \sqrt{\mu(B(y_\alpha^k,\delta^k))}}
            \exp\Big(-\nu\Big( {d(y^k_\alpha,x)\over\delta^k}\Big)^a\Big),
    \end{eqnarray}
    H\"older regularity
    \begin{eqnarray}\label{Holder-regularity}
        |\psi_\alpha^k(x)-\psi_\alpha^k(y)|
        \leq {C\over \sqrt{\mu(B(y_\alpha^k,\delta^k))}}
            \Big( {d(x,y)\over\delta^k}\Big)^\eta
            \exp\Big(-\nu\Big( {d(y^k_\alpha,x)\over\delta^k}\Big)^a\Big)
    \end{eqnarray}
    for $d(x,y)\leq \delta^k$, and the cancellation property
    \begin{eqnarray}\label{cancellation}
        \int_X \psi_\alpha^k(x)\,d\mu(x) = 0,
        \qquad \text{for $k\in\mathbb{Z}$, $\ y_\alpha^k\in\mathscr{Y}^k$}.
    \end{eqnarray}
\end{theorem}

Moreover, the wavelet expansion is given by
\begin{eqnarray}\label{eqn:AH_reproducing formula}
    f(x)
    = \sum_{k\in\mathbb{Z}}\sum_{\alpha \in \mathscr{Y}^k}
        \langle f,\psi_{\alpha}^k \rangle \psi_{\alpha}^k(x)
\end{eqnarray}
in the sense of $L^2(X)$.

Here $\delta$ is a fixed small parameter, say $\delta \leq
10^{-3} A_0^{-10}$, and $C < \infty$, $\nu > 0$ and
$\eta\in(0,1]$ are constants independent of $k$, $\alpha$, $x$
and~$y_\alpha^k$.


In what follows, we refer to the functions $\psi_\alpha^k$ as
wavelets. Throughout this paper, $a$ denotes the exponent
from~\eqref{eqn:defn_of_a} and $\eta$ denotes the
H\"older-regularity exponent from~\eqref{Holder-regularity}.

\begin{remark}\label{remark:crucial_estimate}
    The wavelets $\{\psi_\alpha^k\}_{k,\alpha}$ form an
    unconditional basis of $L^p(X)$ for $1 < p < \infty$, as
    shown in Corollary~10.4 in~\cite{AH}. Therefore, the
    reproducing formula~(\ref{eqn:AH_reproducing formula}) also
    holds for $f\in L^p(X)$. Moreover, for us the most crucial
    feature of the orthonormal wavelets construction is the
    following estimate, which is a special case
    of~\cite[Lemma~8.3]{AH}:
    \begin{equation}\label{eqn:AHLemma8.3}
        \sum_{j\in \mathbb Z: \delta^j \geq r}
            {{1}\over{\mu(B(x,\delta^j))}}
            \exp\Big(-\nu\Big( \frac{d(x,\mathscr{Y}^j)}{\delta^j} \Big)^a\Big)
        \leq {{C}\over{\mu(B(x,r))}},
    \end{equation}
    for all $x\in X$ and $r > 0$, and for the constants $\nu >
    0$ and $a := (1 + 2\log_2 A_0)^{-1}$ from
    Theorem~\ref{theorem AH orth basis}.  Series of this type
    naturally arise in the context of proving that the
    reproducing formula holds for test functions and
    distributions, as well as in relation to function spaces.
    This estimate allows us to drop the extra assumption, used
    in previous work, of a reverse-doubling property on the
    measure.

    Furthermore, this estimate is crucial for estimating the
    quantity $\sum_k\sum_\alpha \psi_\alpha^k(x)
    \psi_\alpha^k(y)$. We will use this estimate in the proofs
    of Theorems~\ref{thm reproducing formula test function}
    and~\ref{theorem Littlewood Paley} below.
\end{remark}

\section{Test functions, distributions, and wavelet reproducing
formula}\label{sec:testfunctionsdistributions}

\medskip
We now introduce test functions and distributions on spaces of
homogeneous type $(X,d,\mu)$ in the sense of Coifman and Weiss,
and on product spaces $(X_1,d_1,\mu_1) \times (X_2,d_2,\mu_2)$.
We show that the (scaled) Auscher--Hyt\"onen wavelets are test
functions (Theorem~\ref{theorem wavelet is test function}), and
establish the wavelet reproducing formula for test functions
and for distributions (Theorem~\ref{thm reproducing formula
test function}, Corollary~\ref{coro reproducing formula
distribution}, Theorem~\ref{thm product wavelet reproducing
formula test function}). Along the way we establish some
properties of the wavelet operators~$D_k$ (Lemma~\ref{lemma
another version of Lemma 9.1}), and construct smooth cut-off
functions using the splines of Auscher and Hyt\"onen
(Lemma~\ref{lemma technique 3 cut-off function}).

We begin with the one-parameter case. For $x$, $y\in X$ and $r
> 0$, let
\[
    V_r(x) := \mu(B(x,r))
    \qquad
    \text{and}
    \qquad
    V(x,y) := \mu(B(x,d(x,y))).
\]

\subsection{One-parameter test functions,
distributions, and wavelet reproducing
formula}\label{sec:oneparametertestfunctions}

\begin{definition}\label{def-of-test-func-space}
    \textup{(Test functions)} Fix $x_0\in X$, $r > 0$,
    $\beta\in(0,\eta]$ where $\eta$ is the regularity exponent
    from Theorem~\ref{theorem AH orth basis}, and $\gamma
    > 0$. A function $f$ defined on~$X$ is said to be a {\it test
    function of type $(x_0,r,\beta,\gamma)$ centered at $x_0\in
    X$} if $f$ satisfies the following three conditions.
    \begin{enumerate}
        \item[(i)] \textup{(Size condition)} For all $x\in
            X$,
            \[
                |f(x)|
                \leq C \,\frac{1}{V_{r}(x_0) + V(x,x_0)}
                \Big(\frac{r}{r + d(x,x_0)}\Big)^{\gamma}.
            \]

        \item[(ii)] \textup{(H\"older regularity
            condition)} For all $x$, $y\in X$ with $d(x,y)
            < (2A_0)^{-1}(r + d(x,x_0))$,
            \[
                |f(x) - f(y)|
                \leq C \Big(\frac{d(x,y)}{r + d(x,x_0)}\Big)^{\beta}
                \frac{1}{V_{r}(x_0) + V(x,x_0)} \,
                \Big(\frac{r}{r + d(x,x_0)}\Big)^{\gamma}.
            \]

        \item[(iii)] \textup{(Cancellation condition)}
            \[
                \int_X f(x) \, d\mu(x)
                = 0.
            \]
    \end{enumerate}
\end{definition}

\emph{A priori}, this definition makes sense for arbitrary
$\beta > 0$. Here we have used the condition $\beta \in
(0,\eta]$ both for consistency with the earlier literature and
since our focus is on the wavelets~$\psi^k_\alpha$, which (when
scaled) are test functions with $\beta = \eta$, as we will see.

We denote by $G(x_0,r,\beta,\gamma)$ the set of all test
functions of type $(x_0,r,\beta,\gamma)$. The norm of $f$ in $
G(x_0,r,\beta,\gamma)$ is defined by
\[
    \|f\|_{G(x_0,r,\beta,\gamma)}
    := \inf\{C>0:\ {\rm(i)\  and \ (ii)}\ {\rm hold} \}.
\]

For each fixed $x_0$, let $G(\beta,\gamma) :=
G(x_0,1,\beta,\gamma)$. It is easy to check that for each fixed
$x_0'\in X$ and $r > 0$, we have $G(x_0',r,\beta,\gamma) =
G(\beta,\gamma)$ with equivalent norms. Furthermore, it is also
easy to see that $G(\beta,\gamma)$ is a Banach space with
respect to the norm on $G(\beta,\gamma)$.


For $\beta \in (0,\eta]$ and $\gamma > 0$, let
$\GGs(\beta,\gamma)$ be the completion of the space
$G(\eta,\gamma)$ in the norm of $G(\beta,\gamma)$; of course
when $\beta = \eta$ we simply have $\GGs(\beta,\gamma) =
\GGs(\eta,\gamma) = G(\eta,\gamma)$. We define the norm on
$\GGs(\beta,\gamma)$ by $\|f\|_{\GGs(\beta,\gamma)} :=
\|f\|_{G(\beta,\gamma)}$.

It is immediate from the definition that the sets
$\GGs(\beta,\gamma)$ are nested; for example if $0 < \beta' <
\beta$ and $0 < \gamma' < \gamma$, then $\GGs(\beta,\gamma)
\subset \GGs(\beta',\gamma')$.

\begin{definition}
    \textup{(Distributions)} Fix $x_0\in X$, $r > 0$,
    $\beta\in(0,\eta]$ where $\eta$ is the regularity exponent
    from Theorem~\ref{theorem AH orth basis}, and $\gamma > 0$.
    The \emph{distribution space} $(\GGs(\beta,\gamma))'$ is
    defined to be the set of all linear functionals
    $\mathcal{L}$ from $\GGs(\beta,\gamma)$ to $\mathbb{C}$
    with the property that there exists $C > 0$ such that for
    all $f\in \GGs(\beta,\gamma)$,
    \[
        |\mathcal{L}(f)|
        \leq C\|f\|_{\GGs(\beta,\gamma)}.
    \]
\end{definition}

We note that for each $\beta \in (0,\eta]$ and $\gamma > 0$,
the set $\GGs(\beta,\gamma) \subset L^2(X)$, while each $f \in
L^2(X)$ induces a distribution in $(\GGs(\beta,\gamma))'$.

We now prove that the wavelets constructed in~\cite{AH},
suitably scaled, are test functions.

\begin{theorem}\label{theorem wavelet is test function}
    Suppose that $\{\psi_\alpha^k\}_{k\in\mathbb{Z},
    \alpha\in\mathscr{Y}^k}$ is an orthonormal wavelet basis as
    in Theorem~\ref{theorem AH orth basis} with H\"older
    regularity of order~$\eta$. Then for each $k\in\Z$,
    $y^k_\alpha \in \mathscr{Y}^k$, and $\gamma > 0$, the
    scaled wavelet $\psi_\alpha^k(x)/
    \sqrt{\mu(B(y_\alpha^k,\delta^k))}$ belongs to the set
    $G(y_\alpha^k, \delta^k,\eta,\gamma)$ of test functions of
    type $(y_\alpha^k, \delta^k,\eta,\gamma)$ centered at
    $y_\alpha^k$.
\end{theorem}

Before proving Theorem~\ref{theorem wavelet is test function},
we make the following useful observation: by the doubling
property~\eqref{upper dimension} on the measure~$\mu$, for each
$x_0$, $x\in X$ and $r
> 0$ with $r\leq d(x_0,x)$, we have $V(x_0,x) \leq
C (d(x_0,x)/r)^{\omega} V_r(x_0)$ and hence
\begin{eqnarray}\label{an observation}
    {1\over V_{r}(x_0)}
    \leq C \, \frac{1}{V_r(x_0)+V(x_0,x)}
        \Big(\frac{d(x_0,x)}{r}\Big)^\omega.
\end{eqnarray}

\begin{proof}[Proof of Theorem~\ref{theorem wavelet is test function}]
By property~\eqref{exponential decay} of $\psi_\alpha^k$ from
Theorem~\ref{theorem AH orth basis}, we obtain that
\begin{eqnarray*}
    {\psi_\alpha^k(x)\over \sqrt{\mu(B(y_\alpha^k,\delta^k))} }
    &\leq& {C\over \mu(B(y_\alpha^k,\delta^k))}
        \exp(-\nu(\delta^{-k} d(y_\alpha^k,x))^a)\\
    &\leq&{C\over V_{\delta^k}(y_\alpha^k)}
        \Big({\delta^k\over \delta^{k}+ d(y_\alpha^k,x)}\Big)^\Gamma
\end{eqnarray*}
for all $\Gamma > 0$, with a constant $C$ depending only on
$\nu$, $a = (1 + 2\log_2A_0)^{-1}$, and~$\Gamma$. To see that
$\psi_\alpha^k(x)/ \sqrt{\mu(B(y_\alpha^k,\delta^k))}$
satisfies the size condition
Definition~\ref{def-of-test-func-space}(i) for $\gamma > 0$, we
consider two cases. First, if $\delta^{k} > d(y_\alpha^k,x)$,
then $V(y_\alpha^k,x)\leq V_{\delta^k}(y_\alpha^k)$ and hence
\begin{eqnarray*}
    {\psi_\alpha^k(x)\over \sqrt{\mu(B(y_\alpha^k,\delta^k))} }
    &\leq&{C\over V_{\delta^k}(y_\alpha^k)+V(y_\alpha^k,x)}
        \Big({\delta^k\over \delta^{k}+ d(y_\alpha^k,x)}\Big)^\Gamma.
\end{eqnarray*}
For the second case, if $\delta^{k}\leq d(y_\alpha^k,x)$, an
application of (\ref{an observation}) shows that
\begin{eqnarray}
    {\psi_\alpha^k(x)\over \sqrt{\mu(B(y_\alpha^k,\delta^k))} }
    &\leq& C  {1\over V_{\delta^k}(y_\alpha^k)+V(y_\alpha^k,x)}
    \Big({\delta^k\over \delta^{k}+
    d(y_\alpha^k,x)}\Big)^{\Gamma - \omega}.\nonumber
\end{eqnarray}
Taking $\Gamma > \omega$ and setting $\gamma := \Gamma -
\omega$, we see that $\psi_\alpha^k(x)/
\sqrt{\mu(B(y_\alpha^k,\delta^k))}$ satisfies the size
condition Definition~\ref{def-of-test-func-space}(i) with $x_0
= y_\alpha^k$ and $r = \delta^k$, and for arbitrary $\gamma >
0$.

We now show that $\psi_\alpha^k(x)/
\sqrt{\mu(B(y_\alpha^k,\delta^k))}$ satisfies the smoothness
condition Definition~\ref{def-of-test-func-space}(ii) for
$\gamma > 0$ and  $\beta = \eta$, if $d(x,y) \leq
(2A_0)^{-1}(\delta^k + d(y_\alpha^k,x))$. We consider three
cases. First, suppose $d(x,y)\leq\delta^k \leq
(2A_0)^{-1}(\delta^k + d(y_\alpha^k,x))$. Then
property~\eqref{Holder-regularity} yields
\begin{eqnarray*}
\Big|{\psi_\alpha^k(x)-\psi_\alpha^k(y)\over
\sqrt{\mu(B(y_\alpha^k,\delta^k))} }\Big|
    &\leq& C\Big( {d(x,y)\over\delta^k}
    \Big)^\eta{1\over \mu(B(y_\alpha^k,\delta^k))}
    \Big({\delta^k\over \delta^{k}+ d(y_\alpha^k,x)}\Big)^\Gamma.
\end{eqnarray*}
Note that in this case, $\delta^k \leq (2A_0-1)^{-1}
d(y_\alpha^k,x)$, and so we may apply~\eqref{an observation}
with $r = \delta^k$ and $x_0 = y^k_\alpha$ to conclude that
\begin{eqnarray*}
    \Big|{\psi_\alpha^k(x) - \psi_\alpha^k(y)\over
        \sqrt{\mu(B(y_\alpha^k,\delta^k))} }\Big|
    \leq C \Big( {d(x,y)\over\delta^k +d(y_\alpha^k,x)}\Big)^\eta
        {1\over V_{\delta^k}(y_\alpha^k)+V(y_\alpha^k,x)}
        \Big({\delta^k\over \delta^{k} +
        d(y_\alpha^k,x)}\Big)^{\Gamma - \omega - \eta}.
\end{eqnarray*}

Second, consider the case where $\delta^k \leq d(x,y) \leq
(2A_0)^{-1}(\delta^k + d(y_\alpha^k,x))$. It is straightforward
to verify from the quasi-triangle inequality that in this case,
\begin{equation}\label{eqn:smoothness_case2_equiv}
    \delta^k + d(y_\alpha^k,y)
    \sim \delta^k + d(y_\alpha^k,x).
\end{equation}
Applying property~(\ref{exponential decay}) for
$\psi_\alpha^k(x)$ and~$\psi_\alpha^k(y)$, we find
\begin{eqnarray*}
    \Big|{\psi_\alpha^k(x)-\psi_\alpha^k(y)\over
        \sqrt{\mu(B(y_\alpha^k,\delta^k))} }\Big|
    &\leq& {C\over V_{\delta^k}(y_\alpha^k)}
        \Big({\delta^k\over \delta^{k}+ d(y_\alpha^k,x)}\Big)^\Gamma
        + {C\over V_{\delta^k}(y_\alpha^k)}
        \Big({\delta^k\over \delta^{k}+ d(y_\alpha^k,y)}\Big)^\Gamma\\
    &\leq& C \Big( {d(x,y)\over\delta^k}
        \Big)^\eta{1\over V_{\delta^k}(y_\alpha^k)}
        \Big({\delta^k\over \delta^{k}+ d(y_\alpha^k,x)}\Big)^\Gamma \\
    &\leq& C \Big( {d(x,y)\over\delta^k +d(y_\alpha^k,x)}
        \Big)^\eta {1\over V_{\delta^k}(y_\alpha^k)+V(y_\alpha^k,x)}
        \Big({\delta^k\over \delta^{k}+
        d(y_\alpha^k,x)}\Big)^{\Gamma - \omega - \eta}.
\end{eqnarray*}
Here the second inequality follows
from~\eqref{eqn:smoothness_case2_equiv} and the fact that
$d(x,y)/\delta^k \geq 1$,
and the third inequality follows from~\eqref{an observation}.

For the third and last case, if $\delta^k >
(2A_0)^{-1}(\delta^k + d(y_\alpha^k,x)) \geq d(x,y)$, then we
have $d(x,y) < \delta^k$ and $d(y_\alpha^k,x)\leq
(2A_0-1)\delta^k$. Therefore, applying property
(\ref{Holder-regularity}) together with the fact that
$V(y_\alpha^k,x)\leq C\mu(B(y_\alpha^k,\delta^k))$ yields
\begin{eqnarray*}
\Big|{\psi_\alpha^k(x)-\psi_\alpha^k(y)\over
\sqrt{\mu(B(y_\alpha^k,\delta^k))} }\Big|
    &\leq& C\Big( {d(x,y)\over\delta^k}\Big)^\eta
        {1\over V_{\delta^k}(y_\alpha^k)}
        \Big({\delta^k\over \delta^{k}+ d(y_\alpha^k,x)}\Big)^\Gamma\\
    &\leq& C \Big( {d(x,y)\over\delta^k +d(y_\alpha^k,x)}\Big)^\eta
        {1\over V_{\delta^k}(y_\alpha^k)+V(y_\alpha^k,x)}
        \Big({\delta^k\over \delta^{k} +
        d(y_\alpha^k,x)}\Big)^{\Gamma - \eta}.
\end{eqnarray*}

Combining all the cases above, in the first and second cases
take $\Gamma > \omega + \eta$ and set $\gamma := \Gamma -
\omega - \eta$, and in the third case take $\Gamma > \eta$ and
set $\gamma := \Gamma - \eta$. We see that the function
$\psi_\alpha^k(x)/ \sqrt{\mu(B(y_\alpha^k,\delta^k))}$
satisfies the smoothness condition (ii) in Definition
\ref{def-of-test-func-space} with $x_0 = y_\alpha^k$, $r =
\delta^k$, $\beta = \eta$ and for arbitrary $\gamma > 0$.

Moreover, the cancellation property of the function
$\psi_\alpha^k(x)/ \sqrt{\mu(B(y_\alpha^k,\delta^k))}$ is
immediate from property (\ref{cancellation}) of
$\psi_\alpha^k$, by Theorem~\ref{theorem AH orth basis}.
Thus, $\psi_\alpha^k(x)/ \sqrt{\mu(B(y_\alpha^k,\delta^k))}$
belongs to the test function
space~$G(y_\alpha^k,\delta^k,\eta,\gamma)$. This completes the
proof of Theorem~\ref{theorem wavelet is test function}.
\end{proof}

Now we state and prove the main result of this subsection,
which will be the crucial tool for establishing the
Littlewood--Paley theory and developing the Hardy spaces.

\begin{theorem}\label{thm reproducing formula test function}
    \textup{(Wavelet reproducing formula for test functions)} Suppose
    that $f\in \GGs(\beta,\gamma)$ with $\beta$,
    $\gamma\in(0,\eta)$. Then the wavelet reproducing formula
    \begin{equation}\label{eqn:reproducing_formula}
        f(x)
        = \sum_{k\in\mathbb{Z}}\sum_{\alpha \in \mathscr{Y}^k}
            \langle f,\psi_{\alpha}^k \rangle \psi_{\alpha}^k(x)
    \end{equation}
    holds in $\GGs(\beta',\gamma')$ for all $\beta'\in(0,\beta)$ and
    $\gamma'\in(0,\gamma)$.
\end{theorem}

As an immediate consequence of Theorem~\ref{thm reproducing
formula test function}, we obtain the following result.

\begin{corollary}\label{coro reproducing formula distribution}
    \textup{(Wavelet reproducing formula for distributions)}
    Take $\beta$, $\gamma\in(0,\eta)$. Then the wavelet
    reproducing formula~\eqref{eqn:reproducing_formula} also
    holds in the space $(\GGs(\beta,\gamma))'$ of
    distributions.
\end{corollary}

\begin{proof}[Proof of Theorem~\ref{thm reproducing formula test function}]
Take $f\in \GGs(\beta,\gamma)$ with $\beta$,
$\gamma\in(0,\eta)$. It suffices to show that
\begin{eqnarray}\label{convergence in test function spaces}
    \Big\|\sum_{|k| > L}\sum_{\alpha \in \mathscr{Y}^k}
        \langle \psi_{\alpha}^k,f\rangle
        \psi_{\alpha}^k(\cdot)\Big\|_{G(\beta',\gamma')}
    \longrightarrow 0
\end{eqnarray}
as $L$ tends to infinity, for each $\beta'\in(0,\beta)$ and
$\gamma'\in(0,\gamma)$.

The proof of (\ref{convergence in test function spaces}) is
based on the following estimate: for each $\beta'\in(0,\beta)$
and $\gamma'\in(0,\gamma)$, there is a constant $\sigma > 0$
such that for each $L\in\N$,
\begin{eqnarray}\label{decay estimate of the remaining terms}
    \Big\|\sum_{|k|>L}\sum_{ \alpha \in \mathscr{Y}^k}
        \langle \psi_{\alpha}^k,f\rangle
        \psi_{\alpha}^k(\cdot)\Big\|_{G(\beta',\gamma')}
    \leq C \delta^{\sigma L}\|f\|_{G(\beta,\gamma)},
\end{eqnarray}
where $C$ is a positive constant independent of
$f\in\GGs(\beta,\gamma)$.

To verify (\ref{decay estimate of the remaining terms}), it
suffices to show that the following decay and smoothness
estimates hold: for each $\gamma'\in(0,\gamma)$, there exist a
positive constant~$C$ independent of~$f$, and a positive number
$\sigma'$, such that
\begin{eqnarray}
    &&\Big|\sum_{|k|>L}\sum_{ \alpha \in \mathscr{Y}^k}
        \langle \psi_{\alpha}^k,f\rangle \psi_{\alpha}^k(x)\Big|
    \leq C \delta^{\sigma' L}{1\over V_1(x_0)+V(x,x_0)}
        \Big(\frac{\displaystyle 1}{\displaystyle 1 + d(x,x_0)}\Big)^{\gamma'}
        \|f\|_{G(\beta,\gamma)},\ \ {\rm and}
        \label{size decay estimate of the remaining terms}\\
    &&\Big|\sum_{|k|>L}\sum_{ \alpha \in \mathscr{Y}^k}
        \langle \psi_{\alpha}^k,f\rangle
        \psi_{\alpha}^k(x) - \sum_{|k|>L}\sum_{ \alpha \in \mathscr{Y}^k}
        \langle \psi_{\alpha}^k,f\rangle \psi_{\alpha}^k(x')\Big|
        \label{smooth estimate of the remaining terms}\\
    &&\hskip2cm\leq C
        \Big(\frac{\displaystyle d(x,x')}{\displaystyle 1 + d(x,x_0)}\Big)^{\beta}
        {1\over V_1(x_0)+V(x,x_0)}
        \Big(\frac{\displaystyle 1}{\displaystyle 1 + d(x,x_0)}\Big)^{\gamma}
        \|f\|_{G(\beta,\gamma)}\nonumber
\end{eqnarray}
for all $x$ and $x'$ such that $d(x,x') \leq (2A_0)^{-1}(1 +
d(x,x_0))$.

Indeed, to see that (\ref{size decay estimate of the remaining
terms}) and (\ref{smooth estimate of the remaining terms})
imply (\ref{decay estimate of the remaining terms}), we take
the geometric mean between (\ref{smooth estimate of the
remaining terms}) and the following estimate (obtained directly
from (\ref{size decay estimate of the remaining terms})):
\begin{eqnarray*}
    \lefteqn{\Big|\sum_{|k|>L}\sum_{ \alpha \in \mathscr{Y}^k}
        \langle \psi_{\alpha}^k,f\rangle \psi_{\alpha}^k(x)
        - \sum_{|k|>L}\sum_{ \alpha \in \mathscr{Y}^k}
        \langle \psi_{\alpha}^k,f\rangle \psi_{\alpha}^k(x')\Big|} \hspace{1cm}\\
    &&\leq \Big|\sum_{|k|>L}\sum_{ \alpha \in \mathscr{Y}^k}
        \langle \psi_{\alpha}^k,f\rangle \psi_{\alpha}^k(x)\Big|
        + \Big|\sum_{|k|>L}\sum_{ \alpha \in \mathscr{Y}^k}
        \langle \psi_{\alpha}^k,f\rangle \psi_{\alpha}^k(x')\Big|\\
    &&\leq C \delta^{\sigma' L}{1\over V_1(x_0)+V(x,x_0)}
        \Big(\frac{1}{1 + d(x,x_0)}\Big)^{\gamma'}\|f\|_{G(\beta,\gamma)}.
\end{eqnarray*}
This gives
\begin{eqnarray}\label{smooth decay estimate of the remaining terms}
    &&\Big|\sum_{|k|>L}\sum_{ \alpha \in \mathscr{Y}^k}
        \langle \psi_{\alpha}^k,f\rangle \psi_{\alpha}^k(x)
        - \sum_{|k|>L}\sum_{ \alpha \in \mathscr{Y}^k}
        \langle \psi_{\alpha}^k,f\rangle \psi_{\alpha}^k(x')\Big|\\
    &&\hskip2cm\leq C  \delta^{\sigma L}\Big(\frac{d(x,x')}{r + d(x,x_0)}\Big)^{\beta'}
        {1\over V_1(x_0)+V(x,x_0)}\Big(\frac{1}{1 + d(x,x_0)}\Big)^{\gamma'}
        \|f\|_{G(\beta,\gamma)}\nonumber
\end{eqnarray}
for some $\sigma < \sigma'$. Now (\ref{size decay estimate of
the remaining terms}) and (\ref{smooth decay estimate of the
remaining terms}), together with the fact that
\[
    \int_X\sum_{|k|>L}\sum_{ \alpha \in \mathscr{Y}^k}
        \langle\psi_{\alpha}^k,f\rangle
        \psi_{\alpha}^k(x)\, d\mu(x) = 0,
\]
imply that $\sum_{|k|>L}\sum_{ \alpha \in \mathscr{Y}^k}\langle
\psi_{\alpha}^k,f\rangle \psi_{\alpha}^k$ is a test function in
$G(\beta',\gamma')$. Moreover, we see from the upper bounds
in~\eqref{size decay estimate of the remaining terms}
and~\eqref{smooth decay estimate of the remaining terms}
that~\eqref{decay estimate of the remaining terms} holds, as
required.

To prove the decay and smoothness estimates (\ref{size decay
estimate of the remaining terms}) and (\ref{smooth estimate of
the remaining terms}), we need the following lemma which gives
estimates for the kernels $D_k(x,y) = \sum_{\alpha \in
\mathscr{Y}^k}\psi_{\alpha}^k(x)\psi_{\alpha}^k(y)$ of the
wavelet operators~$D_k$. These wavelet operators are defined by
\[
    D_k f(x)
    := \sum_{\alpha \in \mathscr{Y}^k} \langle \psi^k_\alpha, f \rangle
        \psi^k_\alpha(x)
    = \int_X D_k(x,y) f(y) \, dy,
\]
for $k\in\Z$. We note that the first two estimates in
Lemma~\ref{lemma another version of Lemma 9.1} are similar to
estimates given in Lemma~9.1 in~\cite{AH}.

\begin{lemma}\label{lemma another version of Lemma 9.1}
    \textup{(Properties of wavelet operators~$D_k$)} Let
    \[
        D_k(x,y)
        := \sum_{\alpha \in \mathscr{Y}^k}
            \psi_{\alpha}^k(x)\psi_{\alpha}^k(y)
    \]
    for $x$, $y\in X$. Fix $\gamma > 0$. Then the following
    estimates hold.
    \begin{enumerate}
        \item[(i)] \textup{(Decay condition)} For all $x$,
            $y\in X$, we have
            \begin{equation}\label{e1 in lemma another
            version of Lemma 9.1}
            \big|D_k(x,y)\big|
            \leq C{1\over V_{\delta^k}(x) + V(x,y)}
                \Big({\delta^k\over \delta^{k} + d(x,y)}\Big)^{\gamma}.
            \end{equation}

        \item[(ii)] \textup{(Smoothness condition)} If
            $d(y,y')\leq (2A_0)^{-1} \max\{\delta^k +
            d(x,y),\delta^k + d(x,y')\}$, then
            \begin{equation}\label{e2 in lemma another version of Lemma 9.1}
            \big |D_k(x,y) - D_k(x,y')\big|
                \leq C \Big({d(y,y') \over \delta^k + d(x,y)}\Big)^\eta
                {1\over V_{\delta^k}(x) + V(x,y)}
                \Big({\delta^k\over \delta^{k} + d(x,y)}\Big)^{\gamma}.
            \end{equation}
            The same estimate holds with $x$ and $y$
            interchanged.

        \item[(iii)] \textup{(Double smoothness condition)}
            If
            \begin{align*}
                d(x,x')
                &\leq (2A_0)^{-1}
                    \max\{\delta^k + d(x,y), \delta^k + d(x',y)\} \quad \text{and}\\
                d(y,y')
                &\leq (2A_0)^{-1}
                    \max\{\delta^k + d(x,y), \delta^k + d(x,y')\},
            \end{align*}
            then
            \begin{eqnarray}\label{e3 in lemma another version of Lemma 9.1}
            \lefteqn{\big|D_k(x,y) - D_k(x',y) - D_k(x,y') + D_k(x',y')\big |}\hspace{0.3cm}
                \nonumber \\
            &&\leq C \Big({d(x,x')\over\delta^k +d(x,y)}\Big)^\eta
                 \Big({d(y,y')\over\delta^k +d(x,y)}\Big)^\eta
                {1\over V_{\delta^k}(x)+V(x,y)}
                \Big({\delta^k\over \delta^{k} + d(x,y)}\Big)^{\gamma}.
            \end{eqnarray}
    \end{enumerate}
\end{lemma}

We defer the proof of Lemma~\ref{lemma another version of Lemma
9.1} until after the end of the proof of Theorem~\ref{thm
reproducing formula test function}.

Returning to the proof of Theorem~\ref{thm reproducing formula
test function}, we first show (\ref{size decay estimate of the
remaining terms}). Write
\begin{eqnarray*}
    \Big|\sum_{|k|>L} \sum_{ \alpha \in \mathscr{Y}^k}
        \langle \psi_{\alpha}^k,f\rangle \psi_{\alpha}^k(x)\Big|
    &\leq& \Big|\sum_{k>L} \sum_{ \alpha \in \mathscr{Y}^k}
        \langle \psi_{\alpha}^k,f\rangle \psi_{\alpha}^k(x)\Big|
        + \Big|\sum_{k<-L}\sum_{ \alpha \in \mathscr{Y}^k}
        \langle \psi_{\alpha}^k,f\rangle \psi_{\alpha}^k(x)\Big|\\
    &=:& \textup{(A)} + \textup{(B)}.
\end{eqnarray*}
For $\textup{(A)}$, using the cancellation property
(\ref{cancellation}) of the wavelet $\psi_\alpha^k$ and
integrating over the sets $W_1 := \{ y\in X: d(x,y) \leq
(2A_0)^{-1} (1 + d(x,x_0))\}$ and $W_2 := X\setminus W_1$, we
obtain
\begin{eqnarray*}
    \textup{(A)}
    &\leq & \sum_{k>L} \int_{W_1} \big|D_k(x,y)\big|
        \,\big|f(y)-f(x)\big|\,d\mu(y)\\
    && {}+ \sum_{k>L}\int_{W_2} \big|D_k(x,y)\big|
        \,\big(|f(y)|+|f(x)|\big)\,d\mu(y)\\
    &=:&\textup{(A)}_1+\textup{(A)}_2.
\end{eqnarray*}
To deal with $\textup{(A)}_1$, applying the decay
condition~\eqref{e1 in lemma another             version of
Lemma 9.1} from Lemma~\ref{lemma another version of Lemma 9.1}
on $D_k(x,y)$ and the H\"older regularity property (Definition
\ref{def-of-test-func-space}(ii)) of the test function $f\in
G(\beta,\gamma)$ gives
\begin{eqnarray*}
    \textup{(A)}_1
    &\leq& C\|f\|_{G(\beta,\gamma)}\sum_{k>L}
        \int_{W_1}  \frac{1}{V_{\delta^k}(x) + V(x,y)}
        \Big(\frac{\delta^k}{\delta^k + d(x,y)}\Big)^{\gamma}\\
    &&\hskip1cm {}\times\Big(\frac{d(x,y)}{1 + d(x,x_0)}\Big)^{\beta}
        \frac{1}{V_{1}(x_0) + V(x,x_0)}
        \Big(\frac{1}{1 + d(x,x_0)}\Big)^{\gamma}\,d\mu(y)\\
    &\leq& C \delta^{\beta L} {1\over V_1(x_0) + V(x,x_0)}
        \Big(\frac{1}{1 + d(x,x_0)}\Big)^{\gamma}
        \|f\|_{G(\beta,\gamma)}.
\end{eqnarray*}
To estimate $\textup{(A)}_2,$ applying the size conditions on
both $D_k(x,y)$ (Lemma~\ref{lemma another version of Lemma
9.1}) and $f$ (Definition \ref{def-of-test-func-space}(i))
gives
\begin{eqnarray*}
\textup{(A)}_2
    &\leq & C\|f\|_{G(\beta,\gamma)}
 \sum_{k>L}\int_{W_2}\frac{\displaystyle
1}{\displaystyle V_{\delta^k}(x)+V(x,y)} \Big(\frac{\displaystyle
\delta^k}{\displaystyle \delta^k+d(x,y)}\Big)^{\gamma}\\
&&{}\times \bigg[\frac{\displaystyle 1}{\displaystyle
V_{1}(x_0)+V(y,x_0)} \Big(\frac{\displaystyle 1}{\displaystyle
1+d(y,x_0)}\Big)^{\gamma}
+\frac{\displaystyle 1}{\displaystyle
V_{1}(x_0)+V(x,x_0)} \Big(\frac{\displaystyle 1}{\displaystyle
1+d(x,x_0)}\Big)^{\gamma}\bigg]\,d\mu(y).
\end{eqnarray*}
For the first sum, use the fact that if $d(x,y)
> (2A_0)^{-1} (1+d(x,x_0))$ then $V(x,y)\geq CV(x,
(2A_0)^{-1} (1+d(x,x_0)) \geq C[V_1(x_0)+V(x,x_0)]$. For the
second sum, apply the following estimate:
\[
    \int_{W_2}\frac{\displaystyle 1}{\displaystyle V_{\delta^k}(x)+V(x,y)}
        \Big(\frac{\displaystyle\delta^k}{\displaystyle\delta^k+d(x,y)}\Big)^{\gamma}
        \mu(y)
    \leq C\delta^{\gamma k}.
\]
We obtain
\begin{eqnarray*}
    \textup{(A)}_2
    \leq C\|f\|_{G(\beta,\gamma)}\delta^{\gamma L}{1\over V_1(x_0)+V(x,x_0)}
    \Big(\frac{\displaystyle 1}{\displaystyle 1+d(x,x_0)}\Big)^{\gamma}.
\end{eqnarray*}
Hence for some $\sigma > 0$,
\[
    \textup{(A)}
    \leq C \delta^{\sigma L} {1\over V_1(x_0)+V(x,x_0)}
        \Big(\frac{\displaystyle 1}{\displaystyle 1+d(x,x_0)}\Big)^{\gamma}
        \|f\|_{G(\beta,\gamma)}.
\]

We now turn to $\textup{(B)}$. Using the fact that $\int_X
f(x)\, d\mu(x) = 0$ and considering the sets $W_3 := \{y\in X:
d(y,x_0)\leq (2A_0)^{-1} (\delta^k+d(x,x_0))\}$ and $W_4 :=
X \setminus W_3$, we have
\begin{eqnarray*}
    \textup{(B)}
    &\leq& \sum_{k<-L}\int_{W_3}\big|D_k(x,y)-D_k(x,x_0)\big|\,|f(y)|\,d\mu(y)\\
    &&
    {}+ \sum_{k<-L}\int_{W_4}\Big(\big|D_k(x,y)\big|
        + \big|D_k(x,x_0)\big|\Big)\,|f(y)|\,d\mu(y)\\
    &=:& \textup{(B)}_1+\textup{(B)}_2.
\end{eqnarray*}

For $\textup{(B)}_1$, applying the smoothness estimate from
Lemma~\ref{lemma another version of Lemma 9.1}(ii) and the size
estimate of the test function $f$ (Definition
\ref{def-of-test-func-space}(i)) yields
\begin{eqnarray*}
    \textup{(B)}_1
     &\leq& C\|f\|_{G(\beta,\gamma)}
        \sum_{k<-L}\int_{W_3}
            \Big( {d(y,x_0)\over\delta^k + d(x,x_0)}\Big)^{\eta'}
            {1\over V_{\delta^k}(x) + V(x,x_0)}\\
    &&{}\times \Big({\delta^k\over \delta^{k} + d(x,x_0)}\Big)^{\gamma'}
        \frac{1}{V_{1}(x_0) + V(y,x_0)}
        \Big(\frac{1}{1 + d(y,x_0)}\Big)^{\gamma}\, d\mu(y)\\
    &\leq& C\|f\|_{G(\beta,\gamma)}
        \delta^{(\eta' - \gamma')L}
        {1\over V_1(x_0) + V(x,x_0)}
        \Big(\frac{1}{1 + d(x,x_0)}\Big)^{\gamma'},
\end{eqnarray*}
where we choose $\gamma' < \eta' < \gamma$. To estimate
$\textup{(B)}_2$, we first write $\textup{(B)}_2 =
\textup{(B)}_{21} + \textup{(B)}_{22}$ where
\begin{eqnarray*}
    \textup{(B)}_{21}
    := \sum_{k<-L} \int_{W_4}\big|D_k(x,y)\big|\,|f(y)|\,d\mu(y)
\end{eqnarray*}
and
\begin{eqnarray*}
    \textup{(B)}_{22}
    := \sum_{k<-L}\int_{W_4}\big|D_k(x,x_0)\big|\,|f(y)|\,d\mu(y).
\end{eqnarray*}
Since here $d(y,x_0) > (2A_0)^{-1} (\delta^k+d(x,x_0))$, the
size estimates for the test function~$f$
(Definition~\ref{def-of-test-func-space}(i)) imply that for $0
< \gamma' < \gamma$,
\begin{eqnarray*}
    |f(y)|
    &\leq& C \|f\|_{G(\beta,\gamma)}{1\over V_1(x_0)+V(y,x_0)}
        \Big(\frac{\displaystyle 1}{\displaystyle
        1+d(y,x_0)}\Big)^{\gamma}\\
    &\leq& C\|f\|_{G(\beta,\gamma)}\delta^{k(\gamma'-\gamma)}{1\over
        V_1(x_0)+V(x,x_0)}\Big(\frac{\displaystyle 1}{\displaystyle
        1+d(x,x_0)}\Big)^{\gamma'}.
\end{eqnarray*}
The above estimate, together with the fact that
$\int_X\big|D_k(x,y)\big|\, d\mu(y)\leq C$, yields
\begin{eqnarray*}
    \textup{(B)}_{21}
    \leq C\delta^{(\gamma - \gamma')L}{1\over V_1(x_0)+V(x,x_0)}
        \Big(\frac{1}{1 + d(x,x_0)}\Big)^{\gamma'}
        \|f\|_{G(\beta,\gamma)}.
\end{eqnarray*}
The estimate for $\textup{(B)}_{22}$ is similar, but easier.
Indeed, by Lemma~\ref{lemma another version of Lemma 9.1}, we
have
\[
    \big|D_k(x,x_0)\big|
    \leq C {1\over V_\delta^k(x_0) + V(x,x_0)}
        \Big(\frac{\delta^k}{\delta^k + d(x,x_0)}\Big)^{\gamma'}
\]
and
\[
    \int_{W_4}|f(y)|\, d\mu(y)
    \leq C \Big(\frac{1}{\delta^k + d(x,x_0)}\Big)^{\gamma}.
\]
Thus, we obtain the same estimate for $\textup{(B)}_{22}$ as
for $\textup{(B)}_{21}$, but with $\gamma - \gamma'$ replaced
by $\eta' - \gamma'$. This completes the proof of~(\ref{size
decay estimate of the remaining terms}).

Finally, we show (\ref{smooth estimate of the remaining
terms}). To do this, we first need to construct a smooth
cut-off function. For this purpose, we recall the following
result on the properties of the {\it spline
functions}~$s^k_\alpha$ on~$(X,d,\mu)$ that were constructed by
Auscher and Hyt\"onen~\cite{AH}.

\begin{theorem}[\cite{AH}, Theorem~3.1]\label{theorem AH spline}
    The spline functions $s_\alpha^k$ satisfy the following
    properties: bounded support
    \begin{eqnarray*}
        \chi_{B(x_\alpha^k,{1/8}A_0^{-3}\delta^k)}(x)
        \leq s_\alpha^k(x)
        \leq \chi_{B(x_\alpha^k,8A_0^{5}\delta^k)}(x);
    \end{eqnarray*}
    the interpolation and reproducing properties
    \begin{eqnarray*}
        s_\alpha^k(x_{\beta}^k) = \delta_{\alpha,\beta},
        \hskip1cm
        \sum_\alpha s_\alpha^k(x) = 1,
        \hskip1cm
        s_\alpha^k = \sum_\beta p_{\alpha,\beta}^k s_\beta^{k+1}(x),
    \end{eqnarray*}
    where $\{p_{\alpha,\beta}^k\}_\beta$ is a finite nonzero set of
    nonnegative coefficients with $p_{\alpha,\beta}^k\leq 1$; and
    H\"older continuity
    \begin{eqnarray*}
        |s_\alpha^k(x) - s_\alpha^k(y)|
        \leq C\Big( {d(x,y)\over\delta^k} \Big)^\eta.
    \end{eqnarray*}
\end{theorem}
We point out that in the above theorem, $\alpha$ runs over
$\mathscr{X}^k$.  Using these splines we can construct a smooth
cut-off function, as follows.

\begin{lemma}\label{lemma technique 3 cut-off function}
    \textup{(Smooth cut-off function)} For each fixed $x_0\in
    X$ and $R_0\in(0,\infty)$, there exists a smooth cut-off
    function $h(x)$ such that $0\leq h(x)\leq1$,
    \begin{eqnarray*}
        h(x)\equiv1 \textup{\ \  \ when\ } x\in B(x_0,R_0/4),
        \hskip1cm
        h(x)\equiv0 \textup{\ \  \ when\ } x\in B(x_0,A_0^2R_0)^c,
    \end{eqnarray*}
    and there exists a positive constant $C$ independent of
    $x_0$, $R_0$, $x$, $y$ such that
    \begin{eqnarray*}
        |h(x) - h(y)|
        \leq C\Big( {d(x,y)\over\delta^k} \Big)^\eta.
    \end{eqnarray*}
\end{lemma}

\begin{proof}
For a fixed $R_0\in(0,\infty)$, we choose $k_0\in\mathbb{Z}$
such that
\[
    8A_0^5\delta^{k_0}
    \leq R_0/4
    \hskip1cm
    \textup{and}
    \hskip1cm
    8A_0^5\delta^{k_0 - 1}
    > R_0/4.
\]
Next, we define the index set $\mathcal{I}_{k_0}$ as follows:
\[
    \mathcal{I}_{k_0}
    := \big\{\alpha\in\mathscr{X}^{k_0}:\
        B(x_\alpha^{k_0},8A_0^{5}\delta^{k_0})
        \cap B(x_0,R_0/4) \not= \emptyset \big\}.
\]
Then the number of indices contained in $\mathcal{I}_{k_0}$ is
bounded by a constant independent of $R_0$, $k_0$, and $x_0$,
since $8A_0^{5}\delta^{k_0}$ is comparable to $R_0$ and the
reference dyadic points $\{x_\alpha^{k_0}\}$ are
$\delta^{k_0}$-separated.

Now define
$$ h(x) := \sum_{\alpha\in \mathcal{I}_{k_0}}s_\alpha^{k_0}(x). $$
From the properties of the spline functions $s_\alpha^k(x)$
(Theorem~\ref{theorem AH spline}), it is easy to verify that
$h(x)$ satisfies all the properties listed in Lemma~\ref{lemma
technique 3 cut-off function}.
\end{proof}

In what follows, the H\"older-regularity index of the cut-off
function $h(x)$ is the $\eta$ given in Theorem~\ref{theorem AH
spline} (\cite{AH}, Theorem~3.1).

We now return to the proof of (\ref{smooth estimate of the
remaining terms}). It suffices to show that there exists a
constant $C$ such that for each $M > 0$ and for $d(x,x') \leq
(2A_0)^{-1} \big( 1+d(x,x_0)\big)$,
\begin{eqnarray}\label{e1 for smooth estimate of the remainning terms}
    &&\Big|\int_X \sum_{|k|\leq M}\big[D_k(x,y) - D_k(x',y)\big]f(y)\, d\mu(y)\Big|\\
    &&\hskip1cm\leq C\Big(\frac{\displaystyle d(x,x')}{\displaystyle
        1+d(x,x_0)}\Big)^{\beta}{1\over
        V_1(x_0)+V(x,x_0)}\Big(\frac{\displaystyle 1}{\displaystyle
        1+d(x,x_0)}\Big)^{\gamma}.\nonumber
\end{eqnarray}

To do this, fix $M > 0$ and let $T$ denote the wavelet operator given
by
\[
    T(f)(x)
    := \int_X K(x,y)f(y)\, d\mu(y),
\]
with kernel $K(x,y) := \sum_{|k|\leq M} D_k(x,y)$, where the
kernel $D_k(x,y)$ of the wavelet operator~$D_k$ is given by
$\sum_{\alpha \in \mathscr{Y}^k} \psi_{\alpha}^k(x)
\psi_{\alpha}^k(y)$.

To show (\ref{e1 for smooth estimate of the remainning terms}),
it suffices to prove that if $d(x,x')\leq (2A_0)^{-1} \big(
1+d(x,x_0)\big)$, then
\begin{eqnarray*}
    \big|T(f)(x)-T(f)(x')\big|
    \leq C\Big(\frac{d(x,x')}{1 + d(x,x_0)}\Big)^{\beta}
        {1\over V_1(x_0)+V(x,x_0)}
        \Big(\frac{1}{1 + d(x,x_0)}\Big)^{\gamma}.
\end{eqnarray*}
Let $R = d(x,x_0)$ and $r = d(x,x')$, and consider the case
where $R\geq 10$ and $r \leq (20A^2_0)^{-1}(1 + d(x,x_0))$.
Following \cite{M2}, set $1 = I(y) + J(y) + L(y)$, where $I(y)$
is a smooth cut-off function as in Lemma \ref{lemma technique 3
cut-off function}, satisfying
\[
    I(y)\equiv1 \ \ \textup{when}\ y\in B(x,R/32A_0^2)
    \hskip1cm\textup{and}\hskip1cm
    I(y)\equiv0 \ \ \textup{when}\ y\in B(x,R/8)^c,
\]
and
\[
    J(y)\equiv1 \ \ \textup{when}\ y\in B(x_0,R/32A_0^2)
    \hskip1cm\textup{and}\hskip1cm
    I(y)\equiv0 \ \ \textup{when}\ y\in B(x_0,R/8)^c.
\]
Also set
\[
    f_1(y) := f(y)I(y),
    \qquad
    f_2(y) := f(y)J(y)
    \qquad\text{and}\qquad
    f_3(y) := f(y)L(y).
\]
It is easy to see that $f_1$, $f_2$ and $f_3$ satisfy the
following estimates~\eqref{(size f1)}--\eqref{(L1 norm f2)}:
\begin{eqnarray}
    | f_1(y)|
    &\leq& C\|  f \| _{G(\beta ,\gamma )}{1\over V_1(x_0)+V(x,x_0)}
        \Big({1\over 1+d(x,x_0)}\Big)^\gamma \label{(size f1)}
\end{eqnarray}
since $| f_1(y)|\leq | f(y)|$ and $1 + d(y,x_0) \geq C(1 +
d(x,x_0))$ by the form of $f_1$;
\begin{eqnarray}
    \ \ | f_1(y) - f_1(y^\prime )|
    &\leq& |  I(y) | |  f(y) - f(y' )|
        +| f(y' )| |  I(y) - I(y' )| \nonumber\\
    &\leq& C\|  f \| _{G(\beta ,\gamma )}
        \Big({d(y,y')\over 1 + d(x,x_0)}\Big)^\beta
        {1\over V_1(x_0) + V(x,x_0)}
        \Big({1\over 1 + d(x,x_0)}\Big)^\gamma\label{(smoothness f1)}
\end{eqnarray}
for all $y$ and $y^\prime$;
\begin{equation}
    | f_3(y)|
    \leq C\|  f \| _{G(\beta ,\gamma )}
        {1\over V_1(x_0) + V(y,x_0)}
        \Big({1\over 1 + d(y,x_0)}\Big)^\gamma
        \chi_{\lbrace y\in X: d  (y,x_0) > R/8\rbrace};\label{(size f3)}
\end{equation}
\begin{equation}
    \int_X | f_3(y)|  \, d\mu (y)
    \leq C\|  f \| _{G(\beta ,\gamma)}
        \Big({1\over 1 + d(x,x_0)}\Big)^\gamma; 
        \label{(L1 norm f3)}
\end{equation}
and
\begin{eqnarray}
    \bigg| \int_X f_2(y) \, d\mu (y)\bigg|
    &=& \bigg|  - \int_X f_1(y) \, d\mu (y) - \int_X f_3(y)
        \, d\mu (y)\bigg| \nonumber \\
    &\leq& C\|  f \| _{G(\beta ,\gamma )}
        \Big({1\over 1 + d(x,x_0)}\Big)^\gamma. \label{(L1 norm f2)}
\end{eqnarray}

We write
\begin{eqnarray*}
    T(f_1)(x)
    &=& \underbrace{\int_X K(x,y)u(y)[f_1(y) - f_1(x)] \, d\mu (y)}_{p(x)} \\
    && {}+ \underbrace{\int_X K(x,y)v(y)f_1(y) \, d\mu (y)
        + f_1(x)\int_X K(x,y)u(y)\, d\mu (y)}_{q(x)},
\end{eqnarray*}
where $u(y)$ is a smooth cut-off function as in
Lemma~\ref{lemma technique 3 cut-off function}, satisfying
$u(y)\equiv1$ when $y\in B(x, r)$, and $u(y)\equiv0$ when $y\in
B(x,4A_0^2 r)^c$, and where $v(y) := 1 - u(y)$.

In order to estimate the expressions $p(x)$ and $q(x)$, we show
that the kernel $K(x,y)$ of $T$ satisfies the following four
estimates:
\begin{eqnarray}\label{size of C-Z-S-I-O}
    |K(x, y)|
    \leq {{C}\over {V(x, y)}}
\end{eqnarray}
for all $x\not= y$;
\begin{eqnarray}\label{x smooth of C-Z-S-I-O}
    |  K(x, y) - K(x', y) |
    \leq {C\over V(x,y)} \Big({d(x, x')\over d(x,y)}\Big)^\eta
\end{eqnarray}
for $d(x, x')\leq (2A_0)^{-1} d(x, y)$;
\begin{eqnarray}\label{y smooth of C-Z-S-I-O}
    |  K(x, y) - K(x, y') |
    \leq {C\over V(x,y)} \Big({d(y, y')\over d(x,y)}\Big)^\eta
\end{eqnarray}
for $d(y, y')\leq (2A_0)^{-1} d(x, y);$ and
\begin{eqnarray}\label{y smooth2 of C-Z-S-I-O}
    \hspace{0.5cm} |  K(x, y) - K(x, y') - K(x,y') + K(x',y') |
    \leq {C\over V(x,y)} \Big({d(x, x')\over d(x,y)}\Big)^\eta
        \Big({d(y, y')\over d(x,y)}\Big)^\eta
\end{eqnarray}
for $d(x, x') \leq (2A_0)^{-1} d(x, y)$ and $d(y, y')\leq
(2A_0)^{-1} d(x, y)$.

We point out that without additional assumptions on the measure
$\mu$, the decay and smoothness estimates \eqref{e1 in lemma
another version of Lemma 9.1}--\eqref{e2 in lemma another
version of Lemma 9.1} for $D_k(x,y)$ as given in
Lemma~\ref{lemma another version of Lemma 9.1} are not by
themselves sufficient to imply the above estimates for $K(x,y)
= \sum_kD_k(x,y)$.
Fortunately, however, in our setting, although we have no
additional assumptions on~$\mu$, we do have the special form
$K(x,y) = \sum_k D_k(x,y) = \sum_{k,\alpha} \psi_\alpha^k(x)
\psi_\alpha^k(y)$ in terms of the wavelet basis, rather than
the operators $D_k = S_{k+1} - S_k$ from the classical case.
Instead of using the estimates \eqref{e1 in lemma another
version of Lemma 9.1}--\eqref{e2 in lemma another version of
Lemma 9.1} directly, we can use the estimate~\eqref{claim}
proved below, together with the approach of Lemma~9.1
in~\cite{AH}. The estimates \eqref{size of C-Z-S-I-O}--\eqref{y
smooth of C-Z-S-I-O} for our $K(x,y)$ were proved in Lemmas~9.2
and~9.3 of~\cite{AH}. Thus, we only need to show the
estimate~\eqref{y smooth2 of C-Z-S-I-O}. To do this, following
the approach of Lemma~9.3 in~\cite{AH}, we first claim that if
$d(x,x') \leq (2A_0)^{-1} d(x,y)$ and $d(y,y') \leq (2A_0)^{-1}
d(x,y)$, then
\begin{eqnarray}\label{claim}
    \lefteqn{\sum_{\alpha\in\mathscr{Y}^k}
        \Big|\big[\psi_{\alpha}^k(x) - \psi_{\alpha}^k(x')\big]
        \big[\psi_{\alpha}^k(y) - \psi_{\alpha}^k(y')\big]\Big|} \hspace{1cm}\\
    &&\leq {C\over V(x,\delta^k)}
        \min\Big\{ 1, \Big({ d(x,x')\over\delta^k }\Big)^\eta\Big\}
        \min\Big\{ 1, \Big({ d(y,y')\over\delta^k }\Big)^\eta\Big\}\nonumber\\
    &&\hskip1cm {}\times \exp(-\nu(\delta^{-k} d(x,\mathscr{Y}^k))^a)
        \exp(-\nu(\delta^{-k} d(x,y))^a).\nonumber
\end{eqnarray}
Recall that $a := (1 + \log_2 A_0)^{-1}$ is the exponent
defined in \eqref{eqn:defn_of_a} in Theorem~\ref{theorem AH
orth basis}.

We prove~\eqref{claim} following the method used to prove the
second assertion in Lemma~9.1 of~\cite{AH}. We consider four
cases. First suppose $\delta^k \geq d(x,x')$ and $\delta^k\geq
d(y,y')$. Then
\begin{eqnarray*}
    \lefteqn{\Big|\big[\psi_{\alpha}^k(x) - \psi_{\alpha}^k(x')\big]
        \big[\psi_{\alpha}^k(y) - \psi_{\alpha}^k(y')\big]\Big|}\hspace{1cm}\\
    &&\leq {C\over V(x,\delta^k)}
        \Big({ d(x,x')\over\delta^k }\Big)^\eta
        \Big({ d(y,y')\over\delta^k }\Big)^\eta
        \exp(-\nu(\delta^{-k} d(x,y_\alpha^k))^a)
        \exp(-\nu(\delta^{-k} d(y,y_\alpha^k))^a)\\
    &&\leq {C\over V(x,\delta^k)}
        \Big({ d(x,x')\over\delta^k }\Big)^\eta
        \Big({ d(y,y')\over\delta^k }\Big)^\eta
        \exp(-\nu(\delta^{-k} d(x,y_\alpha^k))^a)
        \exp(-\nu(\delta^{-k} d(x,y))^a).
\end{eqnarray*}
(Here the constant~$\nu$ changes from line to line to
accommodate the constant~$A_0$ that arises from the use of the
quasi-triangle inequality.) The sum over
$\alpha\in\mathscr{Y}^k$ of the fourth factor on the right-hand
side is dominated by the expression $\exp(-\nu(\delta^{-k}
d(x,\mathscr{Y}^k))^a)$.

Second, suppose $\delta^k \geq d(x,x')$ and $\delta^k < d(y,y')
\leq (2A_0)^{-1} d(x,y)$. Then we have
\begin{eqnarray*}
    \lefteqn{\Big|\big[\psi_{\alpha}^k(x) - \psi_{\alpha}^k(x')\big]
        \big[\psi_{\alpha}^k(y) - \psi_{\alpha}^k(y')\big]\Big|}\hspace{1cm}\\
    &&\leq \Big|\big[\psi_{\alpha}^k(x) - \psi_{\alpha}^k(x')\big] \psi_{\alpha}^k(y)\Big|
        + \Big|\big[\psi_{\alpha}^k(x)-\psi_{\alpha}^k(x')\big] \psi_{\alpha}^k(y')\Big|.
\end{eqnarray*}
Then, by using the second estimate in Lemma~9.1 from~\cite{AH},
and the quasi-triangle inequality, we obtain that
\begin{eqnarray*}
    &&\lefteqn{\Big|\big[\psi_{\alpha}^k(x) - \psi_{\alpha}^k(x')\big]
        \big[\psi_{\alpha}^k(y) - \psi_{\alpha}^k(y')\big]\Big|}\hspace{1cm}\\
    &&\leq {C\over V(x,\delta^k)}
        \Big({ d(x,x')\over\delta^k }\Big)^\eta
        \exp(-\nu(\delta^{-k} d(x,y_\alpha^k))^a)
        \exp(-\nu(\delta^{-k} d(x,y))^a).
\end{eqnarray*}
Here the sum over $\alpha\in\mathscr{Y}^k$ of the third factor
on the right-hand side is dominated by $\exp(-\nu(\delta^{-k}
d(x,\mathscr{Y}^k))^a)$.

The other two cases, namely when $\delta^k < d(x,x') \leq
(2A_0)^{-1} d(x,y) $ and $\delta^k \geq d(y,y')$, and when
$\delta^k < d(x,x') \leq (2A_0)^{-1} d(x,y) $ and $\delta^k <
d(y,y') \leq (2A_0)^{-1} d(x,y)$, can be handled similarly. We
omit the details.

Combining the estimates for all four cases above, we have
established the claim~(\ref{claim}).

Now we verify (\ref{y smooth2 of C-Z-S-I-O}).  From the
definition of $K(x,y)$ and the claim (\ref{claim}), we see that
\begin{eqnarray*}
    \lefteqn{|  K(x, y) - K(x, y')-K(x,y')+K(x',y') |}\hspace{0.3cm}\\
    &&\leq \sum_k\sum_{\alpha\in\mathscr{Y}^k}
        \Big|\big[\psi_{\alpha}^k(x) - \psi_{\alpha}^k(x')\big]
        \big[\psi_{\alpha}^k(y) - \psi_{\alpha}^k(y')\big]\Big|\\
    &&\leq \sum_{k:\ \delta^k\geq (2A_0)^{-1} d(x,y) }
        {C\over V(x,\delta^k)}\Big({ d(x,x')\over\delta^k }\Big)^\eta
        \Big({ d(y,y')\over\delta^k }\Big)^\eta
        \exp(-\nu(\delta^{-k} d(x,\mathscr{Y}^k))^a)\\
    && {}+ \sum_{\substack{k:\ \delta^k < (2A_0)^{-1} d(x,y), \\ d(x,x')\leq\delta^k, d(y,y')\leq\delta^k }}
        {C\over V(x,y)}\Big( {d(x,y)\over\delta^k} \Big)^{\omega}
        \Big({ d(x,x')\over\delta^k }\Big)^\eta
        \Big({ d(y,y')\over\delta^k }\Big)^\eta
        \exp(-\nu(\delta^{-k} d(x,y))^a)\\
    && {}+ \sum_{\substack{k:\ \delta^k < (2A_0)^{-1} d(x,y), \\ d(x,x')>\delta^k, d(y,y')\leq\delta^k }}
        {C\over V(x,y)}\Big( {d(x,y)\over\delta^k} \Big)^{\omega}
        \Big({ d(y,y')\over\delta^k }\Big)^\eta
        \exp(-\nu(\delta^{-k} d(x,y))^a)\\
    && {}+ \sum_{\substack{k:\ \delta^k < (2A_0)^{-1} d(x,y), \\ d(x,x')\leq\delta^k, d(y,y')>\delta^k }}
        {C\over V(x,y)}\Big( {d(x,y)\over\delta^k} \Big)^{\omega}
        \Big({ d(x,x')\over\delta^k }\Big)^\eta
        \exp(-\nu(\delta^{-k} d(x,y))^a)\\
    && {}+ \sum_{\substack{k:\ \delta^k < (2A_0)^{-1} d(x,y), \\ d(x,x')>\delta^k, d(y,y')>\delta^k }}
        {C\over V(x,y)}\Big( {d(x,y)\over\delta^k} \Big)^{\omega}
        \exp(-\nu(\delta^{-k} d(x,y))^a)\\
    &&=: B_1 + B_2 + B_3 + B_4 + B_5.
\end{eqnarray*}
Following Lemma~8.3 in \cite{AH}, an application of the
estimate in Remark~\ref{remark:crucial_estimate} shows that
$B_1$ is bounded by ${C\over V(x,y)} \big({ d(x,x')\over d(x,y)
}\big)^\eta \big({ d(y,y')\over d(x,y) }\big)^\eta$. Further,
$B_2$ satisfies the same estimate since $\sum_{m=0}^\infty
\delta^{-m(\omega+\eta)} \exp\{-\nu\delta^{-ma}\} \leq C$. We
can deal with $B_3$, $B_4$ and $B_5$ similarly. Thus~\eqref{y
smooth2 of C-Z-S-I-O} is proved.

We remark that the estimate~\eqref{y smooth2 of C-Z-S-I-O} is
crucial for the proof of~\eqref{e1 for smooth estimate of the
remainning terms}; see the estimate for $T(f_2)$ below.

Now that we have established the estimates~\eqref{size of
C-Z-S-I-O}--\eqref{y smooth2 of C-Z-S-I-O} on the
kernel~$K(x,y)$ of~$T$, we return to estimating the expressions
$p(x)$ and~$q(x)$. The size condition on the kernel $K(x,y)$
and the smoothness condition~(\ref{(smoothness f1)}) on~$f_1$
yield
\begin{eqnarray*}
    |  p(x)|
    &\leq& C\|  f \| _{G(\beta ,\gamma )}
        \int_{ d  (x,y)\leq 4A_0^2 r }
        {1\over V(x,y)}
        \Big({{ d(x,y) }\over{1 + d(x,x_0)}}\Big)^\beta \\
    && {}\times {1\over V_1(x_0) + V(x,x_0)}
        \Big({1\over 1 + d(x,x_0)}\Big)^\gamma \, d\mu (y)\\
    &\leq& C\|  f \| _{G(\beta ,\gamma )}
        \Big({{ d(x,x') }\over{1 + d(x,x_0)}}\Big)^\beta
        {1\over V_1(x_0) + V(x,x_0)}
        \Big({1\over 1 + d(x,x_0)}\Big)^\gamma.
\end{eqnarray*}

This estimate still holds when $x$ is replaced by $x^\prime $,
for $ d (x,x^\prime ) = r$. Thus
\[
    |  p(x) - p(x^\prime )|
    \leq C\|  f \| _{G(\beta ,\gamma )}
        \Big({{ d  (x,x')}\over{1 + d(x,x_0)}}\Big)^\beta
        {1\over V_1(x_0) + V(x,x_0)}
        \Big({1\over 1+d(x,x_0)}\Big)^\gamma.
\]
For $q(x)$, since $T1 = 0$ (by the definition of $D_k(x,y)$ and
the cancellation property of~$\psi^k_\alpha$), we obtain
\begin{eqnarray*}
    q(x) - q(x^\prime )
    &=& \int_X [K(x,y) - K(x^\prime ,y)] \, v(y) \,
        [f_1(y) - f_1(x)] \, d\mu (y) \\
    && {}+ [f_1(y) - f_1(x)]\int_X K(x,y)u(y) \, d\mu (y)\\
    &=:& \textup{(E)} + \textup{(F)}.
\end{eqnarray*}

We claim that there exists a constant $C$ such that for
all~$x$,
\begin{equation}\label{eqn:L1_bound_on_K_u}
    \Big|\int_X K(x,y)u(y) \, d\mu (y)\Big|
    \leq C.
\end{equation}
Assuming this claim (which is proved below), together with the
estimate for $f_1$ in (\ref{(smoothness f1)}), we find that
\[
    | \textup{(F)} |
    \leq C| f_1(x) - f_1(x^\prime )|
    \leq C\|  f \| _{G(\beta,\gamma )}
        \Big({{ d(x,x')}\over{1 + d(x,x_0)}}\Big)^\beta
        {1\over V_1(x_0) + V(x,x_0)}
        \Big({1\over 1 + d(x,x_0)}\Big)^\gamma.
\]

Applying the smoothness estimates for both $f_1$ and $K(x,y)$,
we obtain
\begin{eqnarray*}
    | \textup{(E)} |
    &\leq& C \int \limits_{ d  (x,y)\geq 4A^2_0 r }|  K(x,y) -
        K(x',y)| \, |  v(y) | \, | f_1(y) - f_1(x)|  \, d\mu (y)\\
    &\leq& C\|  f \| _{G(\beta ,\gamma )}\int _{ d  (x,y)\geq
        4A_0^2 r }{1\over V(x,y)}\Big({ d (x,x') \over d(x,y)
        }\Big)^\eta \\
    &&\hskip2cm {}\times \Big({{ d  (x,y)
        }\over{1+d(x,x_0)}}\Big)^\beta{1\over V_1(x_0)+V(x,x_0)}\Big({1\over
        1+d(x,x_0)}\Big)^\gamma \, d\mu (y)\\
    &\leq& C\|  f \| _{G(\beta ,\gamma )} \Big({{ d  (x,x')
        }\over{1+d(x,x_0)}}\Big)^\beta{1\over V_1(x_0)+V(x,x_0)}\Big({1\over
        1+d(x,x_0)}\Big)^\gamma,
\end{eqnarray*}
since $\beta < \eta$. Therefore
\[
    |  T(f_1)(x) - T(f_1)(x^\prime )|
    \leq C\|  f \| _{G(\beta ,\gamma )}
    \Big({{ d(x,x')} \over {1 + d(x,x_0)}}\Big)^\beta
    {1\over V_1(x_0) + V(x,x_0)}
    \Big({1\over 1 + d(x,x_0)}\Big)^\gamma.
\]

We consider three cases. First suppose $d(x,x^\prime ) = r \leq
(20A^2_0)^{-1}(1 + R)$ and $R\geq 10$. Then the points $x$ and
$x^\prime $ are not in the supports of $f_2$ and~$f_3$. Using
the double smoothness and smoothness conditions (\eqref{y
smooth2 of C-Z-S-I-O} and~\eqref{x smooth of C-Z-S-I-O}
respectively) on $K(x,y)$, and the estimate~(\ref{(L1 norm
f2)}) of~$f_2$, we find
\begin{eqnarray*}
    && |  T(f_2)(x) - T(f_2)(x^\prime )|
        = \bigg| \int_X [ K(x,y) - K(x^\prime,y)] f_2(y)
        \, d\mu (y)\bigg|\\
    &&\qquad \leq \int_X |  K(x,y) - K(x^\prime,y) - K(x,x_0) - K(x^ \prime,x_0)|
        \, | f_2(y)|  \, d\mu (y)\\
    &&\qquad\quad {}+ | K(x,x_0) - K(x^\prime ,x_0)|
        \bigg| \int_X f_2(y) \, d\mu (y) \bigg| \\
    && \qquad \leq C\|  f \| _{G(\beta ,\gamma )}
        \Big({{ d(x,x')} \over {1 + d(x,x_0)}}\Big)^\beta
        {1\over V_1(x_0) + V(x,x_0)}
        \Big({1\over 1 + d(x,x_0)}\Big)^\gamma.
\end{eqnarray*}
Also,
\begin{eqnarray*}
    && |  T(f_3)(x) - T(f_3)(x^\prime )| = \bigg|
        \int_X [ K(x,y) - K(x^\prime ,y)]f_3(y) \, d\mu (y)\bigg|\\
    && \qquad \leq C\int_{ d  (x,y)\geq {{R}\over{8}}\geq 2A_0 r}
        {1\over V(x,y)} \Big({ d (x,x') \over d(x,y) }\Big)^\eta
        | f_3(y)| \, d\mu (y)\\
    && \qquad \leq C\|  f \| _{G(\beta ,\gamma )}
        \Big( {{ d(x,x')} \over {1 + d(x,x_0)}}\Big)^\beta
        {1\over V_1(x_0) + V(x,x_0)}
        \Big({1\over 1 + d(x,x_0)}\Big)^\gamma.
\end{eqnarray*}
In the second case, where $ d (x, x_0)= R$ and $(2A_0)^{-1} (1
+ R) \geq  d (x,x' ) = r \geq (20A_0^2)^{-1} (1 + R)$, the
desired estimate for $T(f)(x)$ follows from the estimate
of~(\ref{size decay estimate of the remaining terms}). So we
need only consider the third case, where $R \leq 10$ and $r
\leq 11/(20A_0)$. This case is similar, and indeed easier. In
fact, all we need to do is to replace $R$ in the proof above
by~10. We leave the details to the reader. This completes the
proof of~(\ref{smooth estimate of the remaining terms}).

To finish the argument for Theorem~\ref{thm reproducing formula
test function}, it remains to establish the
claim~\eqref{eqn:L1_bound_on_K_u}. To do so, we prove that
there exists a constant $C$ such that
\begin{eqnarray}
    \|T\phi \|_\infty
    \leq C
\end{eqnarray}
for all functions $\phi$ with the properties that $\| \phi
\|_\infty \leq 1$ and there exist $x_0 \in X$ and $t > 0$ such
that $\supp \phi \subseteq B(x_0, t)$ and $\|\phi \|_\eta :=
\sup_{x \neq y} \{|\phi(x) - \phi(y)| / d(x,y)^\eta\} \leq
t^{-\eta }$.

We again follow the idea of Meyer's proof in~\cite{M2}.  Let
$\chi_0(x) = h(x)$, where $h(x)$ is a smooth cut-off function
as in Lemma~\ref{lemma technique 3 cut-off function} with the
property that $h(x)\equiv 1$ on $B(x_0,2t)$ and $h(x)\equiv 0$
on $B(x_0,8A_0^2t)^c$. Set $\chi_1 := 1 - \chi_0$. Then $\phi =
\phi \chi_0$ and for all $\psi \in C^\eta_0(X)$,
\begin{align*}
    \langle T\phi,\psi \rangle
    &= \langle K(x, y), \phi (y)\psi (x)\rangle
        = \langle K(x, y), \chi_0(y) \phi(y)\psi (x)\rangle\\
    &= \langle K(x, y), \chi _0(y)[\phi (y) - \phi (x)]\psi (x)\rangle
        + \langle K(x, y), \chi _0(y)\phi (x)\psi (x)\rangle\\
    &:= \textup{(G)} + \textup{(H)}.
\end{align*}
Applying the size condition~\eqref{size of C-Z-S-I-O} on the
kernel $K(x,y)$ yields
$$\vert \textup{(G)} \vert \leq C\Vert \psi \Vert_1.\ \ \ $$
To estimate $\textup{(H)}$, it suffices to show that for $x\in
B(x_0, t)$,
\begin{eqnarray}\label{e1 for T bd on test function}
 \vert T\chi _0(x)\vert \leq C,
\end{eqnarray}
since as $\textup{(H)} = \langle T\chi _0,\phi \psi \rangle$,
we then have
\[
    \vert \textup{(H)} \vert
    \leq \Vert  T\chi _0\Vert_{{L^ \infty }(B({x_0}, t))}
        \Vert \phi \psi \Vert_{{L^1}(B({x_0,t} ))}
    \leq C\Vert \psi \Vert _1.
\]

To show (\ref{e1 for T bd on test function}), we use Meyer's
idea again~\cite{M2}. Take $\psi \in C^\eta (X)$ with $\supp
\psi \subseteq B(x_0, t)$ and $\int_X \psi (x) \, d\mu(x) = 0$.
Since $T1 = 0$ and $\int_X \psi (x) \, d\mu(x) = 0$, and using
the smoothness condition~\eqref{x smooth of C-Z-S-I-O} on
$K(x,y)$, we obtain
\begin{eqnarray*}
    \vert \langle T\chi _0,\psi \rangle \vert
    &=& \vert -\langle T\chi _1, \psi \rangle \vert
    = \bigg\vert \iint_{X\times X} [K(x,y) - K(x_0,y)]
        \chi _1(y) \psi (x) \, d\mu (x) \, d\mu(y) \bigg\vert \\
    &\leq& C\Vert \psi \Vert _1.
\end{eqnarray*}

Thus, $T\chi_0(x) = \Lambda + \gamma(x)$ for $x\in B(x_0, t)$,
where $\Lambda$ is a constant and $\Vert \gamma \Vert_\infty
\leq C$. To estimate $\Lambda$, choose $\phi _1 \in
C_0^\eta(X)$ with $\supp \phi_1 \subseteq B(x_0, t)$, $\Vert
\phi _1\Vert_\infty \leq 1$, $\Vert \phi_1\Vert_\eta \leq
t^{-\eta}$ and $\int_X \phi_1(x) \, d\mu(x) = C t$. Since $T$
is bounded on $L^2(X)$, we have
\[
    \bigg\vert C t \Lambda
        + \int_X \phi _1(x)\gamma (x) \, d\mu(x) \bigg\vert
    = \vert \langle T\chi _0,\phi _1 \rangle \vert
    \leq C t.
\]
Therefore $\vert \Lambda \vert \leq C$, and hence the
claim~\eqref{eqn:L1_bound_on_K_u} is proved. This completes the
proof of Theorem~\ref{thm reproducing formula test function},
modulo the proof of Lemma~\ref{lemma another version of Lemma
9.1}.
\end{proof}

It remains to prove the technical lemma used in the proof of
Theorem~\ref{thm reproducing formula test function}.

\begin{proof}[Proof of Lemma~\ref{lemma another version of Lemma 9.1}]
(i) To establish the decay condition~\eqref{e1 in lemma another
version of Lemma 9.1}, we write
\begin{eqnarray*}
    \big |D_k(x,y)\big |
    = \bigg|\sum_{\alpha \in \mathscr{Y}^k}
        \mu(B(y_\alpha^k,\delta^k))
        {\psi_{\alpha}^k(x)\over \sqrt{\mu(B(y_\alpha^k,\delta^k))}}
        {\psi_{\alpha}^k(y)\over\sqrt{\mu(B(y_\alpha^k,\delta^k))}}\bigg|.
\end{eqnarray*}
By Theorem \ref{theorem wavelet is test function}, we know that
$\psi_\alpha^k(x)/ \sqrt{\mu(B(y_\alpha^k,\delta^k))}$ belongs
to $G(y_\alpha^k,\delta^k,\eta,\gamma + \eta)$. Applying the
size condition~(i) from
Definition~\ref{def-of-test-func-space}, we see that
\begin{eqnarray*}
    \big |D_k(x,y)\big | &\leq& C\sum_{\alpha \in
    \mathscr{Y}^k}\mu(B(y_\alpha^k,\delta^k)) {1\over
    V_{\delta^k}(y_\alpha^k)+V(y_\alpha^k,x)}
    \Big({\delta^k\over \delta^{k}+
    d(y_\alpha^k,x)}\Big)^{\gamma+\eta}\\
    &&\hskip.7cm {}\times {1\over V_{\delta^k}(y_\alpha^k)+V(y_\alpha^k,y)}
    \Big({\delta^k\over \delta^{k}+
    d(y_\alpha^k,y)}\Big)^{\gamma+\eta}.
\end{eqnarray*}

Note that for each $z\in B(y_\alpha^k,\delta^k)$ one can
replace $y_\alpha^k$ by $z$ to get $\delta^k
+d(y_\alpha^k,x)\sim\delta^k + d(z,x)$ and
$V_{\delta^k}(y_\alpha^k) + V(y_\alpha^k,x) \sim
V_{\delta^k}(z)+V(z,x)$, and similarly for $\delta^k
+d(y_\alpha^k,y)$ and $V_{\delta^k}(y_\alpha^k) +
V(y_\alpha^k,y).$ Thus, first replacing
$\mu(B(y_\alpha^k,\delta^k))$ by $\int_{B(y_\alpha^k,\delta^k)}
\, d\mu(z)$ and then replacing $y_\alpha^k$ by $z$, and finally
summing up over $\alpha \in \mathscr{Y}^k$, we find that the
last sum above is bounded by
\begin{eqnarray*}
    &&C \int_{X} {1\over V_{\delta^k}(z)+V(z,x)}
        \Big({\delta^k\over \delta^{k}+
        d(z,x)}\Big)^{\gamma+\eta}
        {1\over V_{\delta^k}(z)+V(z,y)}
        \Big({\delta^k\over \delta^{k}+
        d(z,y)}\Big)^{\gamma+\eta} \, d\mu(z)\\
    &&=: \textup{(P)} + \textup{(Q)}.
\end{eqnarray*}
Here $\textup{(P)}$ is the result of integrating over the set
$d(x,z) \leq (2A_0)^{-1}(\delta^k + d(x,y))$ and $\textup{(Q)}$
over the set $d(x,z) > (2A_0)^{-1}(\delta^k + d(x,y))$. To
estimate $\textup{(P)}$, note that if $d(x,z) \leq
(2A_0)^{-1}(\delta^k + d(x,y))$ and $2\delta^k \leq d(x,y),$
then $d(y,z) > (10A_0)^{-1}(\delta^k + d(x,y))$ and by the
doubling property,
\[
    V(z,y)
    = \mu(B(y,d(z,y)))
    \geq \mu(B(y,(10A_0)^{-1}d(x,y)))
    \geq (10A_0)^{-\omega} V(x,y).
\]
Therefore, $V_{\delta^k}(x) + V(x,y) \leq C (V_{\delta^k}(z) +
V(z,y))$. Next, if $d(x,z)\leq (2A_0)^{-1}(\delta^k + d(x,y))$
and $2\delta^k > d(x,y)$, then $d(x,z) \leq
3(2A_0)^{-1}\delta^k$. For this case, first suppose that
$d(x,z) \leq \delta^k$ and hence, $V_{\delta^k}(z)\sim
V_{\delta^k}(x)$. On the other hand, if $d(x,z) > \delta^k$,
then $(2A_0 - 1) \delta^k \leq (2A_0)^{-1} d(x,y)$ and hence
$V_{\delta^k}(x) \leq C V(x,y).$ Therefore, in this case, again
$V_{\delta^k}(x) + V(x,y) \leq C (V_{\delta^k}(z) + V(z,y))$
and thus we get
\begin{eqnarray*}
    \textup{(P)}
    \leq C \frac{1}{V_{\delta^k}(x)+V(x,y)}
        \Big({\delta^k\over \delta^{k} + d(x,y)}\Big)^{\gamma+\eta}
    \leq C \frac{1}{V_{\delta^k}(x)+V(x,y)}
        \Big({\delta^k\over \delta^{k} + d(x,y)}\Big)^{\gamma},
\end{eqnarray*}
as required. The estimate for $\textup{(Q)}$ is the same, but
with $x$ and~$y$ reversed.

\noindent (ii) To establish the smoothness condition~\eqref{e2
in lemma another version of Lemma 9.1}, we write
\begin{eqnarray*}
    \lefteqn{\big |D_k(x,y)-D(x,y')\big |}\hspace{0.5cm}\\
    &&= \sum_{\alpha \in \mathscr{Y}^k}
        \mu(B(y_\alpha^k,\delta^k))
        \bigg|{\psi_{\alpha}^k(x)\over \sqrt{\mu(B(y_\alpha^k,\delta^k))}}
        \Big[{\psi_{\alpha}^k(y)\over\sqrt{\mu(B(y_\alpha^k,\delta^k))}}
        - {\psi_{\alpha}^k(y')\over\sqrt{\mu(B(y_\alpha^k,\delta^k))}}\Big]\bigg|\\
    &&=: \textup{(R)} + \textup{(S)}.
\end{eqnarray*}
Here $\textup{(R)}$ is the result of summing over the set of
$\alpha \in \mathscr{Y}^k$ such that $d(y,y') \leq (2A_0)^{-1}
(\delta^k+d(y,y_\alpha^k))$ or $d(y,y') \leq (2A_0)^{-1}
(\delta^k+d(y',y_\alpha^k))$, and $\textup{(S)}$ over the set
of $\alpha \in \mathscr{Y}^k$ such that $d(y,y') > (2A_0)^{-1}
(\delta^k + d(y,y_\alpha^k))$ and $d(y,y') > (2A_0)^{-1}
(\delta^k + d(y',y_\alpha^k))$.

For $\textup{(R)}$, use the size condition
(Definition~\ref{def-of-test-func-space}(i)) for the first
factor $\psi_\alpha^k(x)/ \sqrt{\mu(B(y_\alpha^k,\delta^k))}$
and the H\"older regularity condition
(Definition~\ref{def-of-test-func-space}(ii)) for the terms
$\psi_\alpha^k(y)/ \sqrt{\mu(B(y_\alpha^k,\delta^k))}$ in the
second factor. We find that
\begin{eqnarray*}
    &&\textup{(R)}
    \leq C \sum_{\alpha \in \mathscr{Y}^k}\mu(B(y_\alpha^k,\delta^k))
        \Big( {d(y,y')\over\delta^k +d(y_\alpha^k,y)}
        \Big)^\eta{1\over V_{\delta^k}(y_\alpha^k)+V(y_\alpha^k,y)}
        \Big({\delta^k\over \delta^{k}+
        d(y_\alpha^k,y)}\Big)^{\gamma+\eta}\\
    &&\hskip2cm {} \times \, {1\over V_{\delta^k}(y_\alpha^k)+V(y_\alpha^k,x)}
        \Big({\delta^k\over \delta^{k}+
        d(y_\alpha^k,x)}\Big)^{\gamma+\eta}.
\end{eqnarray*}
Applying the same proof as for (\ref{e1 in lemma another
version of Lemma 9.1}), we see that the last sum above is
bounded by
\begin{eqnarray*}
    \lefteqn{C \int_{X} \Big(
        {d(y,y')\over\delta^k +d(z,y)}
        \Big)^\eta{1\over V_{\delta^k}(z)+V(z,y)}
        \Big({\delta^k\over \delta^{k}+
        d(z,y)}\Big)^{\gamma+\eta}} \hspace{0.5cm}\\
    &&\hskip1cm {}\times \,
        {1\over V_{\delta^k}(z)+V(z,x)}
        \Big({\delta^k\over \delta^{k}+
        d(z,x)}\Big)^{\gamma+\eta}\, d\mu(z)\\
    &&\leq C\Big( {d(y,y')\over\delta^k +d(x,y)}
        \Big)^\eta{1\over V_{\delta^k}(x)+V(x,y)}
        \Big({\delta^k\over \delta^{k}+
        d(x,y)}\Big)^{\gamma}.
\end{eqnarray*}

To deal with $\textup{(S)}$, we can write
\begin{eqnarray*}
    \textup{(S)}
    &\leq& \sum_{\alpha \in \mathscr{Y}^k : d(y,y') > (2A_0)^{-1} (\delta^k+ d(y,y_\alpha^k))}
        \mu(B(y_\alpha^k,\delta^k))
        \bigg|{\psi_{\alpha}^k(x)\over \sqrt{\mu(B(y_\alpha^k,\delta^k))}}
        {\psi_{\alpha}^k(y)\over\sqrt{\mu(B(y_\alpha^k,\delta^k))}}\bigg|\\
    &&\hskip0.2cm+\sum_{\alpha \in \mathscr{Y}^k : d(y,y') > (2A_0)^{-1} (\delta^k + d(y',y_\alpha^k))}
        \mu(B(y_\alpha^k,\delta^k))
        \bigg|{\psi_{\alpha}^k(x)\over \sqrt{\mu(B(y_\alpha^k,\delta^k))}}
        {\psi_{\alpha}^k(y')\over\sqrt{\mu(B(y_\alpha^k,\delta^k))}}\bigg|.
\end{eqnarray*}
For the first sum, following the same approach as for
$\textup{(R)}$ but with $d(y,y') > (2A_0)^{-1}(\delta^k +
d(y,y_\alpha^k))$, we must deal with the integral
$$\int_{X} \!\!\Big( {d(y,y')\over\delta^k +d(z,y)}
    \Big)^\eta{1\over V_{\delta^k}(z)+V(z,y)}
    \Big({\delta^k\over \delta^{k}+
    d(z,y)}\Big)^{\gamma+\eta}
    {1\over V_{\delta^k}(z)+V(z,x)}
    \Big({\delta^k\over \delta^{k}+
    d(z,x)}\Big)^{\gamma+\eta} \, d\mu(z).$$
Applying the same proof as for~$\textup{(R)}$, but using the
size condition (Definition~\ref{def-of-test-func-space}(i)) for
both factors, we obtain that this integral is bounded by
\begin{eqnarray*}
    C \Big({d(y,y')\over\delta^k + d(x,y)}\Big)^\eta
        {1\over V_{\delta^k}(x) + V(x,y)}
        \Big({\delta^k\over \delta^{k} + d(x,y)}\Big)^{\gamma}.
\end{eqnarray*}
The second sum is similar to the first one, with $y$ and $y'$
reversed. Thus, by the same proof we find that the second sum
is bounded by
\begin{eqnarray*}
    C \Big({d(y,y')\over\delta^k + d(x,y')}\Big)^\eta
        {1\over V_{\delta^k}(x) + V(x,y')}
        \Big({\delta^k\over \delta^{k} + d(x,y')}\Big)^{\gamma}.
\end{eqnarray*}
Note that the fact $d(y,y') \leq (2A_0)^{-1}(\delta^k +
d(x,y))$ implies that $\delta^k + d(x,y) \sim \delta^k+d(x,y')$
and $V_{\delta^k}(x) + V(x,y) \sim V_{\delta^k}(x) + V(x,y')$.
Therefore, we obtain the desired estimate for the second sum.

\noindent (iii) The proof for the double smoothness
condition~\eqref{e3 in lemma another version of Lemma 9.1} is
similar to that for~\eqref{e2 in lemma another version of Lemma
9.1}, and we omit the details.

This completes the proof of Lemma~\ref{lemma another version of
Lemma 9.1}.
\end{proof}


\subsection{Product test functions,
distributions, and wavelet reproducing
formula}\label{sec:producttestfunctions}

We now consider the product setting $(X_1,d_1,\mu_1) \times
(X_2,d_2,\mu_2)$, where $(X_i,d_i,\mu_i)$, $i = 1$, 2, are
spaces of homogeneous type as defined in the Introduction. For
$i = 1$, 2, let $C_{\mu_i}$ be the doubling constant as in
inequality~(\ref{doubling condition}), let $\omega_i$ be the
upper dimension as in inequality~(\ref{upper dimension}), and
let $A_0^{(i)}$ be the constant in the quasi-triangle
inequality~\eqref{eqn:quasitriangleineq}. In this subsection we
use the notation $(x,y)$ for an element of~$X_1 \times X_2$.

On each $X_i$ there is a wavelet
basis~$\{\psi^{k_i}_{\alpha_i}\}$, with H\"older
exponent~$\eta_i$ as in inequality~\eqref{Holder-regularity}.

We now define the spaces of test functions and distributions on
the product space $X_1\times X_2$.

\begin{definition}\label{def-of-test-func-space-product}
    \textup{(Product test functions)} Let $(x_0,y_0)\in
    X_1\times X_2$ and $r = (r_1,r_2)$ with $r_1$, $r_2 > 0$.
    Take $\beta = (\beta_1,\beta_2)$, with $\beta_1 \in
    (0,\eta_1]$, $\beta_2 \in (0,\eta_2]$, and $\gamma =
    (\gamma_1,\gamma_2)$ with $\gamma_1$, $\gamma_2 > 0$. A
    function $f(x,y)$ defined on $X_1\times X_2$ is said to be
    a {\it test function of type} $(x_0,y_0;r;\beta;\gamma)$ if
    the following three conditions hold.
    \begin{enumerate}
        \item[(a)] For each fixed $y\in X_2$, $f(x,y)$ as a
            function of the variable $x\in X_1$ is a test
            function in $G(x_0,r_1,\beta_1,\gamma_1)$.

        \item[(b)] For each fixed $x\in X_1$, $f(x,y)$ as a
            function of the variable $y\in X_2$ is a test
            function in $G(y_0,r_2,\beta_2,\gamma_2)$.

        \item[(c)] The following properties hold:

            \begin{enumerate}
                \item[(i)] (Size condition) For all $y
                    \in X_2$,
                    \[
                        \Vert f(\cdot,y)\Vert_{G(x_0,r_1,\beta_1,\gamma_1)}
                        \leq C \frac{\displaystyle 1}{\displaystyle V_{r_2}(y_0)+V(y,y_0)}
                            \Big(\frac{\displaystyle r_2}{\displaystyle r_2 + d_2(y,y_0)}
                            \Big)^{\gamma_2}.
                    \]

                \item[(ii)] (H\"older regularity
                    condition) For all $y$, $y'\in X_2$
                    such that $d_2(y,y')\leq
                    (2A_0^{(2)})^{-1} (r_2 +
                    d_2(y,y_0))$, we have
                    \[
                        \Vert f(\cdot,y) - f(\cdot,y')
                            \Vert_{G(x_0,r_1,\beta_1,\gamma_1)}
                        \leq C \Big(\frac{d_2(y,y')}{r_2 + d_2(y,y_0)}\Big)^{\beta_2}
                            \frac{1}{V_{r_2}(y_0) + V(y_0,y)}
                            \Big(\frac{r_2}{r_2 + d_2(y,y_0)}\Big)^{\gamma_2}.
                    \]

                \item[(iii)] Properties (i) and (ii)
                    also hold with $x$ and $y$
                    interchanged.

                \item[(iv)] (Cancellation condition)
                    $\int_{X_1} f(x,y)\, d\mu_1(x) = 0$
                    for all $y\in X_2$, and $\int_{X_2}
                    f(x,y) \linebreak\, d\mu_2(y) = 0$
                    for all $x\in X_1$.
            \end{enumerate}
    \end{enumerate}
\end{definition}

When $f$ is a test function of type $(x_0,y_0;r;\beta;\gamma)$,
we write $f\in G(x_{0},y_{0};r;\beta;\gamma)$. Note the use of
semicolons here to distinguish the product definition from the
one-parameter version.

The expression
\[
    \|f\|_{G(x_{0},y_{0};r;\beta;\gamma)}
    := \inf\{C:\ {\rm(i),\ (ii)\ and\ (iii)}\ \ {\rm hold}\}
\]
defines a norm on $G(x_{0},y_{0};r;\beta;\gamma)$.

We denote by $G(\beta_{1},\beta_{2};\gamma_{1},\gamma_{2})$ the
class $G(x_{0},y_{0};1,1;\beta;\gamma)$ for arbitrary fixed
$(x_{0},y_{0})\in X_1\times X_2$.  Then
$G(x_{0},y_{0};r;\beta;\gamma) =
G(\beta_{1},\beta_{2};\gamma_{1},\gamma_{2})$, with equivalent
norms, for all $(x_{0},y_{0})\in X_1\times X_2$ and $r_1$, $r_2
> 0$. Furthermore,
$G(\beta_{1},\beta_{2};\gamma_{1},\gamma_{2})$ is a Banach
space with respect to the norm on
$G(\beta_{1},\beta_{2};\gamma_{1},\gamma_{2})$.

For $\beta_i \in (0,\eta_i]$ and $\gamma_i > 0$, for $i = 1$,
2, let $\GGp(\beta_1,\beta_2;\gamma_1,\gamma_2)$ be the
completion of the space $G(\eta_1,\eta_2;\gamma_1,\gamma_2)$ in
$G(\beta_1,\beta_2;\gamma_1,\gamma_2)$ in the norm of
$G(\beta_1,\beta_2;\gamma_1,\gamma_2)$. We define the norm on
$\GGp(\beta_{1},\beta_{2};\gamma_{1},\gamma_{2})$ by
$\|f\|_{\GGp(\beta_{1},\beta_{2};\gamma_{1},\gamma_{2})} :=
\|f\|_{G(\beta_{1},\beta_{2};\gamma_{1},\gamma_{2})}$.

The (scaled) product wavelets given by
$\psi^{k_1}_{\alpha_1}(x) \psi^{k_1}_{\alpha_1}(y)
\big(\mu_1(B(y^{k_1}_{\alpha_1}, \delta_1^{k_1}))
\mu_2(B(y^{k_2}_{\alpha_2}, \delta_2^{k_2}))\big)^{-1/2}$ are
product test functions in
$G(y^{k_1}_{\alpha_1},y^{k_2}_{\alpha_2}; \delta; \beta;
\gamma)$ for each $\gamma = (\gamma_1,\gamma_2)$ with $\gamma_1
> 0$, $\gamma_2 > 0$, where $\delta = (\delta_1^{k_1},
\delta_2^{k_2})$ and $\beta = (\eta_1,\eta_2)$; this is
straightforward to check.

\begin{definition}
    \textup{(Product distributions)} Let $(x_0,y_0)\in
    X_1\times X_2$ and $r = (r_1,r_2)$ with $r_1$, $r_2 > 0$.
    Take $\beta = (\beta_1,\beta_2)$, with $\beta_1 \in
    (0,\eta_1]$, $\beta_2 \in (0,\eta_2]$, and $\gamma =
    (\gamma_1,\gamma_2)$ with $\gamma_1$, $\gamma_2 > 0$. We
    define the \emph{distribution space}
    $\big(\GGp(\beta_{1},\beta_{2};\gamma_{1},\gamma_{2})\big)'$
    to consist of all linear functionals $\mathcal{L}$ from
    $\GGp(\beta_{1},\beta_{2};\gamma_{1},\gamma_{2})$ to
    $\mathbb{C}$ with the property that there exists a constant
    $C$ such that for all $f\in
    \GGp(\beta_{1},\beta_{2};\gamma_{1},\gamma_{2})$,
    \[
        |\mathcal{L}(f)|
        \leq C \|f\|_{\GGp(\beta_{1},\beta_{2};\gamma_{1},\gamma_{2})}.
    \]
\end{definition}

We have the following version of the wavelet reproducing formula in the
product setting $X_1\times X_2$.

\begin{theorem}\label{thm product wavelet reproducing formula test function}
    \textup{(Product reproducing formula)} Take $\beta_i$,
    $\gamma_i \in (0,\eta_i)$ for $i = 1$, $2$.
    \begin{enumerate}
        \item[(a)] The wavelet reproducing formula
            \begin{eqnarray}\label{product reproducing formula}
                f(x,y)
                = \sum_{k_1}\sum_{\alpha_1 \in \mathscr{Y}^{k_1}}
                    \sum_{k_2}\sum_{\alpha_2 \in \mathscr{Y}^{k_2}}
                    \langle f,\psi_{\alpha_1}^{k_1}\psi_{\alpha_2}^{k_2} \rangle
                    \psi_{\alpha_1}^{k_1}(x)\psi_{\alpha_2}^{k_2}(y)
            \end{eqnarray}
            holds in the space of test functions
            $\GGp(\beta_{1}',\beta_{2}';\gamma_{1}',\gamma_{2}')$
            for each $\beta_{i}' \in (0,\beta_i)$ and
            $\gamma_i' \in (0,\gamma_i)$, for $i = 1$, $2$.

        \item[(b)] The wavelet reproducing
            formula~\eqref{product reproducing formula}
            also holds in the space of distributions
            \linebreak
            $(\GGp(\beta_{1},\beta_{2};\gamma_{1},\gamma_{2}))'$.
    \end{enumerate}
\end{theorem}

\begin{proof}
As before, the wavelet reproducing formula for distributions follows
immediately from that for test functions. The proof for test
functions proceeds by iteration of Theorem~\ref{thm reproducing
formula test function}. Write
\begin{eqnarray*}
    g(x,y)
    &:=& \sum_{|k_1|\le L_1}\sum_{\alpha_1 \in \mathscr{Y}^{k_1}}
        \sum_{|k_2|\le L_2}\sum_{\alpha_2 \in \mathscr{Y}^{k_2}}
        \langle f,\psi_{\alpha_1}^{k_1}\psi_{\alpha_2}^{k_2} \rangle
        \psi_{\alpha_1}^{k_1}(x)\psi_{\alpha_2}^{k_2}(y)-f(x,y)\\
    &=:& g_1(x,y) + g_2(x,y),
\end{eqnarray*}
where
\begin{eqnarray*}
    g_1(x,y)
    &:=& \sum_{|k_1|\le L_1}\sum_{\alpha_1 \in \mathscr{Y}^{k_1}}
    \Big\langle \psi_{\alpha_1}^{k_1},
        \sum_{|k_2|\le L_2}\sum_{\alpha_2 \in \mathscr{Y}^{k_2}}
        \langle f(\cdot,\cdot),\psi_{\alpha_2}^{k_2} \rangle
        \psi_{\alpha_2}^{k_2}(y)\Big\rangle \psi_{\alpha_1}^{k_1}(x)\\
    && {} - \sum_{|k_2|\le L_2}\sum_{\alpha_2 \in \mathscr{Y}^{k_2}}
        \langle f(x,\cdot),\psi_{\alpha_2}^{k_2} \rangle
        \psi_{\alpha_2}^{k_2}(y)
\end{eqnarray*}
and
\begin{eqnarray*}
    g_2(x,y)
    := \sum_{|k_2|\le L_2}\sum_{\alpha_2 \in \mathscr{Y}^{k_2}}
        \langle f(x,\cdot),\psi_{\alpha_2}^{k_2} \rangle
        \psi_{\alpha_2}^{k_2}(y) - f(x,y).
\end{eqnarray*}

To see the convergence in the space of test functions, we
recall the following (one-parameter) estimate on~$X$, as shown
in the proof of Theorem~\ref{thm reproducing formula test
function}: Given $\beta$, $\gamma\in(0,\eta)$, for each $\beta'
\in (0,\beta)$ and $\gamma' \in (0,\gamma)$ there is a constant
$\sigma > 0$ such that for each positive integer~$L$
\begin{eqnarray}\label{eqn:waveletreproducingformula2}
    \Big\|f(\cdot) - \sum_{|k|\leq L}\sum_{\alpha \in \mathscr{Y}^k}
        \langle \psi_{\alpha}^k,f\rangle
        \psi_{\alpha}^k(\cdot)\Big\|_{G(\beta',\gamma')}
    \leq C \delta^{\sigma L}\Vert f\Vert_{G(\beta,\gamma)},
\end{eqnarray}
where $C$ is a constant independent of $f\in
\GGs(\beta,\gamma)$. Note that
inequality~\eqref{eqn:waveletreproducingformula2} is the same
as inequality~\eqref{decay estimate of the remaining terms},
slightly rewritten.
Inequality~\eqref{eqn:waveletreproducingformula2}, together
with the triangle inequality, implies that
\begin{eqnarray}\label{eqn:waveletreproducingformula1}
    \Big\|\sum_{|k|\leq L} \sum_{\alpha \in \mathscr{Y}^k}
        \langle \psi_{\alpha}^k,f\rangle
        \psi_{\alpha}^k(\cdot)\Big\|_{G(\beta',\gamma')}
    \leq C\Vert f\Vert_{G(\beta,\gamma)}.
\end{eqnarray}

We observe that if $f\in G(\beta_1,\beta_2;\gamma_1,\gamma_2),$
then $\Vert f(\cdot, y)\Vert_{G(\beta_1,\gamma_1)},$ as a
function of the variable $y,$ is in $G(\beta_2,\gamma_2)$, and
satisfies $\big\Vert \Vert f(\cdot,
\cdot)\Vert_{G(\beta_1,\gamma_1)}\big\Vert_{G(\beta_2,\gamma_2)}\le
\Vert f\Vert_{G(\beta_1,\beta_2;\gamma_1,\gamma_2)}.$
Similarly, $\big\Vert \Vert f(\cdot,
\cdot)\Vert_{G(\beta_2,\gamma_2)}\big\Vert_{G(\beta_1,\gamma_1)}\le
\Vert f\Vert_{G(\beta_1,\beta_2;\gamma_1,\gamma_2)}.$
Therefore, we obtain
\begin{eqnarray*}
    \Vert g_1(\cdot,y)\Vert_{G(\beta'_1,\gamma'_1)}
    &\le& C \delta^{L_1\sigma}
        \big\Vert \sum_{|k_2|\le L_2}\sum_{\alpha_2 \in \mathscr{Y}^{k_2}}
        \langle f(\cdot,\cdot),\psi_{\alpha_2}^{k_2} \rangle
        \psi_{\alpha_2}^{k_2}(y)\big\Vert_{G(\beta_1,\gamma_1)}\\
    &\le& C \delta^{L_1\sigma}
        \big\Vert \Vert f(\cdot,\cdot)\Vert_{G(\beta_2,\gamma_2)}
        \frac{\displaystyle 1}{\displaystyle
        V_{r_2}(y_0)+V(y_0,y)}\Big(\frac{\displaystyle r_2}{\displaystyle
        r_2+d(y,y_0)}\Big)^{\gamma_2}\big\Vert_{G(\beta_1,\gamma_1)}\\
    &\le& C \delta^{L_1\sigma}
        \Vert f(\cdot,\cdot)\Vert_{G(\beta_1,\beta_2;\gamma_1,\gamma_2)}
        \frac{\displaystyle
        1}{\displaystyle V_{r_2}(y_0)+V(y_0,y)}\Big(\frac{\displaystyle
        r_2}{\displaystyle r_2+d(y,y_0)}\Big)^{\gamma_2},
\end{eqnarray*}
where the first inequality follows
from~\eqref{eqn:waveletreproducingformula2} and the second
inequality follows from~\eqref{eqn:waveletreproducingformula1}.
Similarly,
\begin{eqnarray*}
    \Vert g_2(x,y)\Vert_{G(\beta'_1,\gamma'_1)}
    \le C \delta^{L_2\sigma}\Vert f \Vert_{G(\beta_1,\beta_2;\gamma_1,\gamma_2)}
        \frac{\displaystyle 1}{\displaystyle V_{r_2}(y_0) + V(y_0,y)}
        \Big(\frac{\displaystyle r_2}{\displaystyle r_2+d(y,y_0)}\Big)^{\gamma_2}.
\end{eqnarray*}
Noting that $g(x,y) - g(x,y') = [g_1(x,y) - g_1(x,y')] +
[g_2(x,y) - g_2(x,y')]$, by repeating the same estimates we
obtain
\begin{eqnarray*}
    \lefteqn{\Vert g(\cdot,y) - g(\cdot, y')\Vert_{G(\beta'_1,\gamma'_1)}}
        \hspace{1cm}\\
    &&\le C(\delta^{L_1\sigma} + \delta^{L_2\sigma})\Vert
        f(\cdot,\cdot)\Vert_{G(\beta_1,\beta_2;\gamma_1,\gamma_2)}\\
    &&\hskip1cm {}\times \Big(\frac{\displaystyle
        d(y,y')}{\displaystyle
        r_2+d(y,y_0)}\Big)^{\beta_2}\frac{\displaystyle 1}{\displaystyle
        V_{r_2}(y_0)+V(y_0,y)}\Big(\frac{\displaystyle r_2}{\displaystyle
        r_2+d(y,y_0)}\Big)^{\gamma_2}
\end{eqnarray*}
where $d(y,y')\leq (2A_0^{(2)})^{-1} (r_2+d(y,y_0)).$

The same proof can be carried out for the estimates with $x$
and $y$ interchanged. Hence
\[
    \Vert g \Vert_{G(\beta'_1,\beta'_2;\gamma'_1,\gamma'_2)}
    \le C(\delta^{L_1\sigma} + \delta^{L_2\sigma})
        \Vert f\Vert_{G(\beta_1,\beta_2;\gamma_1,\gamma_2)},
\]
which yields the convergence in
$\GGs(\beta'_1,\beta'_2;\gamma'_1,\gamma'_2)$.
\end{proof}


\section{Littlewood--Paley square functions and
Plancherel--P\'olya inequalities}\label{sec:squarefunctionPP}

We now carry out the philosophy described near the end of the
introduction, in order to establish the Littlewood--Paley
theory for the discrete square function in terms of wavelet
coefficients. We define the discrete and continuous square
functions, and prove their norm-equivalence via
Plancherel--P\'olya inequalities, whose proof takes up most of
this section. Again we begin with the one-parameter case.

\subsection{One-parameter square functions via wavelets, and
Plancherel--P\'olya
inequalities}\label{oneparametersquarefunction} We first apply
the orthonormal wavelet basis constructed in~\cite{AH} to
introduce the discrete Littlewood--Paley square function,
defined via the wavelet coefficients as follows.

\begin{definition}\label{def:discrete_square_function}
    \textup{(Discrete square function in terms of wavelet
    coefficients)} For $f$ in $(\GGs(\beta,\gamma))'$ with
    $\beta$, $\gamma\in(0,\eta)$, the \emph{discrete
    Littlewood--Paley square function $S(f)$ of $f$} is defined
    by
    \begin{eqnarray}
        S(f)(x)
        := \Big\{ \sum_{k}\sum_{\alpha\in\mathscr{Y}^k} \big|
            \langle \psi_{\alpha}^{k},f \rangle
            \widetilde{\chi}_{{Q}_{\alpha}^{k}}(x)
            \big|^2 \Big\}^{1/2},
    \end{eqnarray}
    where $\widetilde{\chi}_{{ Q}_{\alpha}^{k}}(x) := \chi_{{
    Q}_{\alpha}^{k}}(x) \mu({ Q}_{\alpha}^{k})^{-1/2}$ and $\chi_{{
    Q}_{\alpha}^{k}}(x)$ is the indicator function of the dyadic
    cube~${ Q}_{\alpha}^{k}$.
\end{definition}

It is straightforward that $\|S(f)\|_{L^2(X)} =
\|f\|_{L^2(X)}$, since $\lbrace\psi_{\alpha}^{k}\rbrace$ forms
an orthonormal wavelet basis for~$L^2(X)$. However, it is not
easy to see why $\|S(f)\|_{L^p(X)} \sim \|f\|_{L^p(X)}$ for $1
< p < \infty$ with $p\not = 2$. This difficulty is because the
classical method, namely the vector-valued Calder\'on--Zygmund
operator theory, cannot be carried out here due to the lack of
smoothness in the $x$~variable. For this reason, we introduce
the following continuous Littlewood--Paley square function in
terms of the wavelet operators~$D_k$.

\begin{definition}\label{def:continuous_square_function}
    \textup{(Continuous square function in terms of wavelet
    operators)} Let $D_k$ be the operator with kernel $D_k(x,y)
    = \sum_{\alpha \in \mathscr{Y}^k}\psi_{\alpha}^k(x)
    \psi_{\alpha}^k(y)$. For $f\in (\GGs(\beta,\gamma))'$ with
    $\beta$, $\gamma\in(0,\eta)$, the {\it continuous
    Littlewood--Paley square function $S_c(f)$ of $f$} is
    defined by
    \[
        S_c(f)(x)
        := \Big\{ \sum_{k}|D_k(f)(x)|^2 \Big\}^{1/2}.
    \]
\end{definition}

The two main results in this subsection are as follows.

\begin{theorem}\label{theorem Littlewood Paley}
    \textup{(Littlewood--Paley theory)} Suppose $\beta$,
    $\gamma\in (0,\eta)$ and $\frac{\omega}{\omega + \eta} < p
    < \infty$, where $\omega$ is the upper dimension
    of~$(X,d,\mu)$. For $f$ in $(\GGs(\beta,\gamma))'$, we have
    \[
        \|S(f)\|_{L^p(X)}
        \sim \|S_c(f)\|_{L^p(X)}.
    \]
    Moreover, if $1 < p < \infty$, then
    \[
        \|S(f)\|_{L^p(X)}
        \sim \|S_c(f)\|_{L^p(X)} \sim \|f\|_{L^p(X)}.
    \]
\end{theorem}

The key idea in proving Theorem~\ref{theorem Littlewood
Paley} is the following Plancherel--P\'olya type
inequalities.

\begin{theorem}\label{theorem P P inequality one-parameter}
    \textup{(Plancherel--P\'olya inequalities)} Suppose
    $\beta$, $\gamma\in (0,\eta)$ and $\frac{\omega}{\omega +
    \eta} < p < \infty$, where $\omega$ is the upper dimension
    of~$(X,d,\mu)$. Fix $N\in \N$. Then there is a positive
    constant~$C$ such that for all $f\in
    (\GGs(\beta,\gamma))'$, we have
    \begin{eqnarray}\label{eqn:PP1}
        &&\Big\|\Big\{\sum_{k'}\sum_{\alpha'\in \mathscr{X}^{k'+N}}
            \Big[\sup_{z\in {Q}^{{k^\prime +N }}_{\alpha ^\prime }}
            \big|D_{k^\prime}(f)(z)\big|^2\Big]
            \chi_{{ Q}^{{k^\prime +N }}_{\alpha^\prime }}(\cdot)
            \Big\}^{1/2}\Big\|_{L^p(X)}\nonumber\\
        &&\hskip.5cm \leq C \Big\|\Big\{\sum_{k}\sum_{\alpha\in\mathscr{Y}^k}
            \big|\langle \psi_{\alpha}^{k},f \rangle
            \widetilde{\chi}_{{ Q}_{\alpha}^{k}}(\cdot) \big|^2
            \Big\}^{1/2} \Big\|_{L^p(X)}.
    \end{eqnarray}
    Moreover, for a fixed sufficiently large integer~$N$ ($N$ will be
    determined later in the proof), there is a positive
    constant~$C$ such that for all $f\in
    (\GGs(\beta,\gamma))'$, we have
    \begin{eqnarray}\label{eqn:PP2}
        &&\Big\|\Big\{ \sum_{k}\sum_{\alpha\in\mathscr{Y}^k}
            \big| \langle \psi_{\alpha}^{k},f \rangle
            \widetilde{\chi}_{{Q}_{\alpha}^{k}}(\cdot) \big|^2
            \Big\}^{1/2} \Big\|_{L^p(X)}\nonumber\\
        &&\hskip.5cm \leq C \Big\|\Big\{\sum_{k'}\sum_{\alpha'\in \mathscr{X}^{k'+N}}
            \Big[\inf_{z\in { Q}^{{k^\prime + N}}_{\alpha ^\prime }}
            \big|D_{k^\prime}(f)(z)\big|^2\Big]
            \chi_{{Q}^{{k^\prime +N }}_{\alpha^\prime }}(\cdot)
            \Big\}^{1/2}\Big\|_{L^p(X)}.
    \end{eqnarray}
\end{theorem}

Note that in each of the inequalities~\eqref{eqn:PP1}
and~\eqref{eqn:PP2}, on one side, for each $k\in\Z$ the sum
runs over the set $\mathscr{Y}^k$, while on the other side for
each $k'\in\Z$ the sum runs over the set $\mathscr{X}^{k' +
N}$. Besides the distinction between $\mathscr{Y}$ and
$\mathscr{X}$, the other difference here is that in the
expressions involving $D_{k'}$, it is not sufficient to sum at
the scale of~$k'$, but rather, following~\cite{DH}, we must sum
over all cubes at the smaller scale $k' + N$.

\begin{proof}[Proof of Theorem~\ref{theorem Littlewood Paley}]
\label{prf:Littlewood-Paley}
Theorem~\ref{theorem Littlewood Paley} follows from
Theorem~\ref{theorem P P inequality one-parameter}, by standard
arguments that can be found in~\cite{DH}. We sketch the idea.
The first estimate in Theorem~\ref{theorem Littlewood Paley}
follows from Theorem~\ref{theorem P P inequality one-parameter}
together with the following observation:
\begin{align*}
    \sum_{k'}\sum_{\alpha'\in \mathscr{X}^{k'+N}}
        \inf_{z\in Q^{{k'+N }}_{\alpha'}}
        \big|D_{k'}(f)(z)\big|^2
        \chi_{Q^{k'+N }_{\alpha'}}(x)
    &\leq \sum_{k}|D_k(f)(x)|^2 \\
    &\leq \sum_{k'}\sum_{\alpha'\in \mathscr{X}^{k'+N}}
        \sup_{z\in Q^{{k'+N }}_{\alpha'}}
        \big|D_{k'}(f)(z)\big|^2
        \chi_{Q^{{k'+N } }_{\alpha'}}(x).
\end{align*}
For the second estimate in Theorem~\ref{theorem Littlewood
Paley}, when $1 < p < \infty$ one obtains from the classical
method of vector-valued Calder\'on--Zygmund operator theory
that $\|S_c(f)\|_{L^p(X)} \leq C\|f\|_{L^p(X)}$. This estimate
together with the wavelet expansion as
in~\eqref{eqn:AH_reproducing formula} gives $\|f\|_{L^p(X)}
\leq C\|S_c(f)\|_{L^p(X)}$, and Theorem~\ref{theorem Littlewood
Paley} follows.
\end{proof}

We would like to point out that to consider $S_c(f)$ as a
vector-valued Calder\'on--Zygmund operator, we need to use
the crucial estimate mentioned in
Remark~\ref{remark:crucial_estimate} to show that the
kernel of the operator $S_c(f)$ satisfies all conditions
for the Calder\'on--Zygmund singular integral operator. We
omit the details.

\smallskip
\noindent{\textbf{Outline of proof of Theorem~\ref{theorem P P
inequality one-parameter}.}} Since the proof (below) is rather
complex, we begin by outlining our approach. For the first
Plancherel--P\'olya inequality~\eqref{eqn:PP1}, we substitute
the wavelet reproducing formula~\eqref{eqn:reproducing_formula} for~$f$
into the left-hand side. Thus the desired wavelet coefficients
$\langle \psi^k_\alpha, f\rangle$ appear. To deal with the
unwanted terms $D_{k'}$ and $\psi^k_\alpha$, we apply the
almost-orthogonality
estimates~\eqref{eqn:almost_orthogonality_estimate} given
below. Then the standard technique, as in~\cite{DH}, of
applying an estimate from~\cite{FJ} and the Fefferman--Stein
vector-valued maximal function inequality~\cite{FS}
establishes~\eqref{eqn:PP1}.

The second Plancherel--P\'olya inequality~\eqref{eqn:PP2} is
harder. Roughly speaking, we need to control the wavelet
coefficients by the quantities $D_{k'}(f)$. Now for spaces of
homogeneous type with additional assumptions, one proceeds as
in~\cite{DH} via a frame reproducing formula of the form
\[
    f(x)
    = \sum_{k'} \sum_{\alpha'\in\mathscr{X}^{k' + N}}
    \mu(Q^{k' + N}_{\alpha'})
    \widetilde{D}_{k'}(x,x^{k'+N}_{\alpha'})
    D_k(f)(x^{k' + N}_{\alpha'}).
\]
However, for our spaces of homogeneous type with no additional
assumptions on~$d$ and~$\mu$, no such frame reproducing formula
is available. A new idea is needed. We introduce a suitable
operator~$T_N$, show that $T_N$ is bounded and that the
$L^p(X)$ norm of $S(T_N^{-1}(f))$ is controlled by that of
$S(f)$ (Lemma~\ref{lem:TN_bounded} below), and rewrite the
wavelet coefficient as $\langle \psi^k_\alpha, T_N^{-1} T_N
f\rangle$. Pulling out the operator $T_N^{-1}$ from the
left-hand side of~\eqref{eqn:PP2}, we obtain expressions of the
form $\langle \psi^k_\alpha/\sqrt{\mu(Q^k_\alpha)}, T_N
f\rangle$. Because of the form of~$T_N$, we can now apply the
almost-orthogonality
estimates~\eqref{eqn:almost_orthogonality_estimate} to these
terms and complete the remainder of the proof of the second
Plancherel--Polya inequality~\eqref{eqn:PP2} by following the
approach used for~\eqref{eqn:PP1}.

We now give the details.


\begin{proof}[Proof of Theorem~\ref{theorem P P inequality one-parameter}]
For each $f\in (\GGs(\beta,\gamma))'$, by Theorem~\ref{thm
reproducing formula test function}, the functions
\[
    f_n(x)
    = \sum_{|k|\leq n}\sum_{\alpha \in \mathscr{Y}^k}
        \langle f, \psi^k_\alpha\rangle \psi^k_\alpha(x)
\]
belong to~$L^2(X)$ and converge to $f$ in
$(\GGs(\beta,\gamma))'$ as $n\to\infty$. Note that $\langle
f_n, \psi^k_\alpha\rangle = \langle f, \psi^k_\alpha\rangle $
for $|k|\leq n$, and $\langle f_n, \psi^k_\alpha\rangle = 0$
for $|k| > n$. Thus,
\begin{align*}
    \lefteqn{\sum_{|k'|\leq n}\sum_{\alpha'\in \mathscr{X}^{k'+N}}
            \Big[\sup_{z\in {Q}^{{k^\prime +N }}_{\alpha ^\prime }}
            \big|D_{k^\prime}(f)(z)\big|^2\Big]
            \chi_{{ Q}^{{k^\prime +N }}_{\alpha^\prime }}(x)} \hspace{1cm} \\
    &= \sum_{k'}\sum_{\alpha'\in \mathscr{X}^{k'+N}}
            \Big[\sup_{z\in {Q}^{{k^\prime +N }}_{\alpha ^\prime }}
            \big|D_{k^\prime}(f_n)(z)\big|^2\Big]
            \chi_{{ Q}^{{k^\prime +N }}_{\alpha^\prime }}(x)
\end{align*}
and
\[
    \sum_{|k|\leq n}\sum_{\alpha\in\mathscr{Y}^k}
            \big|\langle \psi_{\alpha}^{k},f \rangle
            \widetilde{\chi}_{{Q}_{\alpha}^{k}}(\cdot)\big|^2
    = \sum_{k}\sum_{\alpha\in\mathscr{Y}^k}
            \big|\langle \psi_{\alpha}^{k},f_n \rangle
            \widetilde{\chi}_{{Q}_{\alpha}^{k}}(\cdot)\big|^2.
\]
Therefore it suffices to show the inequality~\eqref{eqn:PP1} of
Theorem~\ref{theorem P P inequality one-parameter} for $f\in
L^2(X)$, and similarly for the inequality~\eqref{eqn:PP2}.

We first prove~\eqref{eqn:PP1}. Fix $N\in \N$. The idea is to
apply an almost-orthogonality estimate
(\eqref{eqn:almost_orthogonality_estimate} below). First, for
each $f\in L^2(X)$, by the wavelet expansion (Theorem~\ref{thm
reproducing formula test function}),
\begin{eqnarray*}
   f(x)
   = \sum_{k\in\mathbb{Z}}\sum_{\alpha \in \mathscr{Y}^k}
    \langle f,\psi_{\alpha}^k \rangle \psi_{\alpha}^k(x).
\end{eqnarray*}
Thus for each $z\in {Q}^{{k^\prime + N }}_{\alpha^\prime } $ we
have
\begin{eqnarray*}
    D_{k^\prime}(f)(z)
    = \sum_{k\in\mathbb{Z}}\sum_{\alpha \in \mathscr{Y}^k}
        \mu(Q_\alpha^k)
        \Big\langle f,{\psi_{\alpha}^k\over \sqrt{\mu(Q_\alpha^k)}} \Big\rangle
        \Big \langle {\psi_{\alpha}^k(\cdot)\over \sqrt{\mu(Q_\alpha^k)}},
        D_{k^\prime}(\cdot,z)\Big\rangle.
\end{eqnarray*}


\smallskip
\noindent {\bf Claim:} (\emph{Almost-orthogonality estimate})
We claim that $\big\langle \psi_{\alpha}^{k}(\cdot) /
\sqrt{\mu({ Q}_{\alpha}^{k})}, D_{k'}(\cdot,z) \big\rangle$
satisfies the following almost-orthogonality estimate: There
exists a constant~$C$ such that for each positive integer $N$,
each $\gamma > 0$, each point $z\in
Q^{k^\prime+N}_{\alpha^\prime}$ and each point
$x^{k^\prime+N}_{\alpha^\prime}\in
Q^{k^\prime+N}_{\alpha^\prime}$, we have
\begin{eqnarray}\label{eqn:almost_orthogonality_estimate}
    &&\Big|\Big\langle
        \frac{\psi_{\alpha}^{k}(\cdot)}{\sqrt{\mu({ Q}_{\alpha}^{k})}},
        D_{k'}(\cdot,z)
        \Big\rangle\Big| \nonumber \\
    &&\hskip.5cm \leq C \delta^{\vert k - {k^\prime }\vert{\eta }}
        \frac{1}{V_{\delta^{(k'\wedge k)}}(x_\alpha^k)
            + V_{\delta^{(k'\wedge k)}}(x^{{k^ \prime +N }}_{\alpha ^\prime })
            + V(x_\alpha^k,x^{{k^ \prime +N }}_{\alpha ^\prime })}
        \Big({{\delta^{(k\wedge {k^\prime })}}
            \over {\delta^{(k\wedge {k^\prime })}
            + d(x_\alpha^k,x^{{k^ \prime+N}}_{\alpha^\prime })}}
        \Big)^{\gamma }.
\end{eqnarray}
As usual, $k\wedge k' = \min\{k,k^\prime\}$ denotes the minimum
of $k$ and~$k'$.

\begin{remark}\label{rem:wavelet_test_function_conditions}
The key idea used below to prove the
claim~\eqref{eqn:almost_orthogonality_estimate} is that both
$\psi_{\alpha}^{k}(x) / \sqrt{\mu({ Q}_{\alpha}^{k})}$ and
$D_{k'}(\cdot,z)$ satisfy size conditions, H\"older regularity
conditions, and cancellation, since as we have shown,
$\psi_{\alpha}^{k}(x) / \sqrt{\mu({ Q}_{\alpha}^{k})}$ is a
test function in~$\GGs(\beta,\gamma)$ while $D_{k'}(\cdot,z)$
satisfies the properties~\eqref{e1 in lemma another version of
Lemma 9.1}--\eqref{e3 in lemma another version of Lemma 9.1} in
Lemma~\ref{lemma another version of Lemma 9.1}. Further, we
point out that if $D_k(x,y)$ satisfies the same size
condition~\eqref{e1 in lemma another version of Lemma 9.1}
together with the following H\"older regularity condition
(which is weaker than~\eqref{e2 in lemma another version of
Lemma 9.1}),
\begin{eqnarray}\label{eqn:weak_Holder}
    &&\big |D_k(x,y) - D_k(x,y')\big|
    \leq C \Big({d(y,y')\over\delta^k }\Big)^\eta
        \Big[{1\over V_{\delta^k}(x) + V(x,y)}
        \Big({\delta^k\over \delta^{k} + d(x,y)}\Big)^{\gamma} \nonumber \\
    &&\hspace{5cm} {}+ {1\over V_{\delta^k}(x) + V(x,y')}
        \Big({\delta^k\over \delta^{k} + d(x,y')}\Big)^{\gamma}\Big],
\end{eqnarray}
and if the above estimate holds with $x$ and $y$ interchanged,
then the above almost-orthogonality
estimate~\eqref{eqn:almost_orthogonality_estimate} still holds,
but with $\eta$ replaced by some $\eta^\prime \in (0, \eta)$.
\end{remark}

Assuming the claim for the moment, we obtain that
\begin{eqnarray*}
    \lefteqn{\sup_{z\in {Q}^{{k^\prime + N }}_{\alpha ^\prime }}
        \big|D_{k^\prime}(f)(z)\big|} \\
    &&\hskip.5cm
    \leq C \sum_{k\in\mathbb{Z}}\sum_{\alpha \in \mathscr{Y}^k}\mu(Q_\alpha^k)
        \Big|\Big\langle f,{\psi_{\alpha}^k\over \sqrt{\mu(Q_\alpha^k)}} \Big\rangle\Big|
        \delta^{\vert k - {k^\prime }\vert {\eta }}\\
    &&\hskip1cm {} \times \frac{1}{V_{\delta^{(k'\wedge k)}}(x_\alpha^k)
        + V_{\delta^{(k'\wedge k)}}(x^{{k^ \prime + N }}_{\alpha ^\prime })
        + V(x_\alpha^k,x^{{k^ \prime + N }}_{\alpha ^\prime })}
        \Big({{\delta^{(k\wedge {k^\prime })}}\over{\delta^{(k\wedge {k^\prime })}
        + d(x_\alpha^k,x^{{k^ \prime + N}}_{\alpha^\prime })}}\Big)^{\gamma }.
\end{eqnarray*}

As a consequence, we have
\begin{eqnarray*}
    &&\Big\{\sum_{k'}\sum_{\alpha'\in \mathscr{X}^{k'+N}}
        \sup_{z\in {Q}^{{k^\prime +N }}_{\alpha ^\prime }}
        \big|D_{k^\prime}(f)(z)\big|^2
        \chi_{Q^{k^\prime +N }_{\alpha^\prime }}(x)\Big\}^{1/2}\\
    &&\hskip.5cm \leq C \bigg\{\sum_{k'}\sum_{\alpha'\in \mathscr{X}^{k'+N}}
        \bigg|\sum_{k\in\mathbb{Z}}\sum_{\alpha \in \mathscr{Y}^k}\mu(Q_\alpha^k)
        \Big|\Big\langle f,{\psi_{\alpha}^k\over \sqrt{\mu(Q_\alpha^k)}} \Big\rangle\Big|
        \delta^{\vert k - {k^\prime }\vert{\eta }} \\
    &&\hskip1.7cm {}\times \frac{1}{V_{\delta^{(k'\wedge k)}}(x_\alpha^k)
        + V_{\delta^{(k'\wedge k)}}(x)
        + V(x_\alpha^k,x)}
        \Big({{\delta^{(k\wedge {k^\prime })}}\over{\delta^{(k\wedge {k^\prime })}
        + d(x_\alpha^k,x)}}\Big)^{\gamma } \bigg|^2
        \chi_{ Q^{{k^\prime +N }}_{\alpha^\prime }}(x) \bigg\}^{1/2}.
\end{eqnarray*}
Using the same estimate as in \cite{FJ}, pp.147--148 (see also
Lemma~2.12 in~\cite{HLL2}), we obtain
\begin{eqnarray*}
    \lefteqn{\sum_{\alpha \in \mathscr{Y}^k}\mu(Q_\alpha^k)
        \frac{1}{V_{\delta^{(k'\wedge k)}}(x_\alpha^k)
        + V_{\delta^{(k'\wedge k)}}(x)+V(x_\alpha^k,x)}
        \Big({{\delta^{(k\wedge {k^\prime })}}\over{\delta^{(k\wedge {k^\prime })}
        + d(x_\alpha^k,x)}}\Big)^{\gamma }
        \Big|\Big\langle f,{\psi_{\alpha}^k\over \sqrt{\mu(Q_\alpha^k)}} \Big\rangle\Big|}\hspace{0.5cm}\\
    &&\hskip.4cm\leq C \delta^{[(k\wedge k')-k]\omega(1-1/r)}
        \bigg\{\mathcal{M}\bigg( \sum \limits_{{ \alpha }\in {\mathscr{Y}^{k }}}
        \Big|\Big\langle f,{\psi_{\alpha}^k\over \sqrt{\mu(Q_\alpha^k)}} \Big\rangle\Big|^r
        \chi_{{Q}^{{k}}_{\alpha }}(\cdot) \bigg)(x)\bigg\}^{1/r},
\end{eqnarray*}
where $\mathcal{M}$ is the Hardy--Littlewood maximal function
on $X$ and $\frac{\omega}{\omega +\eta} < r < p$.

Thus, by the Fefferman--Stein vector-valued maximal function
inequality with $p/r>1$ (see~\cite{FS}), we obtain
\begin{eqnarray*}
    &&\Big\|\Big\{ \sum_{k'}\sum_{\alpha'\in \mathscr{X}^{k'+N}}
        \sup_{z\in {Q}^{{k^ \prime +N }}_{\alpha ^\prime}}
        \big|D_{k^\prime }(f)(z)\big|
        \chi_{Q^{{k^\prime + N }}_{\alpha^\prime }}(\cdot)
        \Big\}^{1/2}\Big\|_{L^p(X)} \nonumber\\
    &&\hskip.5cm \leq C \Big\|\Big\{
        \sum_{k}\sum_{\alpha\in\mathscr{Y}^k}
        \big| \langle \psi_{\alpha}^{k},f \rangle
        \widetilde{\chi}_{{ Q}_{\alpha}^{k}}(\cdot)
        \big|^2 \Big\}^{1/2} \Big\|_{L^p(X)}.
\end{eqnarray*}

It remains to show the claimed almost-orthogonality
estimate~\eqref{eqn:almost_orthogonality_estimate}. We first
consider the case $k\geq k'$. Applying the cancellation
property for $\psi_{\alpha}^{k}(x)$ yields
\begin{eqnarray*}
    \Big|\Big\langle {\psi_{\alpha}^{k}(\cdot)\over
        \sqrt{\mu({ Q}_{\alpha}^{k})}}, D_{k'}(\cdot,z) \Big\rangle\Big|
    &=& \Big|\int_X {\psi_{\alpha}^{k}(x)\over \sqrt{\mu({ Q}_{\alpha}^{k})}}
        \Big[D_{k'}(x,z) - D_{k'}(x_{\alpha}^{k},z) \Big]\, d\mu(x)\Big|\\
    &\leq& \int_{W_1} {|\psi_{\alpha}^{k}(x)|\over \sqrt{\mu({ Q}_{\alpha}^{k})}}\big|
        D_{k'}(x,z) - D_{k'}(x_{\alpha}^{k},z) \big|\, d\mu(x)\\
    && {}+ \int_{W_2} {|\psi_{\alpha}^{k}(x)|\over \sqrt{\mu({ Q}_{\alpha}^{k})}}
        \Big[\big| D_{k'}(x,z)\big| + \big|D_{k'}(x_{\alpha}^{k},z) \big|\Big]\, d\mu(x)\\
    &=:& \textup{U}_1 + \textup{U}_2,
\end{eqnarray*}
where $W_1 := \{x\in X: d(x,x_{\alpha}^{k}) \leq (2A_0)^{-1}
(\delta^{k'} + d(x_{\alpha}^{k},z))\}$ and $W_2 := X \setminus
W_1$.

Similarly to the estimate of $\textup{(A)}_1$ in the proof of
Lemma~\ref{lemma another version of Lemma 9.1}, for
$\textup{U}_1$, using the size condition
(Definition~\ref{def-of-test-func-space}(i)) on
$\psi_{\alpha}^{k}(x)/\sqrt{\mu({ Q}_{\alpha}^{k})}$ and the
smoothness condition (Lemma~\ref{lemma another version of Lemma
9.1}(ii)) on $D_{k'}(x,y)$, we obtain that for all $z$,
$x_{\alpha'}^{k'+N}\in Q_{\alpha'}^{k'+N}$,
\begin{eqnarray*}
\textup{U}_1
&\leq& C\int_{W_1}  {1\over V_{\delta^k}(x_\alpha^k) + V(x_\alpha^k,x)}
    \Big({\delta^k\over \delta^{k} +
    d(x_\alpha^k,x)}\Big)^{\Gamma}\\
&&{} \times \Big({d(x,x_\alpha^k)\over\delta^{k'} + d(x_\alpha^k,x_{\alpha'}^{k'+N})}\Big)^\eta
    {1\over V_{\delta^{k'}}(x_\alpha^k) + V(x,x_{\alpha'}^{k'+N})}
    \Big({\delta^{k'}\over \delta^{k'} +
    d(x_\alpha^k,x_{\alpha'}^{k'+N})}\Big)^{\gamma}
\,d\mu(x)\\
&\leq& C \delta^{( k - k' ) {\eta }} \int_{W_1}
    {1\over V_{\delta^k}(x_\alpha^k) + V(x_\alpha^k,x)}
    \Big({\delta^k\over \delta^{k} +
    d(x_\alpha^k,x)}\Big)^{\Gamma-\eta} \,d\mu(x) \\
&&{}\times
    \frac{1}{V_{\delta^{k'}}(x_\alpha^k) + V_{\delta^{k'}}(x^{{k^ \prime +N }}_{\alpha ^\prime })
    + V(x_\alpha^k,x^{{k^ \prime +N }}_{\alpha ^\prime })}
    \Big({{\delta^{k'}}\over{\delta^{k'} + d(x_\alpha^k,x^{{k^ \prime +N }}_{\alpha ^\prime })}}\Big)^{\gamma }
\end{eqnarray*}
for $\Gamma>\eta$ and $\gamma>0$.

The estimate for $\textup{U}_2$ is similar to the proof for
$\textup{(A)}_2$ as in Lemma~\ref{lemma another version of
Lemma 9.1}. Specifically, we have
\begin{align*}
\textup{U}_2
&\leq C\int_{W_2}  {1\over V_{\delta^k}(x_\alpha^k)+V(x_\alpha^k,x)}
    \Big({\delta^k\over \delta^{k}
    + d(x_\alpha^k,x)}\Big)^{\Gamma}\Big[{1\over V_{\delta^{k'}}(x)+V(x,x^{{k^ \prime +N }}_{\alpha ^\prime })}
    \Big({\delta^{k'}\over \delta^{k'}
    + d(x,x^{{k^ \prime +N }}_{\alpha ^\prime })}\Big)^{\gamma}\\
& \hspace{1cm} {}+ {1\over V_{\delta^{k'}}(x_\alpha^k)+V(x_\alpha^k,x^{{k^ \prime +N }}_{\alpha ^\prime })}
    \Big({\delta^{k'}\over \delta^{k'}
    + d(x_\alpha^k,x^{{k^ \prime +N }}_{\alpha ^\prime })}\Big)^{\gamma}\Big]
    \,d\mu(x)\\
&\leq C \delta^{( k - k' ) {\eta }}
    \frac{1}{V_{\delta^{k'}}(x_\alpha^k)
        + V_{\delta^{k'}}(x^{{k^ \prime +N }}_{\alpha ^\prime })
        + V(x_\alpha^k,x^{{k^ \prime +N }}_{\alpha ^\prime })}
    \Big({{\delta^{k'}}\over{\delta^{k'} + d(x_\alpha^k,x^{{k^ \prime +N }}_{\alpha ^\prime })}}
    \Big)^{\gamma }.
\end{align*}
These estimates of $\textup{U}_1$ and $\textup{U}_2$ establish
the claimed almost-orthogonality
estimate~\eqref{eqn:almost_orthogonality_estimate} when $k\geq
k'$. The proof for the case $k < k'$ is similar. This completes
the proof of the almost-orthogonality
estimate~\eqref{eqn:almost_orthogonality_estimate}, and hence
the proof of the first Plancherel--P\'olya
inequality~\eqref{eqn:PP1}.

To show the second Plancherel--P\'olya
inequality~\eqref{eqn:PP2}, we need the following result about
the operator~$T_N$, as mentioned in the outline of the proof of
Theorem~\ref{theorem P P inequality one-parameter}.

\begin{lemma}\label{lem:TN_bounded}
    \textup{(Properties of~$T_N$)} Suppose that $f\in L^2(X)$
    and $\frac{\omega}{\omega + \eta} < p < \infty$, where
    $\omega$ is the upper dimension of~$(X,d,\mu)$. Let $N$ be
    a positive integer. In each cube $Q^{k + N}_\alpha$, fix a
    point $x^{k + N}_\alpha$. Define the operator $T_{N}$ by
    \begin{equation}\label{eqn:T_Nbounded}
        T_{N}(f)(x)
        := \sum_k\sum_{\alpha\in {\mathscr X}^{k+N}}\mu({ Q}^{k+N}_\alpha)
            D_k(x, x^{k+N}_\alpha)D_k(f)(x_\alpha^{k+N}).
    \end{equation}
    Then the following assertions hold.
    \begin{enumerate}
        \item[(i)] $T_N$ is bounded on $L^2(X)$.

        \item[(ii)] There exists a constant~$C$ independent
            of~$f$ and of the choice of~$x^{k+N}_\alpha$
            such that
            \[
                \|S(T_N(f))\|_{L^p(X)}
                \leq C \|S(f)\|_{L^p(X)},
            \]
            where $S$ is the discrete Littlewood--Paley
            square function as in
            Definition~\ref{def:discrete_square_function}.

        \item[(iii)] If $N$ is chosen sufficiently large,
            then $T_N$ is invertible and there is a
            constant~$C$ independent of $f$ and of the
            choice of $x^{k+N}_\alpha$ such that
            \begin{equation}\label{eqn:TNinversebounded}
                \|S(T_N^{-1}(f))\|_{L^p(X)}
                \leq C \|S(f)\|_{L^p(X)}.
            \end{equation}
        \end{enumerate}
\end{lemma}

We defer the proof of this technical lemma until after the end
of the proof of Theorem~\ref{theorem P P inequality
one-parameter}. We now continue the proof of the second
Plancherel--P\'olya inequality~\eqref{eqn:PP2}. Choose $N$
sufficiently large that $T_N$ is invertible and
\eqref{eqn:TNinversebounded} holds. For $f\in L^2 (X)$, write
$f=T_N^{-1} T_N f$. Applying Lemma~\ref{lem:TN_bounded}, we
find that
\begin{eqnarray}\label{TN-1 TN}
    \Big\|\Big\{ \sum_{k}\sum_{\alpha\in\mathscr{Y}^k} \big| \langle
    \psi_{\alpha}^{k},f \rangle \widetilde{\chi}_{{ Q}_{\alpha}^{k}}(\cdot)
    \big|^2 \Big\}^{1/2}
      \Big\|_{L^p(X)} 
    &=& \Big\|\Big\{ \sum_{k}\sum_{\alpha\in\mathscr{Y}^k} \big| \langle
    \psi_{\alpha}^{k}, T_N^{-1} T_Nf \rangle \widetilde{\chi}_{{ Q}_{\alpha}^{k}}(\cdot)
    \big|^2 \Big\}^{1/2}
      \Big\|_{L^p(X)}\nonumber\\
    &\leq& C\Big\|\Big\{ \sum_{k}\sum_{\alpha\in\mathscr{Y}^k} \big| \langle
    \psi_{\alpha}^{k}, T_Nf \rangle \widetilde{\chi}_{{ Q}_{\alpha}^{k}}(\cdot)
    \big|^2 \Big\}^{1/2}
      \Big\|_{L^p(X)}.
\end{eqnarray}

By the definition of $T_N(f)$, we have
\begin{eqnarray*}
    \Big\langle {\psi_{\alpha}^{k}\over \sqrt{\mu({ Q}_{\alpha}^{k})}},
        T_Nf \Big\rangle
    &=& \sum_{k'}\sum_{\alpha'\in{\mathscr X}^{k'+N}}
        \mu({ Q}^{k'+N}_{\alpha'})
        \Big\langle {\psi_{\alpha}^{k}(\cdot)\over {\mu({ Q}_{\alpha}^{k})}},
        D_{k'}(\cdot, x^{k'+N}_{\alpha'})\Big\rangle
        D_{k'}(f)(x_{\alpha'}^{k'+N}).
\end{eqnarray*}

Therefore, for each fixed $\eta' \in (0,\eta),$
\begin{eqnarray*}
    &&\Big\{ \sum_{k}\sum_{\alpha\in\mathscr{Y}^k} \big| \langle
        \psi_{\alpha}^{k}, T_Nf \rangle \widetilde{\chi}_{{ Q}_{\alpha}^{k}}(x)
        \big|^2 \Big\}^{1/2}\\
    &&\leq C\Big\{ \sum_{k}\sum_{\alpha\in\mathscr{Y}^k} \Big[ \sum
        \limits_{{k^\prime }\in \mathbb{Z}}\sum \limits_{{ \alpha^\prime }\in
        {\mathscr{X}^{k^\prime +N }}}\mu
        ({Q}^{{k^\prime +N }}_{\alpha^\prime })
        \delta^{\vert k - {k^\prime }\vert {\eta'
        }}\frac{1}{V_{\delta^{k'}}(x_\alpha^k)+V_{\delta^{k'}}(x^{{k^ \prime
        +N}}_{\alpha ^\prime })+V(x_\alpha^k,x^{{k^ \prime
        +N}}_{\alpha ^\prime })}\\
    &&\hskip1.5cm {}\times
        \Big({{\delta^{(k\wedge
        {k^\prime })} }\over{\delta^{(k\wedge {k^\prime })}+
        d(x_\alpha^k,x^{{k^ \prime +N }}_{\alpha^\prime
        })}}\Big)^{\gamma } D_{k^\prime }(f)(x^{{k^ \prime +N }
        }_{\alpha ^\prime })  \chi_{Q_{\alpha}^{k}}(x) \Big]^2
        \Big\}^{1/2}\\
    &&\leq C\Big\{ \sum_{k}\sum_{\alpha\in\mathscr{Y}^k} \Big[ \sum
        \limits_{{k^\prime }\in\mathbb{Z}}\delta^{\vert k - {k^\prime }\vert
        {\eta' }} \sum \limits_{{ \alpha^\prime }\in
        {\mathscr{X}^{k^\prime +N }}}\mu
        ({Q}^{{k^\prime +N }}_{\alpha^\prime })
        \frac{1}{V_{\delta^{k'}}(x)+V_{\delta^{k'}}(x^{{k^ \prime +N }
        }_{\alpha ^\prime })+V(x,x^{{k^ \prime
        +N}}_{\alpha ^\prime })}\\
    &&\hskip1.5cm {}\times
        \Big({{\delta^{(k\wedge
        {k^\prime })} }\over{\delta^{(k\wedge {k^\prime })}+ d(x,x^{{k^
        \prime +N }}_{\alpha^\prime })}}\Big)^{\gamma }
        D_{k^\prime }(f)(x^{{k^ \prime +N }}_{\alpha ^\prime
        })\chi_{{ Q}_{\alpha}^{k}}(x) \Big]^2 \Big\}^{1/2}.
\end{eqnarray*}
By the same estimate from~\cite{FJ} as in the proof
of~\eqref{eqn:PP1} above, we have
\begin{eqnarray*}
    &&\sum_{{ \alpha^\prime }\in {\mathscr{X}^{k^\prime +N }}}
        \mu ({Q}^{{k^\prime +N }}_{\alpha^\prime })
        \frac{1}{V_{\delta^{k'}}(x)
            + V_{\delta^{k'}}(x^{{k^ \prime +N }}_{\alpha ^\prime })
            + V(x,x^{{k^ \prime +N }}_{\alpha ^\prime })} \\
    && \hspace{1cm} {}\times\Big({{\delta^{(k\wedge {k^\prime })} }
            \over {\delta^{(k\wedge {k^\prime })}
            + d(x,x^{{k^\prime +N }}_{\alpha^\prime })}}\Big)^{\gamma }
        D_{k^\prime }(f)(x^{{k^ \prime +N }}_{\alpha ^\prime })\\
    &&\hskip.4cm\leq C \delta^{[(k\wedge k')-k]\omega(1-1/r)}
        \Big\{\mathcal{M}\Big(
        \sum_{{ \alpha^\prime }\in {\mathscr{X}^{k^\prime +N }}}
        \big|D_{k^\prime }(f)(x^{{k^ \prime +N }}_{\alpha ^\prime })\big|^r
        \chi_{{Q}^{{k^\prime +N }}_{\alpha^\prime }}(\cdot)
        \Big)(x)\Big\}^{1/r},
\end{eqnarray*}
where $\mathcal{M}$ is the Hardy--Littlewood maximal function
on $X$ and $\frac{\omega}{\omega + \eta} < r < p$. Note that
the above inequality still holds when the point $x^{{k^ \prime
+N } }_{\alpha ^\prime }$ on the right-hand side is replaced by
an arbitrary point $z$ in ${Q}^{{k^ \prime +N } }_{\alpha
^\prime }$, and therefore also holds when the expression
$|D_{k^\prime }(f)(x^{{k^ \prime +N }}_{\alpha ^\prime })|^r$
on the right-hand side is replaced by the infimum of
$|D_{k^\prime }(z)|^r$ over all $z \in Q^{k' + N}_{\alpha'}$.
Thus, we have
\begin{eqnarray*}
    \lefteqn{\Big\{ \sum_{k}\sum_{\alpha\in\mathscr{Y}^k} \big| \langle
        \psi_{\alpha}^{k},f \rangle \widetilde{\chi}_{{ Q}_{\alpha}^{k}}(x)
        \big|^2 \Big\}^{1/2}}\hspace{1cm}\\
    &&\leq C\bigg\{ \sum_{k}\sum_{\alpha\in\mathscr{Y}^k} \bigg[ \sum
        \limits_{{k^\prime }\in\mathbb{Z}}\delta^{\vert k - {k^\prime }\vert
        {\epsilon }}\delta^{[(k\wedge k')-k]\omega(1-1/r)}\\
    &&\hskip1cm {}\times \Big\{ \mathcal{M}\Big( \sum \limits_{{
        \alpha^\prime }\in {\mathscr{X}^{k^\prime +N }}} \inf_{z\in
        {Q}^{{k^ \prime +N } }_{\alpha'}}\big|D_{k^\prime }(f)(z)\big|^r
        \chi_{{Q}^{{k^\prime +N } }_{\alpha^\prime }}(\cdot)
        \Big)(x)\Big\}^{1/r} \bigg]^2\chi_{{ Q}_{\alpha}^{k}}(x) \bigg\}^{1/2}.
\end{eqnarray*}
Applying the Fefferman--Stein vector-valued maximal function
inequality with $p/r > 1$, from~\cite{FS}, we obtain
\begin{eqnarray*}
    &&\Big\|\Big\{ \sum_{k}\sum_{\alpha\in\mathscr{Y}^k} \big| \langle
        \psi_{\alpha}^{k},f \rangle \widetilde{\chi}_{{ Q}_{\alpha}^{k}}(x)
        \big|^2 \Big\}^{1/2}
          \Big\|_{L^p(X)}\leq C\Big\|\Big\{ \sum_{k}\sum_{\alpha\in\mathscr{Y}^k} \big| \langle
        \psi_{\alpha}^{k}, T_Nf \rangle \widetilde{\chi}_{{ Q}_{\alpha}^{k}}(x)
        \big|^2 \Big\}^{1/2}
          \Big\|_{L^p(X)}  \nonumber\\
    &&\hskip.5cm \leq C \Big\|\Big\{  \sum \limits_{{k^\prime }\in\mathbb{Z}}
        \Big[\mathcal{M}\Big( \sum \limits_{{ \alpha^\prime }\in
        {\mathscr{X}^{k^\prime +N }}} \inf_{z\in {Q}^{{k^ \prime +N
        }}_{\alpha ^\prime }}\big|D_{k^\prime }(f)(z)\big|^r
        \chi_{Q^{{k^\prime +N } }_{\alpha^\prime }}(\cdot)
        \Big)(x) \Big]^{2/r} \Big\}^{1/2}\Big\|_{L^p(X)}\\
    &&\hskip.5cm\leq C \Big\|\Big\{\sum_{k'}\sum_{\alpha'\in
        \mathscr{X}^{k'+N}}
             \inf_{z\in {Q}^{{k^
        \prime +N }}_{\alpha ^\prime }}\big|D_{k^\prime }(f)(z)\big|
        \chi_{{Q}^{{k^\prime +N } }_{\alpha^\prime }}(x)
            \Big\}^{1/2}\Big\|_{L^p(X)},
\end{eqnarray*}
which implies that the second Plancherel--P\'olya
inequality~\eqref{eqn:PP2} holds for $f\in L^2(X)$. The proof
of Theorem~\ref{theorem P P inequality one-parameter} is
complete, except for the proof of Lemma~\ref{lem:TN_bounded}.
\end{proof}

It remains to prove the technical lemma used in the preceding
proof.

\begin{proof}[Proof of Lemma~\ref{lem:TN_bounded}]
\noindent (i) Fix $N\in \N$. We show that the operator~$T_N$ is
bounded on~$L^2(X)$. Write $T_N(f)(x) = \sum_k E_k(f)(x),$
where the kernel $E_k(x,y)$ of $E_k$ is given by
\[
    E_k(x,y)
    := \sum_{\alpha\in {\mathscr X}^{k+N}} \mu({ Q}^{k+N}_\alpha)
        D_k(x,x^{k+N}_\alpha) D_k(x^{k+N}_\alpha,y).
\]
This kernel $E_k(x,y)$ satisfies the same decay and smoothness
estimates \eqref{e1 in lemma another version of Lemma 9.1}
and~\eqref{e2 in lemma another version of Lemma 9.1} as
$D_k(x,y)$ does, with bounds independent of $x^{k+N}_\alpha$,
as can be shown by a proof similar to that for $D_k(x,y)$.
Moreover, $\int_X E_k(x,y) \, d\mu(y) = 0$ for each $x\in X$
and $\int_X E_k(x,y) \, d\mu(x) = 0$ for each $y\in X$.
Therefore the Cotlar--Stein lemma can be applied to show that
$T_N$ is bounded on~$L^2(X)$.

\noindent (ii) Suppose that $f\in L^2(X)$ and
$\frac{\omega}{\omega + \eta} < p < \infty$. Then by the
definition of $D_k$ and the wavelet reproducing
formula~(\ref{eqn:AH_reproducing formula}), we have
\begin{align*}
    f(x)
    &=\sum_{k}D_kD_k(f)(x)
        = \sum_k\sum_{\alpha\in {\mathscr X}^{k+N}}\mu({ Q}^{k+N}_\alpha)
        D_k(x, x^{k+N}_\alpha)D_k(f)(x_\alpha^{k+N})\\
    & \hspace{1cm} {}+ \Big(\sum_{k}D_kD_k(f)(x)
        - \sum_k\sum_{\alpha\in {\mathscr X}^{k+N}}
        \mu({ Q}^{k+N}_\alpha) D_k(x, x^{k+N}_\alpha)
        D_k(f)(x_\alpha^{k+N})\Big)\\
    &=: T_{N}(f)(x) + R_{N}(f)(x),
\end{align*}
where $x^{k+N}_\alpha$ are arbitrary fixed points in
${Q}^{k+N}_\alpha$.

Since $T_N = I - R_N$ by definition, to show (ii) in
Lemma~\ref{lem:TN_bounded}, it suffices to show that
\begin{equation}\label{eqn:S(R_N)_Lp_estimate}
    \|S(R_N(f))\|_{L^p(X)}
    \leq C\delta^{\eta N}\|S(f)\|_{L^p(X)}.
\end{equation}
For then
\[
    \|S(T_N(f))\|_{L^p(X)}
    \leq \|S(f)\|_{L^p(X)} + \|S(R_N(f))\|_{L^p(X)}
    \leq (1 + C\delta^{\eta N}) \|S(f)\|_{L^p(X)},
\]
as required.

To establish~\eqref{eqn:S(R_N)_Lp_estimate}, we write
\[
    R_N(f)(x)
    = \sum_k\sum_{\alpha\in {\mathscr X}^{k+N}}
        \int_{{Q}^{k+N}_\alpha}
        [D_k(x,z)D_k(f)(z) - D_k(x,x^{k+N}_\alpha)
        D_k(f)(x_\alpha^{k+N})] \, d\mu(z).
\]
Thus the kernel $R_N(x,y)$ of $R_N$ is given by
\begin{eqnarray*}
    R_N(x,y)
    &:=& \sum_k\sum_{\alpha\in {\mathscr X}^{k+N}}
        \int_{{Q}^{k+N}_\alpha}
        [D_k(x,z)D_k(z,y) - D_k(x,x^{k+N}_\alpha)
        D_k(x_\alpha^{k+N},y)] \, d\mu(z) \\
    &=& \sum_k\sum_{\alpha\in {\mathscr X}^{k+N}}\int_{{Q}^{k+N}_\alpha}
        [D_k(x,z) - D_k(x,x^{k+N}_\alpha)]
        D_k(z,y) \, d\mu(z)\\
    && \hspace{1cm} {}+ \sum_k\sum_{\alpha\in {\mathscr X}^{k+N}}
        \int_{{Q}^{k+N}_\alpha}
        D_k(x, x^{k+N}_\alpha)
        [D_k(z,y) - D_k(x_\alpha^{k+N},y)] \, d\mu(z)\\
    &=:& R^{(1)}_N(x,y) + R^{(2)}_N(x,y).
\end{eqnarray*}

Note that by the same proof as for~$T_N$, both $R^{(1)}_N$ and
$R^{(2)}_N$ are bounded on~$L^2(X)$, and therefore the inner
products $\langle \psi_{\alpha}^{k},R^{(1)}_N(f) \rangle$ and
$\langle \psi_{\alpha}^{k},R^{(2)}_N(f) \rangle$ are well
defined. To estimate $\|S(R^{(1)}_N(f))\|_{L^p(X)}$, we write
\begin{equation*}
    \|S(R^{(1)}_N(f))\|_{L^p(X)}
    = \Big\|\Big\{ \sum_{k}\sum_{\alpha\in\mathscr{Y}^k}
        \big| \langle \psi_{\alpha}^{k},R^{(1)}_N(f) \rangle
        \widetilde{\chi}_{{Q}_{\alpha}^{k}}(\cdot)
        \big|^2 \Big\}^{1/2} \Big\|_{L^p(X)}.
\end{equation*}
By the $L^2(X)$-boundedness of $R^{(1)}_N$ and the wavelet
reproducing formula~(\ref{eqn:AH_reproducing formula}) for
$f\in L^2(X)$, we have
\begin{eqnarray*}
    &&\Big\langle {\psi_{\alpha}^{k}\over \sqrt{\mu({Q}_{\alpha}^{k})}},
        R^{(1)}_N(f) \Big\rangle \\
    &&= \int_X {\psi_{\alpha}^{k}(x)\over \sqrt{\mu({Q}_{\alpha}^{k})}}
        \sum_{k'}\sum_{\alpha'\in {\mathscr X}^{k'+N}}\int_{{Q}^{k'+N}_{\alpha'}}
        [D_{k'}(x,z)- D_{k'}(x,x^{k'+N}_{\alpha'})]  \\
    &&\hskip2cm {}\times \int_X D_{k'}(z,y)
        \sum_{k^{''}\in \Z}\sum_{\alpha^{''}\in {\mathscr Y}^{k^{''}}}
        \langle \psi_{\alpha^{''}}^{k^{''}}, f\rangle
        \psi_{\alpha^{''}}^{k^{''}}(y) \, d\mu(y)\, d\mu(z)\,   d\mu(x)\\
    &&= \sum_{k^{''}}\sum_{\alpha^{''} \in {\mathscr Y}^{k^{''}}}
        \mu({Q}_{\alpha^{''}}^{k^{''}})
        \Big\langle {\psi_{\alpha^{''}}^{k^{''}} \over \sqrt{\mu({Q}_{\alpha^{''}}^{k^{''}})} }, f \Big\rangle
        \sum_{k'} \int_X \int_X {\psi_{\alpha}^{k}(x)\over \sqrt{\mu({Q}_{\alpha}^{k})}}\
        \overline{D}_{k'}(x,y)\
        {\psi_{\alpha^{''}}^{k^{''}}(y) \over \sqrt{\mu({Q}_{\alpha^{''}}^{k^{''}})} }
        \, d\mu(y)\, d\mu(x)\\
    &&= \sum_{k^{''}}\sum_{\alpha^{''} \in {\mathscr Y}^{k^{''}}}
        \mu({Q}_{\alpha^{''}}^{k^{''}})
        \Big\langle {\psi_{\alpha^{''}}^{k^{''}} \over \sqrt{\mu({Q}_{\alpha^{''}}^{k^{''}})} }, f \Big\rangle
        \Big\langle {\psi_{\alpha}^{k}(\cdot) \over \sqrt{\mu({Q}_{\alpha}^{k})} },
        F_{k^{''}(\cdot,x_{\alpha^{''}}^{k^{''}})}\Big\rangle,
\end{eqnarray*}
where
\[
    \overline{D}_{k'}(x,y)
    := \sum_{\alpha'\in {\mathscr X}^{k'+N}}
        \int_{{ Q}^{k'+N}_{\alpha'}}
        \big[D_{k'}(x,z) - D_{k'}(x, x^{k'+N}_{\alpha'})\big]
        D_{k'}(z,y) \, d\mu(z)
\]
and
\[
    F_{k^{''}}(x, x_{\alpha^{''}}^{k^{''}})
    := \sum_{k'}\int_X \overline{D}_{k'}(x,y)\
    {\psi_{\alpha^{''}}^{k^{''}}(y) \over \sqrt{\mu({Q}_{\alpha^{''}}^{k^{''}})} }
    \, d\mu(y).
\]

We now show that $\Big\langle
\psi_{\alpha}^{k}(\cdot)/\sqrt{\mu({Q}_{\alpha}^{k})} ,
F_{k^{''}}(\cdot, x_{\alpha^{''}}^{k^{''}}) \Big\rangle$
satisfies an almost-orthogonality estimate similar
to~\eqref{eqn:almost_orthogonality_estimate}, by following the
philosophy of
Remark~\ref{rem:wavelet_test_function_conditions}. Recall that
$\psi_{\alpha}^{k}(\cdot) / \sqrt{\mu({Q}_{\alpha}^{k})}$ is a
test function. It remains to show that the function
$F_{k^{''}}(x,x_{\alpha^{''}}^{k^{''}})$ satisfies a size
condition, a H\"older regularity condition, and cancellation.

Next, it seems unlikely that $F_{k^{''}}(x,
x_{\alpha^{''}}^{k^{''}})$ satisfies the H\"older regularity
condition~\eqref{e2 in lemma another version of Lemma 9.1}.
However, as noted in
Remark~\ref{rem:wavelet_test_function_conditions}, it suffices
to establish the weaker H\"older regularity
condition~\eqref{eqn:weak_Holder}, which we now do. To begin,
we show that $F_{k^{''}}(x, x_{\alpha^{''}}^{k^{''}})$
satisfies
\begin{eqnarray*}
    &&\textup{(a)}'\ \ |F_{k^{''}}(x, x_{\alpha^{''}}^{k^{''}})|
        \leq C \delta^{\eta N}
            {1\over V_{\delta^{k^{''}}}(x) + V(x,x_{\alpha^{''}}^{k^{''}})}
            \Big({\delta^{k^{''}}\over \delta^{k^{''}}
            + d(x,x_{\alpha^{''}}^{k^{''}})}\Big)^{\gamma}, \hspace{4.2cm}
\end{eqnarray*}
for all $\gamma \in (0,\eta)$, and
\begin{eqnarray*}
    \lefteqn{\textup{(b)}'\ \
            |F_{k^{''}}(x, x_{\alpha^{''}}^{k^{''}}) - F_{k^{''}}(x',x_{\alpha^{''}}^{k^{''}})|} \hspace{3cm}\\
    &&\leq C\delta^{\eta N} \Big({d(x,x')\over \delta^{k'}}\Big)^{\eta'}
    {}\times
        \bigg[{1\over V_{\delta^{k^{''}}}(x) + V(x,x_{\alpha^{''}}^{k^{''}})}
        \Big({\delta^{k^{''}}\over \delta^{k^{''}} + d(x,x_{\alpha^{''}}^{k^{''}})}\Big)^{\gamma} \\
    && \hspace{1cm} {}+ {1\over V_{\delta^{k^{''}}}(x') + V(x',x_{\alpha^{''}}^{k^{''}})}
        \Big({\delta^{k^{''}}\over \delta^{k^{''}}
        + d(x',x_{\alpha^{''}}^{k^{''}})}\Big)^{\gamma}\bigg],
\end{eqnarray*}
for all $\eta' \in (0,\eta)$.

To prove $\textup{(a)}'$ and $\textup{(b)}'$, we first show
that $\overline{D}_{k'}(x,y)$ satisfies the same estimates
\eqref{e1 in lemma another version of Lemma 9.1} and~\eqref{e2
in lemma another version of Lemma 9.1} as~$D_k(x,y)$, but with
the constant $C$ replaced by $C\delta^{\eta N}$, that is,
\begin{eqnarray*}
    &&\textup{(a)}\ \ |\overline{D}_{k'}(x,y)|
        \leq C \delta^{\eta N}
        {1\over V_{\delta^{k'}}(x) + V(x,y)}
        \Big({\delta^{k'}\over \delta^{k'}
            + d(x,y)}\Big)^{\gamma};\\
    &&\textup{(b)}\ \ |\overline{D}_{k'}(x,y) - \overline{D}_{k'}(x',y)|
        \leq C\delta^{\eta N}
            \Big({d(x,x')\over \delta^{k'} + d(x,y)}\Big)^\eta
            {1\over V_{\delta^{k'}}(x) + V(x,y)}
            \Big({\delta^{k'}\over \delta^{k'} + d(x,y)}\Big)^{\gamma}
\end{eqnarray*}
for $d(x,x')\leq (2A_0)^{-1}\max\{\delta^{k'} + d(x,y),
\delta^{k'} + d(x',y)\}$, and
\begin{eqnarray*}
    &&\textup{(c)}\ \  |\overline{D}_{k'}(x,y) - \overline{D}_{k'}(x,y')|
        \leq C\delta^{\eta N}
            \Big({d(y,y')\over \delta^{k'} + d(x,y)}\Big)^\eta
            {1\over V_{\delta^{k'}}(x) + V(x,y)}
            \Big({\delta^{k'}\over \delta^{k'} + d(x,y)}\Big)^{\gamma}
\end{eqnarray*}
for $d(y,y')\leq (2A_0)^{-1} \max\{\delta^{k'} + d(x,y),
\delta^{k'} + d(x,y')\}$. Moreover,
\[
    \int_X \overline{D}_{k'}(x,y) \, d\mu(x) = 0
    \qquad\text{and}\qquad
    \int_X \overline{D}_{k'}(x,y) \, d\mu(y)
    = 0,
\]
for all $y\in X$ and all $x\in X$, respectively. Indeed, note
that $\big[D_{k'}(x,z) - D_{k'}(x,x^{k'+N}_{\alpha'})\big]$
satisfies the same estimates~\eqref{e1 in lemma another version
of Lemma 9.1} and~\eqref{e2 in lemma another version of Lemma
9.1} as $D_{k'}(x,z)$ does, but with the constant~$C$ replaced
by $C\delta^{\eta N}$. Therefore, the proofs for (a), (b)
and~(c) follow from a similar proof to that for
Lemma~\ref{lemma another version of Lemma 9.1}. As a
consequence, the almost-orthogonality
estimate~\eqref{eqn:almost_orthogonality_estimate} holds for
$\Big\langle \overline{D}_{k'}(x,\cdot),
\psi_{\alpha^{''}}^{k^{''}}(\cdot)/
\sqrt{\mu({Q}_{\alpha^{''}}^{k^{''}})} \, \Big\rangle$.  We
omit the details.

Now, to verify the estimate in $\textup{(a)}'$, applying this
almost-orthogonality estimate yields that
\begin{eqnarray*}
    |F_{k^{''}}(x, x_{\alpha^{''}}^{k^{''}})|
    &\leq&\sum_{k'}
        \Big|\Big\langle \overline{D}_{k'}(x,\cdot),
        \frac{\psi_{\alpha^{''}}^{k^{''}}(\cdot)}
        {\sqrt{\mu({Q}_{\alpha^{''}}^{k^{''}})} } \Big\rangle\Big|\\
    &\leq& C\delta^{\eta N}\sum_{k'}\delta^{\vert k' - {k^{''} }\vert
        {\eta}} \frac{1}{V_{\delta^{(k'\wedge k^{''})}}(x)+V(x,x^{{k^{''}}
        }_{\alpha ^{''}})} \Big({{\delta^{(k'\wedge {k^{''}})}
        }\over{\delta^{(k'\wedge {k^{''}})}+
        d(x,x^{{k^{''}}}_{\alpha^{''}})}}\Big)^{\gamma }\\
    &\leq& C\delta^{\eta N}\sum_{k'}\delta^{\vert k - {k^\prime }\vert
        {(\eta-\gamma)}} \frac{1}{V_{\delta^{k^{''}}}(x)+V(x,x^{{k^{''}}
        }_{\alpha ^{''}})} \Big({\delta^{k^{''}}\over \delta^{k^{''}}+
            d(x,x_{\alpha^{''}}^{k^{''}})}\Big)^{\gamma}\\
    &\leq& C \delta^{\eta N} {1\over
        V_{\delta^{k^{''}}}(x)+V(x,x_{\alpha^{''}}^{k^{''}})}
        \Big({\delta^{k^{''}}\over \delta^{k^{''}}+
        d(x,x_{\alpha^{''}}^{k^{''}})}\Big)^{\gamma},
\end{eqnarray*}
for all $\gamma \in (0,\eta)$.

Next we show the estimate in $\textup{(b)}'$. Note that
\begin{eqnarray*}
    &&\big|\overline{D}_{k'}(x,y)-\overline{D}_{k'}(x',y)\big|\leq
        C\delta^{\eta N} \Big( {d(x,x')\over \delta^{k'}}\Big)^{\eta}\\
    &&\hskip.5cm {}\times\Big[{1\over V_{\delta^{k'}}(x) + V(x,y)} \Big(
        {\delta^{k'}\over \delta^{k'}+d(x,y)}\Big)^\gamma +{1\over
        V_{\delta^{k'}}(x') + V(x,y)} \Big( {\delta^{k'}\over
        \delta^{k'}+d(x',y)}\Big)^\gamma\Big]
\end{eqnarray*}
and
$$\int_X \Big[\overline{D}_{k'}(x,y)-\overline{D}_{k'}(x',y)\Big]\, d\mu(y)=0.$$
Therefore, as pointed out in
Remark~\ref{rem:wavelet_test_function_conditions}, we obtain
for $k' > k^{''}$ that
\begin{eqnarray*}
    &&\hskip-.7cm\Big|E_{k^{''}}(x, x_{\alpha^{''}}^{k^{''}})-E_{k^{''}}(x',
        x_{\alpha^{''}}^{k^{''}})\Big|\\
    &&\leq \sum_{k'}\Big|\Big\langle
        \big[\,\overline{D}_{k'}(x,\cdot)-\overline{D}_{k'}(x',\cdot)\big],{\psi_{\alpha^{''}}^{k^{''}}(\cdot) \over
        \sqrt{\mu({Q}_{\alpha^{''}}^{k^{''}})}
        }\Big\rangle\Big|\\
    &&\leq C\delta^{\eta N}\sum_{k'}\delta^{\vert k' - {k^{''} }\vert
        {\eta}}\Big(
        {d(x,x')\over \delta^{k'}}\Big)^{\eta'} \Big[\frac{1}{V_{\delta^{
        k^{''}}}(x)+V(x,x^{{k^{''}}
        }_{\alpha ^{''}})} \Big({{\delta^{{k^{''}}}
        }\over{\delta^{{k^{''}}} + d(x,x^{{k^{''}}}_{\alpha^{''}})}}\Big)^{\gamma }\\
    &&\hskip1cm {}+ \frac{1}{V_{\delta^{k^{''}}}(x')+V(x',x^{{k^{''}}
        }_{\alpha ^{''}})} \Big({{\delta^{{k^{''}}}
        }\over{\delta^{{k^{''}}}+ d(x',x^{{k^{''}}}_{\alpha^{''}})}}\Big)^{\gamma }\Big]\\
    &&\leq C\delta^{\eta N}\sum_{k'}\delta^{\vert k' - {k^{''} }\vert
        {(\eta - \eta')}} \Big(
        {d(x,x')\over \delta^{k''}}\Big)^{\eta'}  \Big[\frac{1}{V_{\delta^{
        k^{''})}}(x)+V(x,x^{{k^{''}}
        }_{\alpha ^{''}})} \Big({{\delta^{{k^{''}}}
        }\over{\delta^{{k^{''}}}+ d(x,x^{{k^{''}}}_{\alpha^{''}})}}\Big)^{\gamma }\\
    &&\hskip1cm {}+ \frac{1}{V_{\delta^{
        k^{''})}}(x')+V(x',x^{{k^{''}}
        }_{\alpha ^{''}})} \Big({{\delta^{{k^{''}}}
        }\over{\delta^{{k^{''}}}+ d(x',x^{{k^{''}}}_{\alpha^{''}})}}\Big)^{\gamma }\Big],\\
    &&\leq  C\delta^{\eta N} \Big({d(x,x')\over
        \delta^{k^{''}}}\Big)^{\eta'} \Big[{1\over
        V_{\delta^{k^{''}}}(x)+V(x,x_{\alpha^{''}}^{k^{''}})}
        \Big({\delta^{k^{''}}\over \delta^{k^{''}}+
        d(x,x_{\alpha^{''}}^{k^{''}})}\Big)^{\gamma}\\
    &&\hskip1cm {}+ {1\over V_{\delta^{k^{''}}}(x')+V(x,x_{\alpha^{''}}^{k^{''}})}
        \Big({\delta^{k^{''}}\over \delta^{k^{''}}+
        d(x',x_{\alpha^{''}}^{k^{''}})}\Big)^{\gamma}\Big],
\end{eqnarray*}
for all $\eta' \in (0,\eta)$.

For $k'\leq k^{''}$, we have
\begin{eqnarray*}
    &&\hskip-.3cm|E_{k^{''}}(x, x_{\alpha^{''}}^{k^{''}})-E_{k^{''}}(x',
        x_{\alpha^{''}}^{k^{''}})|\\
    &&\leq \sum_{k'<k^{''}}\int_X
        \big|\big[\,\overline{D}_{k'}(x,y) - \overline{D}_{k'}(x',y)\big]\big|\
        \Big|{\psi_{\alpha^{''}}^{k^{''}}(y) \over
        \sqrt{\mu({Q}_{\alpha^{''}}^{k^{''}})} }\Big| \, d\mu(y)\\
    &&\leq C\delta^{\eta N}\sum_{k'\leq k^{''}}\int_X
        \Big( {d(x,x')\over \delta^{k'}}
        \Big)^{\eta} \Big[{1\over V_{\delta^{k'}}(x) +
        V(x,y)} \Big( {\delta^{k'}\over \delta^{k'}+d(x,y)} \Big)^\gamma \\
    &&\hskip1cm {}+ {1\over V_{\delta^{k'}}(x') +
        V(x',y)} \Big( {\delta^{k'}\over \delta^{k'}+d(x',y)} \Big)^\gamma\Big]
        {1\over V_{\delta^{k^{''}}}(y) +
        V(y,x_{\alpha^{''}}^{k^{''}})} \Big( {\delta^{k^{''}}\over
        \delta^{k^{''}}+d(y,x_{\alpha^{''}}^{k^{''}})} \Big)^\gamma
         \, d\mu(y)\\
    &&\leq C\delta^{\eta N}\Big( {d(x,x')\over \delta^{k^{''}}}
        \Big)^{\eta}\sum_{k'\leq k^{''}}  \delta^{\eta(k^{''}-k')}
        \Big[{1\over V_{\delta^{k'}}(x) +
        V(x,x_{\alpha^{''}}^{k^{''}})} \Big( {\delta^{k'}\over \delta^{k'}
            + d(x,x_{\alpha^{''}}^{k^{''}})} \Big)^\gamma \\
    &&\hskip1cm {}+ {1\over V_{\delta^{k'}}(x') +
        V(x',x_{\alpha^{''}}^{k^{''}})} \Big( {\delta^{k'}\over \delta^{k'}
            + d(x',x_{\alpha^{''}}^{k^{''}})} \Big)^\gamma\Big]\\
    &&\leq C\delta^{\eta N}\Big( {d(x,x')\over \delta^{k^{''}}}
        \Big)^{\eta}\sum_{k'\leq k^{''}}  \delta^{(\eta-\gamma)(k^{''}-k')}
        \Big[{1\over V_{\delta^{k''}}(x) +
        V(x,x_{\alpha^{''}}^{k^{''}})} \Big( {\delta^{k''}\over \delta^{k''}
            + d(x,x_{\alpha^{''}}^{k^{''}})} \Big)^\gamma \\
    &&\hskip1cm {}+ {1\over V_{\delta^{k''}}(x') +
        V(x',x_{\alpha^{''}}^{k^{''}})} \Big( {\delta^{k''}\over \delta^{k''}
            + d(x',x_{\alpha^{''}}^{k^{''}})} \Big)^\gamma\Big]\\
    &&\leq C\delta^{\eta N}\Big( {d(x,x')\over \delta^{k^{''}}}
        \Big)^{\eta}\Big[{1\over V_{\delta^{k''}}(x) +
        V(x,x_{\alpha^{''}}^{k^{''}})} \Big( {\delta^{k''}\over \delta^{k''}
            + d(x,x_{\alpha^{''}}^{k^{''}})} \Big)^\gamma \\
    &&\hskip1cm {}+ {1\over V_{\delta^{k''}}(x') +
        V(x',x_{\alpha^{''}}^{k^{''}})} \Big( {\delta^{k''}\over \delta^{k''}
            + d(x',x_{\alpha^{''}}^{k^{''}})} \Big)^\gamma\Big],
\end{eqnarray*}
for all $\gamma \in (0,\eta)$.

With the almost-orthogonality estimate in hand, the same
argument as for~\eqref{eqn:PP1}, via the estimate
from~\cite{FJ} and the Fefferman--Stein vector-valued maximal
function, yields
\begin{eqnarray*}
    \|S(R^{(1)}_N(f))\|_{L^p(X)}
    &=&\Big\|\Big\{ \sum_{k}\sum_{\alpha\in\mathscr{Y}^k} \big| \langle
        \psi_{\alpha}^{k},R^{(1)}_N(f) \rangle \widetilde{\chi}_{{Q}_{\alpha}^{k}}(\cdot)
        \big|^2 \Big\}^{1/2}
        \Big\|_{L^p(X)}\\
    &\leq& C\delta^{\eta N}\Big\|\Big\{ \sum_{k}\sum_{\alpha\in\mathscr{Y}^k} \big| \langle
        \psi_{\alpha}^{k},(f) \rangle \widetilde{\chi}_{{Q}_{\alpha}^{k}}(\cdot)
        \big|^2 \Big\}^{1/2}
        \Big\|_{L^p(X)}
    = C\delta^{\eta N}\|S(f)\|_{L^p(X)}.
\end{eqnarray*}
A similar proof shows that $\|S(R^{(2)}_N(f))\|_{L^p(X)} \leq
C\delta^{\eta N}\|S(f)\|_{L^p(X)}$.
Therefore~\eqref{eqn:S(R_N)_Lp_estimate} holds:
\[
    \|S(R_N(f))\|_{L^p(X)}
    \leq C\delta^{\eta N}\|S(f)\|_{L^p(X)},
\]
as required.

\noindent (iii) Consider the Neumann series $(T_N)^{-1} = (I -
R_N)^{-1} = \sum_{i=0}^{\infty}(R_N)^i$.
By~\eqref{eqn:S(R_N)_Lp_estimate} we have
\[
    \|S((T_N)^{-1}(f))\|_{L^p(X)}
    \leq \sum_{i=0}^{\infty} \| S((R_N)^i(f)) \|_{L^p(X)}
    \leq (1 - C\delta^{\eta N})^{-1} \|S(f)\|_{L^p(X)},
\]
as required, if $N$ is chosen sufficiently large
that~$C\delta^{\eta N} < 1$.


This completes the proof of Lemma~\ref{lem:TN_bounded}.
\end{proof}

We turn to the product setting.


\subsection{Product square functions via wavelets, and
Plancherel--P\'olya
inequalities}\label{sec:productsquarefunction}

We now assume that $\widetilde X = X_1\times X_2$ where each
$X_i$ is a space of homogeneous type as above. In this
subsection, $(x_1,x_2)$ denotes an element of~$X_1 \times X_2$.

\begin{definition}\label{def:product square functions}
    (Product square functions) Take $\beta_i \in (0,\eta_i)$
    and $\gamma_i > 0$, for $i = 1$, $2$, and consider $f\in
    \big(\GGp(\beta_{1},\beta_{2};
    \gamma_{1},\gamma_{2})\big)'$.

    \noindent (a) The {\it discrete product Littlewood--Paley square
    function $\widetilde{S}(f)$} in terms of wavelet
    coefficients is defined by
    \begin{eqnarray}\label{g function}
        \widetilde{S}(f)(x_1,x_2)
        := \Big\{
        \sum_{k_1} \sum_{\alpha_1\in\mathscr{Y}^{k_1}}
            \sum_{k_2} \sum_{\alpha_2\in\mathscr{Y}^{k_2}}
            \Big| \langle \psi_{\alpha_1}^{k_1}\psi_{\alpha_2}^{k_2},f \rangle
            \widetilde{\chi}_{ Q_{\alpha_1}^{k_1}}(x_1)
            \widetilde{\chi}_{Q_{\alpha_2}^{k_2}}(x_2) \Big|^2 \Big\}^{1/2},
    \end{eqnarray}
    where $\psi_{\alpha_1}^{k_1}\psi_{\alpha_2}^{k_2} =
    \psi_{\alpha_1}^{k_1}\otimes \psi_{\alpha_2}^{k_2}$ with
    $\psi_{\alpha_i}^{k_i}$ acting on the $X_i$ variable for
    $i = 1$, $2$, and $\widetilde{\chi}_{Q_{\alpha_i}^{k_i}}(x_i)
    := \chi_{Q_{\alpha_i}^{k_i}}(x_i)
    \mu_i(Q_{\alpha_i}^{k_i})^{-1/2}$.

    \noindent (b) The {\it continuous product Littlewood--Paley
    square function ${\widetilde S}_c(f)$} in terms of wavelet
    operators is defined by
    \begin{eqnarray}
        {\widetilde S}_c(f)(x_1,x_2)
        := \Big\{\sum_{k_1}\sum_{k_2}
            \Big| D_{k_1}D_{k_2}(f)(x_1,x_2)\Big|^2 \Big\}^{1/2},
    \end{eqnarray}
    where $D_{k_i} := \sum_{\alpha_i\in\mathscr{Y}^{k_i}}
    \psi_{\alpha_i}^{k_i}$ for $i = 1$, $2$, and
    $D_{k_1}D_{k_2} := D_{k_1}\otimes D_{k_2}$.
\end{definition}

The main results of this subsection are the following product
versions of the Littlewood--Paley theory and the
Plancherel--P\'olya inequalities.

\begin{theorem}\label{theorem Littlewood Paley product}
    \textup{(Product Littlewood--Paley theory)} Suppose
    $\beta_i \in (0,\eta_i)$, $\gamma_i > 0$, and \linebreak
    $\max\{\frac{\omega_1}{\omega_1 +
    \eta_1},\frac{\omega_2}{\omega_2 + \eta_2}\} < p < \infty$,
    where $\omega_i$ is the upper dimension of $X_i$, for $i =
    1$, $2$. For all $f\in \big(\GGp(\beta_{1},\beta_{2};
    \gamma_{1},\gamma_{2})\big)'$, we have
    \[
        \|{\widetilde S}(f)\|_{L^p(X_1\times X_2)}
        \sim \|{\widetilde S}_c(f)\|_{L^p(X_1\times X_2)}.
    \]
    Moreover, if $1 < p < \infty$, then
    \[
        \|{\widetilde S}(f)\|_{L^p(X_1\times X_2)}
        \sim \|{\widetilde S}_c(f)\|_{L^p(X_1\times X_2)}
        \sim \|f\|_{L^p(X_1\times X_2)}.
    \]
\end{theorem}

\begin{theorem}\label{theorem P P inequality product}
    \textup{(Product Plancherel--P\'olya inequalities)} Suppose
    $\beta_i \in (0,\eta_i)$, $\gamma_i > 0$,
    and $\max\{\frac{\omega_1}{\omega_1
    +\eta_1},\frac{\omega_2}{\omega_2 +\eta_2}\} < p < \infty$,
    where $\omega_i$ is the upper dimension of $X_i$, for $i = 1$, $2$.
    Take $N_1$, $N_2\in\N$. Then there is a positive constant
    $C$ such that for all $f\in \big(\GGp(\beta_{1},\beta_{2};
    \gamma_{1},\gamma_{2})\big)'$, we have
    \begin{eqnarray}
        &&\Big\|\Big\{\sum_{k'_1}
            \sum_{\alpha'_1\in\mathscr{X}^{k'_1+N_1}}
            \sum_{k'_2}\sum_{\alpha'_2\in\mathscr{X}^{k'_2+N_2}}\nonumber\\
        &&\hskip2cm\sup_{(z_1,z_2)\in Q^{{k'_1 +N_1}}_{\alpha'_1 }
            \times Q^{{k'_2 +N_2}}_{\alpha'_2 }}
            \big|D_{k'_1}D_{k'_2}(f)(z_1,z_2)\big|
            \chi_{Q^{{k'_1 +N_1}}_{\alpha'_1 }}(\cdot)
            \chi_{Q^{{k'_2 +N_2}}_{\alpha'_2 }}(\cdot)
            \Big\}^{1/2}\Big\|_{L^p(X_1\times X_2)}\nonumber\\
        &&\hskip.5cm \leq C \Big\|\Big\{ \sum_{k_1}
            \sum_{\alpha_1\in\mathscr{Y}^{k_1}}
            \sum_{k_2}\sum_{\alpha_2\in\mathscr{Y}^{k_2}}
            \big| \langle \psi_{\alpha_1}^{k_1}\psi_{\alpha_2}^{k_2},f \rangle
            \widetilde{\chi}_{Q_{\alpha_1}^{k_1}}(\cdot)
            \widetilde{\chi}_{Q_{\alpha_2}^{k_2}}(\cdot)
            \big|^2 \Big\}^{1/2}
            \Big\|_{L^p(X_1\times X_2)}.\label{P P 2}
    \end{eqnarray}
    Further, suppose $N_1$ and $N_2$ are sufficiently large
    positive integers, to be determined during the proof below.
    Then there is a positive constant $C$ such that for all
    $f\in \big(\GGp(\beta_{1},\beta_{2};
    \gamma_{1},\gamma_{2})\big)'$, we have
    \begin{eqnarray}
        &&\Big\|\Big\{ \sum_{k_1}
            \sum_{\alpha_1\in\mathscr{Y}^{k_1}}
            \sum_{k_2}
            \sum_{\alpha_2\in\mathscr{Y}^{k_2}}
            \big| \langle \psi_{\alpha_1}^{k_1}\psi_{\alpha_2}^{k_2},f \rangle
            \widetilde{\chi}_{Q_{\alpha_1}^{k_1}}(\cdot)
            \widetilde{\chi}_{Q_{\alpha_2}^{k_2}}(\cdot)
            \big|^2 \Big\}^{1/2}
            \Big\|_{L^p(X_1\times X_2)}\nonumber\\
        &&\hskip.5cm \leq C \Big\|\Big\{
            \sum_{k'_1}
            \sum_{\alpha'_1\in\mathscr{X}^{k'_1+N_1}}
            \sum_{k'_2}\sum_{\alpha'_2\in\mathscr{X}^{k'_2+N_2}}\nonumber\\
        &&\hskip2cm\inf_{(z_1,z_2)\in Q^{{k'_1 +N_1}}_{\alpha'_1 }
            \times Q^{{k'_2 +N_2}}_{\alpha'_2 }}
            \big|D_{k'_1}D_{k'_2}(f)(z_1,z_2)
            \big| \chi_{Q^{{k'_1 +N_1}}_{\alpha'_1 }}(\cdot)
            \chi_{Q^{{k'_2 +N_2}}_{\alpha'_2 }}(\cdot)
            \Big\}^{1/2}\Big\|_{L^p(X_1\times X_2)}.\label{P P 1}
    \end{eqnarray}
\end{theorem}

\begin{proof}[Proofs of Theorems~\ref{theorem Littlewood Paley product}
and~\ref{theorem P P inequality product}] The proofs of
Theorems~\ref{theorem Littlewood Paley product}
and~\ref{theorem P P inequality product} are analogous to those
for the case of one factor. As mentioned in that case, the
proofs of Theorems~\ref{theorem Littlewood Paley}
and~\ref{theorem P P inequality one-parameter} follow from the
almost-orthogonality estimates, namely the
claim~\eqref{eqn:almost_orthogonality_estimate}. To see that
these proofs can be carried over to the product case, we make
an observation analogous to the
claim~\eqref{eqn:almost_orthogonality_estimate}, as follows:
\[
    \Big\langle D_{k_1}D_{k_2}(x_1,x_2,\cdot,\cdot),
        \psi_{\alpha_1}^{k_1}\psi_{\alpha_2}^{k_2}(\cdot,\cdot,y_1,y_2)\Big\rangle
    = \big\langle D_{k_1}(x_1,\cdot),
        \psi_{\alpha_1}^{k_1}(\cdot,y_1)\big\rangle
        \big\langle D_{k_2}(x_2,\cdot),\psi_{\alpha_2}^{k_2}(\cdot,y_2)\big\rangle,
\]
which together with the almost-orthogonality estimates for the
one-factor case yields the desired almost-orthogonality
estimates for the product case. For the product case, all
estimates analogous to those in (a)--(c),
$\textup{(a)}'$--$\textup{(b)}'$, \eqref{eqn:PP1}
and~\eqref{eqn:PP2} follow similarly. We omit the details.
\end{proof}


\section{Product $H^p$, $\cmo^p$, $\bmo$ and $\vmo$, and
duality}\label{sec:functionspaces}
\setcounter{equation}{0}

In this section we define the Hardy spaces $H^p$, the Carleson
measure spaces $\cmo^p$ (including the bounded mean oscillation
space $\bmo = \cmo^1$), and the vanishing mean oscillation
space~$\vmo$, in the setting of product spaces of homogeneous
type. Both $H^p$ and $\cmo^p$ are defined here for $p$ in the
range $\max\{{\omega_1\over \omega_1 + \eta_1},{\omega_2\over
\omega_2 + \eta_2}\} < p \leq 1$, where $\omega_i$ is the upper
dimension of $X_i$, for $i = 1$, $2$. We prove that $\cmo^p$ is
the dual of $H^p$, and in particular that $\bmo$ is the dual of
$H^1$, and also that $H^1$ is the dual of~$\vmo$.

We develop this theory in the product case with two parameters.
The generalization to $k$ parameters, $k\in\N$, is similar to
the two-parameter case, while the specialization to one
parameter is immediate. Note the difference from the
Littlewood--Paley theory developed in
Section~\ref{sec:squarefunctionPP} above; there it was
necessary to develop the one-parameter theory first, then to
pass to the product case by iteration.

Fix $\beta_i \in (0,\eta_i)$ and $\gamma_i > 0$, for $i = 1$,
2.

For brevity, we denote by $\GG$ and $(\GG)'$ the test function
space $\GGp(\beta_{1},\beta_{2};\gamma_{1},\gamma_{2})$ and the
space of distributions $\big(\GGp(\beta_{1},\beta_{2};
\gamma_{1},\gamma_{2})\big)^{'}$, respectively.

We are now ready to introduce the Hardy spaces
$H^p(\XX_1\times\XX_2)$ and the Carleson measure
spaces~$\cmo^p(\XX_1\times\XX_2)$. In this section, $(x_1,x_2)$
denotes an element of~$X_1 \times X_2$.

\begin{definition}\label{def-Hp}
    (Hardy spaces) Suppose $\max\{{\omega_1\over \omega_1+\eta_1},
    {\omega_2\over \omega_2+\eta_2}\} < p \leq 1$, where
    $\omega_i$ is the upper dimension of $X_i$ for $i = 1$, 2.
    The {\it Hardy spaces} $H^p(X_1 \times X_2)$ are defined by
    \[
        H^p(X_1\times X_2)
        :=\big\lbrace f \in (\GG)':
            \widetilde{S}(f)\in L^p(\XX_1\times\XX_2)\big\rbrace,
    \]
    where $\widetilde{S}(f)$ is the discrete product Littlewood--Paley
    square function as in~Definition~\ref{def:product square functions}.

    For $f\in H^p(X_1\times X_2)$, we define
    $\|f\|_{H^p(X_1\times X_2)} := \|\widetilde{S}(f)\|_{L^p(X_1\times
    X_2)}$.
\end{definition}

For completeness, we note that as in the classical case, for $1
< p < \infty$ the Hardy space~$H^p(X_1\times X_2)$ of
Definition~\ref{def-Hp} coincides with $L^p(X_1\times X_2)$.

We point out that $\GG$ and hence $H^p(X_1\times X_2)\cap
L^2(X_1\times X_2)$ are dense in $H^p(X_1\times X_2)$. Indeed,
if $f\in H^p(X_1\times X_2)$, then by Theorem~\ref{thm
reproducing formula test function}, the functions
\[
    f_n(x_1,x_2)
    := \sum_{|k_1|,|k_2|\leq n}
        \sum_{\alpha_1\in\mathscr{Y}^{k_1}}
        \sum_{\alpha_2\in\mathscr{Y}^{k_2}}
        \psi_{\alpha_1}^{k_1}(x_1)
        \psi_{\alpha_2}^{k_2}(x_2)
        \langle \psi_{\alpha_1}^{k_1}\psi_{\alpha_2}^{k_2},f \rangle
\]
belong to $\GG$. Moreover,
\[
    \widetilde{S}(f - f_n)(x_1,x_2)
    \leq\Big\{ \sum_{|k_1| > n \ {\rm or}\ |k_2|>n}
        \sum_{\alpha_1\in\mathscr{Y}^{k_1}}
        \sum_{\alpha_2\in\mathscr{Y}^{k_2}}
        \Big| \langle \psi_{\alpha_1}^{k_1}\psi_{\alpha_2}^{k_2},f \rangle
        \widetilde{\chi}_{Q_{\alpha_1}^{k_1}}(x_1)
        \widetilde{\chi}_{Q_{\alpha_2}^{k_2}}(x_2) \Big|^2
        \Big\}^{1/2}.
\]
Therefore $\|\widetilde{S}(f - f_n)\|_{L^p(X_1\times X_2)}$
tends to zero as $n$ tends to infinity. Hence $\GG$ is dense in
$H^p(X_1\times X_2)$.

\begin{definition}\label{def-CMO}
    (Carleson measure spaces, and bounded mean oscillation)
    Suppose that $\max\{{\omega_1\over \omega_1 + \eta_1},
    {\omega_2\over \omega_2 + \eta_2}\} < p \leq 1$, where
    $\omega_i$ is the upper dimension of $X_i$ for $i = 1$, 2.
    We define the {\it Carleson measure spaces} $\cmo^p$ in
    terms of wavelet coefficients by
    $$
        \cmo^p(\XX_1\times\XX_2)
        := \big\{ f \in (\GG)' : \mathcal{C}_p(f)< L^\infty\},
    $$
    with the quantity $\mathcal{C}_p(f)$ defined as follows:
    \begin{eqnarray}\label{Carleson norm}
        \mathcal{C}_p(f)
        := \sup_{\Omega}
            \Big\{ {1\over\mu(\Omega)^{{2\over p} - 1}}
            \sum_{\substack{R = Q_{\alpha_1}^{k_1}\times Q_{\alpha_2}^{k_2}
                \subset \Omega, \\
                k_1,k_2\in\mathbb{Z}, \alpha_1\in\mathscr{Y}^{k_1},
                \alpha_2\in\mathscr{Y}^{k_2} }}
            \big| \langle \psi_{\alpha_1}^{k_1}\psi_{\alpha_2}^{k_2}, f \rangle
            \big|^2 \Big\}^{1/2},
    \end{eqnarray}
    where $\Omega$ runs over all open sets in $\XX_1\times\XX_2$ with
    finite measure.

    The \emph{space $\bmo$ of functions of bounded mean oscillation} is
    defined by
    \[
        \bmo(X_1 \times X_2)
        := \cmo^1(X_1\times X_2).
    \]
\end{definition}

The main result in this section is the following.

\begin{theorem}\label{thm-duality}
    Suppose $\max\{{\omega_1 \over \omega_1 + \eta_1},
    {\omega_2\over \omega_2 + \eta_2}\} < p \leq 1$, where
    $\omega_i$ is the upper dimension of $X_i$ for $i = 1$, $2$.
    Then the Carleson measure space~$\cmo^p(X_1\times X_2)$ is the
    dual of the Hardy space~$H^p(X_1\times X_2)$:
    \[
        \big(H^p(\XX_1\times\XX_2)\big)'
        = \cmo^p(\XX_1\times\XX_2).
    \]
    In particular,
    \[
        \big(H^1(\XX_1\times\XX_2)\big)'
        = \bmo(\XX_1\times\XX_2).
    \]
\end{theorem}

To prove Theorem \ref{thm-duality}, we follow the approach
developed in \cite{HLL2}; see also \cite{HLL1}. We first recall
the definitions of the product sequence spaces $s^p$ and $c^p$
for $0 < p \leq 1$. These sequence spaces are discrete
analogues of $H^p(X_1 \times X_2)$ and $\cmo^p(X_1 \times X_2)$
respectively.

The space $s^p$ is defined to be the set of sequences
$s=\{s_R\}_R$ of real numbers such that
\begin{eqnarray}\label{sp space}
    \|s\|_{s^p}
    := \big\|\big\{\sum_R
        |\mu(R)^{-1/2}s_R\,\chi_R(\cdot,\cdot)|^2
        \big\}^{1/2}\big\|_{L^p(X_1 \times X_2)}
    < \infty,
\end{eqnarray}
where $R$ runs over all dyadic rectangles in $X_1\times X_2$. The
space $c^p$ is defined to be the set of sequences $t=\{t_R\}_R$ of
real numbers such that
\begin{eqnarray}\label{cp space}
    \|t\|_{c^p}
    := \sup_{\Omega}
        \Big({1\over \mu(\Omega)^{{2\over p} - 1}}
        \sum_{R\subset\Omega}|t_R|^2\Big)^{1/2}
    < \infty,
\end{eqnarray}
where $\Omega$ runs over all open sets in $\XX_1\times\XX_2$
with finite measure, and $R$ runs over all dyadic rectangles
contained in~$\Omega$.

We emphasize that in the above definitions of $s^p$ and $c^p$,
the expression ``all dyadic rectangles~$R$'' indicates the
rectangles of the form $R = Q_{\alpha_1}^{k_1}\times
Q_{\alpha_2}^{k_2}$ for all $k_i\in \mathbb Z$ and $\alpha_i\in
{\mathscr X}^{k_i}$ for $i = 1$, 2.

The main result about the sequence spaces $s^p$ and
$c^p$ is the following duality result.

\begin{prop}[\cite{HLL2}]\label{prop dual-of-sequence-space}
    For $0 < p \leq 1$, $\big(s^p\big)' = c^p$.
\end{prop}

We now introduce the lifting and projection operators $T_L$
and~$T_p$, as follows.

\begin{definition}\label{def-of-lifting-operator-on-product-case}
    For $f\in (\GG)'$, the {\it lifting operator} $T_{L}$ is
    defined by
    \begin{eqnarray}\label{lifting-operator}
        \lbrace (T_Lf)_R\rbrace_R
        := \big\{ \langle\psi_{\alpha_1}^{k_1}\psi_{\alpha_2}^{k_2},
            f \rangle \big\}_R,
    \end{eqnarray}
    where $R = Q_{\alpha_1}^{k_1}\times Q_{\alpha_2}^{k_2}$,
    $k_1$, $k_2\in\mathbb{Z}$, $\alpha_1\in \mathscr{Y}^{k_1}$,
    $\alpha_2\in \mathscr{Y}^{k_2}$ are dyadic rectangles in
    $\XX_1\times\XX_2$.
\end{definition}

\begin{definition}\label{def-of-projection-operator-on-product-case}
    Given a sequence $\lambda = \{\lambda_{R}\}$ of real numbers, we
    define the associated {\it projection operator}~$T_P$ by
    \begin{eqnarray}\label{projection-operator}
        T_P(\lambda)(x_1,x_2)
        := \sum_{\substack{R = Q_{\alpha_1}^{k_1}\times Q_{\alpha_2}^{k_2}, \\
            k_1,k_2\in\mathbb{Z},\alpha_1\in\mathscr{Y}^{k_1},
            \alpha_2\in\mathscr{Y}^{k_2}}}
            \lambda_{R}\cdot \psi_{\alpha_1}^{k_1}(x_1)\psi_{\alpha_2}^{k_2}(x_2).
    \end{eqnarray}
\end{definition}

From the definitions of the lifting and projection operators
$T_L$ and $T_P$, it follows that $f = T_P\circ T_{L}(f)$ in the
sense of the test function space $\GG$ and of the
distributions~$(\GG)'$. That is, $T_P\circ T_{L}$ is an
identity operator on the distributions~$(\GG)'$.

Next we give two auxiliary results which will be used in
establishing the duality in Theorem~\ref{thm-duality}.

\begin{prop}\label{prop-of-sp-Hp}
    Suppose $\max\big\{\frac{ \omega_1}{ \omega_1 + \eta_1 },
    \frac{ \omega_2}{\omega_2 + \eta_2 }\big\} < p \leq 1$,
    where $\omega_i$ is the upper dimension of~$X_i$ for $i =
    1$, $2$. Then for all $f\in H^p(\XX_1\times \XX_2)$, we have
    \begin{eqnarray}\label{Hp to sp}
        \|T_L(f)\|_{s^p}
        \lesssim \|f\|_{H^p(\XX_1\times \XX_2)}.
    \end{eqnarray}
    In the other direction, for each $s \in s^p$ we have
    \begin{eqnarray}\label{sp to Hp}
        \|T_P(s)\|_{H^p(\XX_1\times \XX_2)}
        \lesssim \|s\|_{s^p}.
    \end{eqnarray}
\end{prop}

\begin{proof}
Inequality~(\ref{Hp to sp}) follows directly from the
definitions of $H^p(\XX_1\times \XX_2)$
(Definition~\ref{def-Hp}) and the sequence space $s^p$
(formula~(\ref{sp space})).

We now prove (\ref{sp to Hp}). For each $s\in s^p$, by the
definitions of $H^p(X_1 \times X_2)$ and $T_P(s)$, we have
\begin{eqnarray*}
    \lefteqn{\|T_P(s)\|_{H^p(\XX_1\times \XX_2)}
        = \|\widetilde{S}(T_P(s))\|_{L^p(\XX_1\times \XX_2)}} \hspace{1cm}\\
    &&= \Big\|\Big\{ \sum_{\substack{R = Q_{\alpha_1}^{k_1}\times
        Q_{\alpha_2}^{k_2}, \\
        k_1,k_2\in\mathbb{Z},
        \alpha_1\in\mathscr{Y}^{k_1},
        \alpha_2\in\mathscr{Y}^{k_2}}}
        \Big| \Big\langle
        \psi_{\alpha_1}^{k_1}\psi_{\alpha_2}^{k_2},\ \
        \sum_{\substack{R'=Q_{\alpha'_1}^{k'_1}\times
        Q_{\alpha'_2}^{k'_2} \\
        k'_1,k'_2\in\mathbb{Z}, \alpha'_1\in\mathscr{Y}^{k'_1},
        \alpha'_2\in\mathscr{Y}^{k'_2}}}
        s_{R'}\cdot
        \psi_{\alpha'_1}^{k'_1}\psi_{\alpha'_2}^{k'_2}
        \Big\rangle\hskip.5cm \\
    &&\hskip1cm\tilde{\chi}_{Q_{\alpha_1}^{k_1}}(x_1)
        \tilde{\chi}_{Q_{\alpha_2}^{k_2}}(x_2) \Big|^2
        \Big\}^{1/2}\Big\|_{L^p(\XX_1\times
        \XX_2)}\\
    &&=\Big\|\Big\{ \sum_{\substack{R = Q_{\alpha_1}^{k_1}\times
        Q_{\alpha_2}^{k_2}, \\
        k_1,k_2\in\mathbb{Z},
        \alpha_1\in\mathscr{Y}^{k_1},
        \alpha_2\in\mathscr{Y}^{k_2}}}
        \big| s_{R}\cdot
        \tilde{\chi}_{Q_{\alpha_1}^{k_1}}(x_1)
        \tilde{\chi}_{Q_{\alpha_2}^{k_2}}(x_2) \big|^2
        \Big\}^{1/2}\Big\|_{L^p(\XX_1\times \XX_2)}\\
    &&\leq \|s\|_{s^p},
\end{eqnarray*}
where the third equality follows from the orthogonality of the
bases $\{\psi_{\alpha_1}^{k_1}\}$ and $\{\psi_{\alpha_2}^{k_2}\}$.
\end{proof}

\begin{prop}\label{prop-of-cp-CMOp}
    Suppose $\max\big\{\frac{ \omega_1}{ \omega_1 + \eta_1
    },\frac{ \omega_2}{ \omega_2 + \eta_2 }\big\} < p \leq 1$, where
    $\omega_i$ is the upper dimension of~$X_i$ for $i = 1$, $2$.
    For all $f\in \cmo^p( X_1\times X_2)$, we have
    \begin{eqnarray}\label{CMOp to cp}
        \|T_L(f)\|_{c^p}
        \lesssim \mathcal{C}_p(f).
    \end{eqnarray}
    In the other direction, for each $t\in c^p$,
    \begin{eqnarray}\label{cp to CMOp}
        \mathcal{C}_p\big(T_P(t)\big)
        \lesssim \|t\|_{c^p}.
    \end{eqnarray}
\end{prop}

\begin{proof}
Inequality~(\ref{CMOp to cp}) follows directly from the
definitions of $\cmo^p(\XX_1\times \XX_2)$
(Definition~\ref{def-CMO}) and $c^p$ (formula~(\ref{cp
space})).

We now prove~(\ref{cp to CMOp}). For each $t\in c^p$ we have
\begin{eqnarray*}
    \mathcal{C}_p\big(T_P(t)\big)
    &=& \sup_{\Omega}\Big\{{1\over\mu(\Omega)}
        \sum_{\substack{R = Q_{\alpha_1}^{k_1}\times
        Q_{\alpha_2}^{k_2}\subset \Omega, \\
        k_1,k_2\in\mathbb{Z},
        \alpha_1\in\mathscr{Y}^{k_1},
        \alpha_2\in\mathscr{Y}^{k_2}}}
        \big| \langle
        \psi_{\alpha_1}^{k_1}\psi_{\alpha_2}^{k_2},\ T_P(t) \rangle
        \big|^2 \Big\}^{1/2}\\
    &=& \sup_{\Omega}\Big\{{1\over\mu(\Omega)}
        \sum_{\substack{R = Q_{\alpha_1}^{k_1}\times
        Q_{\alpha_2}^{k_2}\subset \Omega, \\
        k_1,k_2\in\mathbb{Z},
        \alpha_1\in\mathscr{Y}^{k_1},
        \alpha_2\in\mathscr{Y}^{k_2}}}
        \big| t_R \big|^2 \Big\}^{1/2}\\
    &\leq& \|t\|_{c^p},
\end{eqnarray*}
where the second equality follows from the orthogonality of the
bases $\{\psi_{\alpha_1}^{k_1}\}$ and $\{\psi_{\alpha_2}^{k_2}\}$.
\end{proof}

We would like to point out that thanks to the orthogonality of
the wavelet basis from~\cite{AH}, the proofs given here of
(\ref{sp to Hp}) and~(\ref{cp to CMOp}) are much simpler than
those given in~\cite{HLL2}.

We are ready to prove the duality
$\big(H^p(\XX_1\times\XX_2)\big)' = \cmo^p(\XX_1\times\XX_2)$.

\begin{proof}[Proof of Theorem \ref{thm-duality}]
Suppose $\max\{{\omega_1\over \omega_1+\eta_1},{\omega_2\over
\omega_2 + \eta_2}\} <p\leq1$. We first show that there exists
a positive constant $C$ such that for each $g\in
\cmo^p(\XX_1\times\XX_2)$,
\begin{eqnarray}\label{duality inequality}
    |\langle f,g\rangle|
    \leq C\|f\|_{H^p(\XX_1\times\XX_2)}\mathcal{C}_p(g)
\end{eqnarray}
for all $f\in \GG$. It follows that
$\cmo^p(\XX_1\times\XX_2)\subset
\big(H^p(\XX_1\times\XX_2)\big)'$, since $\GG$ is dense in
$H^p(X_1\times X_2).$

To prove inequality (\ref{duality inequality}), for each $f \in
\GG$ and $g\in \cmo^p(\XX_1\times \XX_2)$, by the reproducing
formula (\ref{product reproducing formula}) we have
\begin{eqnarray*}
    \langle f,g\rangle &=& \sum_{k_1,k_2\in\mathbb{Z},
        \alpha_1\in\mathscr{Y}^{k_1}, \alpha_2\in\mathscr{Y}^{k_2}}
        \langle\psi_{\alpha_1}^{k_1}\psi_{\alpha_2}^{k_2},\ f\rangle
        \langle\psi_{\alpha_1}^{k_1}\psi_{\alpha_2}^{k_2},\  g \rangle\\
    &=& \sum_{R} (T_L(f))_R\cdot (T_L(g))_R.
\end{eqnarray*}
where $T_L(f)$ and $T_L(g)$ are the lifting operators as in
Definition \ref{def-of-lifting-operator-on-product-case}.

Then, by Propositions \ref{prop-of-sp-Hp}
and~\ref{prop-of-cp-CMOp}, we obtain
\begin{eqnarray*}
|\langle f,g\rangle|\leq |\langle T_L(f),
T_L(g)\rangle| \leq C
\|f\|_{H^p(\XX_1\times \XX_2)}\mathcal{C}_p\big(g).
\end{eqnarray*}

Conversely, suppose $l\in \big(H^p(\XX_1\times \XX_2)\big)'$.
Let $l_1 := l\circ T_P$. By Proposition \ref{prop-of-sp-Hp}, we
see that $l_1 \in (s^p)'$, since for each $s\in s^p$, $|l_1(s)|
= |l\big(T_P(s)\big)| \leq C\|l\|\,\|T_P(s)\|_{H^p(\XX_1\times
\XX_2)} \leq C\|l\|\|s\|_{s^p}$. Now we have
$$ l(g) = l\circ T_P\circ T_L(g) = l_1(T_L(g)) $$
for each $g\in \GG$. So by Proposition \ref{prop
dual-of-sequence-space}, there exists $t\in c^p$ such that
$l_1(s) = \langle t,s\rangle$ for all $s\in s^p$ and
$\|t\|_{c^p} \sim \|l_1\|\lesssim \|l\|$. Hence
\begin{eqnarray*}
    l(g)
    = \langle t, T_L(g)\rangle
    = \langle T_P(t), g\rangle.
\end{eqnarray*}
By Definition \ref{def-CMO} and Proposition
\ref{prop-of-cp-CMOp}, we obtain that
$\|T_P(t)\|_{\cmo^p(\XX_1\times \XX_2)}\lesssim
\|t\|_{c^p}\lesssim \|l\|$. Hence $\big(H^p(\XX_1\times
\XX_2)\big)'\subset \cmo^p(\XX_1\times \XX_2)$.
\end{proof}

Now we introduce the space of functions of vanishing mean
oscillation.

\begin{definition}\label{def-vmo}
    (Vanishing mean oscillation) We define the \emph{space
    $\vmo(X_1\times X_2)$ of functions of vanishing mean
    oscillation} to be the subspace of $\bmo(X_1\times X_2)$
    consisting of those $f\in \bmo(X_1\times X_2)$ satisfying
    the three properties
    \begin{eqnarray*}
        &&{\rm (a)}\hskip.5cm \lim_{\delta\rightarrow 0}\ \sup_{\Omega:\
            \mu(\Omega)<\delta}\Big\{
            {1\over\mu(\Omega)}
            \sum_{\substack{R = Q_{\alpha_1}^{k_1} \times
            Q_{\alpha_2}^{k_2}\subset \Omega \\
            k_1,k_2\in\mathbb{Z},
            \alpha_1\in\mathscr{Y}^{k_1},
            \alpha_2\in\mathscr{Y}^{k_2}}}
            \big| \langle
            \psi_{\alpha_1}^{k_1}\psi_{\alpha_2}^{k_2},f \rangle
             \big|^2 \Big\}^{1/2} = 0;\\
        &&{\rm (b)}\hskip.5cm \lim_{N\rightarrow \infty}\ \sup_{\Omega:\
            {\rm diam}(\Omega)>N}\Big\{
            {1\over\mu(\Omega)}
            \sum_{\substack{R = Q_{\alpha_1}^{k_1}\times
            Q_{\alpha_2}^{k_2}\subset \Omega, \\
            k_1,k_2\in\mathbb{Z},
            \alpha_1\in\mathscr{Y}^{k_1},
            \alpha_2\in\mathscr{Y}^{k_2}}}
            \big| \langle
            \psi_{\alpha_1}^{k_1}\psi_{\alpha_2}^{k_2},f \rangle
             \big|^2 \Big\}^{1/2} = 0;\\
        &&{\rm (c)}\hskip.5cm \lim_{N\rightarrow \infty}\ \sup_{\Omega:\
            \Omega \subset ( B(x_1,N)\times B(x_2,N))^c }\\
        &&\hskip3cm\Big\{ {1\over\mu(\Omega)}
            \sum_{\substack{R=Q_{\alpha_1}^{k_1}\times
            Q_{\alpha_2}^{k_2}\subset \Omega \\
            k_1,k_2\in\mathbb{Z},
            \alpha_1\in\mathscr{Y}^{k_1},
            \alpha_2\in\mathscr{Y}^{k_2}}}
            \big| \langle
            \psi_{\alpha_1}^{k_1}\psi_{\alpha_2}^{k_2},f \rangle
             \big|^2 \Big\}^{1/2} = 0,\ {\rm where\ } \\
        &&\hskip1.15cm \text{$x_1$ and $x_2$ are arbitrary fixed points
            in $X_1$ and $X_2$, respectively.}
    \end{eqnarray*}
\end{definition}

We now show the duality of $\vmo(\XX_1\times\XX_2)$ with
$H^1(\XX_1\times\XX_2)$.

\begin{theorem}\label{thm-duality 2}
    The Hardy space $H^1(\XX_1\times\XX_2)$ is the dual of
    $\vmo(\XX_1\times\XX_2)$:
    \[
        \big(\vmo(\XX_1\times\XX_2)\big)'
        =  H^1(\XX_1\times\XX_2).
    \]
\end{theorem}

\begin{proof}
The proof of this theorem is similar to the proof of the
duality between $\vmo$ and~$H^1$ on Euclidean space given in
Section~5 of~\cite{LTW}. Following \cite{LTW}, we only sketch
the main steps of the proof. First, we use $FW$ to denote the
set of finite linear combinations of terms of the form $\{
\psi_{\alpha_1}^{k_1}\cdot \psi_{\alpha_2}^{k_2} \}$, where $\{
\psi_{\alpha_i}^{k_i} \}$ are wavelets on $X_i$, $i=1,2$, as in
Theorem~\ref{theorem AH orth basis}. Second, from
Definition~\ref{def-vmo}, we obtain that $\vmo(X_1\times X_2)$
is the closure of $FW$ in the $\bmo(X_1\times X_2)$ norm.

The inclusion $ H^1(\XX_1\times\XX_2)\subset\big(
\vmo(\XX_1\times\XX_2)\big)'$ follows from the duality of
$H^1(X_1\times X_2)$ with $\bmo(X_1\times X_2)$, which was
shown in Theorem \ref{thm-duality}. The reverse containment
follows from the fact that $FW$ is dense in
$H^1(\XX_1\times\XX_2)$ in terms of the $H^1(\XX_1\times\XX_2)$
norm and from the following inequality: for $f\in FW$,
\[
    \|f\|_{H^1(\XX_1\times\XX_2)}
    \leq C \sup_{\substack{b\in \vmo(\XX_1\times\XX_2), \\
        \|b\|_{\bmo(\XX_1\times\XX_2)} = 1}}
        |\langle b,f\rangle|.
    \qedhere
\]
\end{proof}

\section{Calder\'on--Zygmund decomposition and interpolation
on Hardy spaces}\label{sec:CZdecomposition}
\setcounter{equation}{0}

In this section we provide the Calder\'on--Zygmund
decomposition and prove an interpolation theorem on $H^p(
X_1\times X_2 )$. Note that $H^p( X_1\times X_2 ) = L^p(
X_1\times X_2 )$ for $1 < p < \infty$. In this section,
$(x_1,x_2)$ denotes an element of~$X_1 \times X_2$.

\begin{theorem}\label{theorem C-Z decomposition for Hp}
    Let $\max\{\frac{\omega_1}{\omega_1 + \eta_1},
    \frac{\omega_2}{\omega_2 + \eta_2}\} < p_2 \le 1$, where
    $\omega_i$ is the upper dimension of $X_i$ for $i = 1$, $2$.
    Suppose $p_2 < p < p_1 < \infty$, $\alpha > 0$, and $f\in
    H^p(X_1\times X_2)$. Then we may write
    \[
        f(x_1,x_2) = g(x_1,x_2) + b(x_1,x_2),
    \]
    where
    \[
        g\in H^{p_1}(X_1\times X_2)
        \qquad\text{and}\qquad
        b\in H^{p_2}(X_1\times X_2)
    \]
    are such that $\|g\|^{p_1}_{H^{p_1}(X_1 \times X_2)} \le
    C\alpha^{p_1-p}\|f\|^p_{H^p(X_1 \times X_2)}$ and
    $\|b\|^{p_2}_{H^{p_2}(X_1 \times X_2)}\le
    C\alpha^{p_2-p}\|f\|^p_{H^p(X_1 \times X_2)}$. Here $C$ is
    an absolute constant.
\end{theorem}

\begin{theorem}\label{theorem interpolation Hp}
    Suppose $\max\{\frac{\omega_1}{\omega_1+\eta_1},
    \frac{\omega_2}{\omega_2+\eta_2}\} < p_2 < p_1 < \infty$,
    where $\omega_i$ is the upper dimension of $X_i$ for $i =
    1$, $2$. Then the following two assertions hold.
    \begin{enumerate}
        \item[(a)] Let $T$ be a linear operator that is
            bounded from $H^{p_2}( X_1\times X_2 )$ to
            $L^{p_2}( X_1\times X_2 )$ and from $H^{p_1}(
            X_1\times X_2 )$ to $L^{p_1}( X_1\times X_2 )$.
            Then $T$ is bounded from $H^p( X_1\times X_2 )$
            to $L^p( X_1\times X_2 )$ for all $p$ with $p_2
            < p < p_1$.

        \item[(b)] Suppose $T$ is bounded on $H^{p_2}(
            X_1\times X_2 )$ and on $H^{p_1}( X_1\times X_2
            )$. Then $T$ is bounded on $H^p( X_1\times X_2
            )$ for all $p$ with $p_2 < p < p_1$.
    \end{enumerate}
\end{theorem}

We first prove Theorem \ref{theorem C-Z decomposition for Hp}.

\begin{proof}[\bf Proof of Theorem \ref{theorem C-Z decomposition for Hp}]
Suppose that $f\in H^p( X_1\times X_2 )$ and $\alpha > 0$. Let
$\Omega_\ell := \{(x_1,x_2)\in  X_1\times X_2 :
\widetilde{S}(f)(x_1,x_2)
> \alpha 2^\ell\}$, where $\widetilde{S}(f)$ is the discrete product square
function defined in~\eqref{g function}.

Let $${\mathcal R}_0 := \left\{R = Q_{\alpha_1}^{k_1}\times
Q_{\alpha_2}^{k_2}: k_1, k_2\in\mathbb{Z},\alpha_1 \in
\mathscr{Y}^{k_1}, \alpha_2 \in \mathscr{Y}^{k_2}, \mu(R\cap
\Omega_0)<\frac{1}{2A_0}\mu(R)\right\}$$ and for $\ell\ge 1$
\begin{eqnarray*}
    {\mathcal R}_\ell
    &:=& \bigg\{R=Q_{\alpha_1}^{k_1}\times Q_{\alpha_2}^{k_2},\
        k_1, k_2\in\mathbb{Z},\alpha_1 \in \mathscr{Y}^{k_1},
        \alpha_2 \in \mathscr{Y}^{k_2}\\
    &&\hskip.6cm\,\text{such that}\,\,
        \mu(R\cap \Omega_{\ell - 1})
        \ge \frac{1}{2A_0}\mu(R)\,\,\text{and}\,\,
        \mu(R\cap \Omega_\ell)
        < \frac{1}{2A_0}\mu(R)\bigg\}.
\end{eqnarray*}
Applying the wavelet reproducing formula from Theorem 3.11, we have
\begin{eqnarray*}
    f(x_1,x_2)
    &=& \sum_{k_1}\sum_{\alpha_1 \in \mathscr{Y}^{k_1}}
        \sum_{k_2}\sum_{\alpha_2 \in \mathscr{Y}^{k_2}}
        \langle f,\psi_{\alpha_1}^{k_1}\psi_{\alpha_2}^{k_2} \rangle
        \psi_{\alpha_1}^{k_1}(x_1)\psi_{\alpha_2}^{k_2}(x_2)\\
    &=& \sum_{\ell\ge 1} \sum_{R = Q_{\alpha_1}^{k_1}\times Q_{\alpha_2}^{k_2}\in {\mathcal R}_\ell}
        \langle f,\psi_{\alpha_1}^{k_1}\psi_{\alpha_2}^{k_2} \rangle
        \psi_{\alpha_1}^{k_1}(x_1)\psi_{\alpha_2}^{k_2}(x_2)\\
    && {}+ \sum_{R=Q_{\alpha_1}^{k_1}\times Q_{\alpha_2}^{k_2}\in {\mathcal R}_0}
        \langle f,\psi_{\alpha_1}^{k_1}\psi_{\alpha_2}^{k_2} \rangle
        \psi_{\alpha_1}^{k_1}(x_1)\psi_{\alpha_2}^{k_2}(x_2)\\
    &=:& b(x,y) + g(x,y).
\end{eqnarray*}
When $p_1 > 1$, the $L^p(X_1 \times X_2)$, $1 < p < \infty$,
estimate for the Littlewood--Paley square function
implies that
\begin{eqnarray*}
    \|g\|_{L^{p_1}(X_1 \times X_2)}
    \le C \Big\|\Big\{
        \sum_{R=Q_{\alpha_1}^{k_1}\times Q_{\alpha_2}^{k_2}\in {\mathcal R}_0}\Big|
        \langle \psi_{\alpha_1}^{k_1}\psi_{\alpha_2}^{k_2},f \rangle
        \widetilde{\chi}_{Q_{\alpha_1}^{k_1}}(x_1)
        \widetilde{\chi}_{Q_{\alpha_2}^{k_2}}(x_2) \Big|^2 \Big\}^{1/2}
        \Big\|_{L^{p_1}(X_1 \times X_2)}.
\end{eqnarray*}

Next, we estimate $\|g\|_{H^{p_1}(X_1\times X_2)}$ when
$\max\{\frac{\omega_1}{\omega_1 + \eta_1},
\frac{\omega_2}{\omega_2 + \eta_2}\} < p_1 \leq 1$. We
estimate the $H^{p_1}(X_1\times X_2)$ norm directly. To this
end, using the wavelet coefficients of $g$, we observe that
\begin{eqnarray*}
    \|g\|_{H^{p_1}(X_1 \times X_2)}
    &\le& \Big\|\Big\{\sum_{k'_1}\sum_{\alpha'_1\in\mathscr{Y}^{k'_1}}
        \sum_{k'_1}\sum_{\alpha_2\in\mathscr{Y}^{k'_2}}\big|
        \langle \psi_{\alpha'_1}^{k'_1}\psi_{\alpha'_2}^{k'_2},g \rangle
        \widetilde{\chi}_{Q_{\alpha'_1}^{k'_1}}(x_1)
        \widetilde{\chi}_{Q_{\alpha'_2}^{k'_2}}(x_2) \big|^2 \Big\}^{1/2}\Big\|_{L^{p_1}(X_1 \times X_2)}\\
    &\leq & C\Big\|\Big\{ \sum\limits_{R=Q_{\alpha_1}^{k_1}\times Q_{\alpha_2}^{k_2}
        \in {\mathcal R}_0}\Big|
        \langle \psi_{\alpha_1}^{k_1}\psi_{\alpha_2}^{k_2},f \rangle
        \widetilde{\chi}_{Q_{\alpha_1}^{k_1}}(x_1)
        \widetilde{\chi}_{Q_{\alpha_2}^{k_2}}(x_2) \Big|^2\Big\}^{1/2}\Big\|_{L^{p_1}(X_1 \times X_2)}.
\end{eqnarray*}
Thus for all $p_1$ with $\max\{\frac{\omega_1}{\omega_1 +
\eta_1}, \frac{\omega_2}{\omega_2 + \eta_2}\} < p_1 < \infty$,
we have
\begin{eqnarray*}
    \|g\|_{H^{p_1}(X_1 \times X_2)}
    \leq C \Big\|\Big\{ \sum\limits_{R=Q_{\alpha_1}^{k_1}\times Q_{\alpha_2}^{k_2}
        \in {\mathcal R}_0}\Big|
        \langle \psi_{\alpha_1}^{k_1}\psi_{\alpha_2}^{k_2},f \rangle
        \widetilde{\chi}_{Q_{\alpha_1}^{k_1}}(x_1)
        \widetilde{\chi}_{Q_{\alpha_2}^{k_2}}(x_2) \Big|^2\Big\}^{1/2}\Big\|_{L^{p_1}(X_1 \times X_2)}.
\end{eqnarray*}

\noindent{\bf Claim 1:} We claim that
\begin{eqnarray*}
    \lefteqn{\int_{\widetilde{S}(f)(x_1,x_2)\le \alpha} \widetilde{S}(f)(x_1,x_2)^{p_1} \, d\mu_1(x_1)\, d\mu_2(x_2)}\hspace{1cm}\\
    &&\ \geq C
        \Big\|\Big\{ \sum_{R=Q_{\alpha_1}^{k_1}\times Q_{\alpha_2}^{k_2}\in {\mathcal R}_0}\Big|
        \langle \psi_{\alpha_1}^{k_1}\psi_{\alpha_2}^{k_2},f \rangle
        \widetilde{\chi}_{Q_{\alpha_1}^{k_1}}(x_1)
        \widetilde{\chi}_{Q_{\alpha_2}^{k_2}}(x_2) \Big|^2\Big\}^{1/2}\Big\|_{L^{p_1}(X_1 \times X_2)}.
\end{eqnarray*}
This implies that
\begin{eqnarray*}
    \|g\|_{H^{p_1}(X_1 \times X_2)}&\le& C \int_{\widetilde{S}(f)(x_1,x_2)\le \alpha} \widetilde{S}(f)(x_1,x_2)^{p_1} \, d\mu_1(x_1)\, d\mu_2(x_2)\\
    &\le&
    C\alpha^{p_1-p}  \int_{\widetilde{S}(f)(x_1,x_2)\le \alpha} \widetilde{S}(f)(x_1,x_2)^{p} \, d\mu_1(x_1)\, d\mu_2(x_2)\\
    &\le & C\alpha^{p_1-p}\|f \|^p_{H^p(X_1 \times X_2)}.
\end{eqnarray*}
To show Claim 1, we choose $0 < q < p_1$ and $q < 2$, and
observe that
\begin{align*}
    &\int_{\widetilde{S}(f)(x_1,x_2)\leq \alpha} \widetilde{S}(f)(x_1,x_2)^{p_1} \, d\mu_1(x_1)\, d\mu_2(x_2)\\
    &= \int_{\Omega_0^c}\Big\{
        \sum_{k_1}\sum_{\alpha_1\in\mathscr{Y}^{k_1}}\sum_{k_1}\sum_{\alpha_2\in\mathscr{Y}^{k_2}}
        \Big| \langle \psi_{\alpha_1}^{k_1}\psi_{\alpha_2}^{k_2},f \rangle
        \widetilde{\chi}_{Q_{\alpha_1}^{k_1}}(x_1)
        \widetilde{\chi}_{Q_{\alpha_2}^{k_2}}(x_2) \Big|^2 \Big\}^{p_1/2} \, d\mu_1(x_1)\, d\mu_2(x_2)\\
    &\geq C\int_{\Omega_0^c} \Big\{ \sum_{R=Q_{\alpha_1}^{k_1}\times Q_{\alpha_2}^{k_2}\in {\mathcal R}_0}\big| \langle \psi_{\alpha_1}^{k_1}\psi_{\alpha_2}^{k_2},f \rangle
        \widetilde{\chi}_{Q_{\alpha_1}^{k_1}}(x_1)
        \widetilde{\chi}_{Q_{\alpha_2}^{k_2}}(x_2) \big|^2 \Big\}^{\frac{p_1}{2}}\, d\mu_1(x_1)\, d\mu_2(x_2)\\
    &= C\int_{X_1\times X_2} \Big\{ \sum_{R=Q_{\alpha_1}^{k_1}\times Q_{\alpha_2}^{k_2}\in {\mathcal R}_0}\big| \langle \psi_{\alpha_1}^{k_1}\psi_{\alpha_2}^{k_2},f \rangle
        \widetilde{\chi}_{Q_{\alpha_1}^{k_1}}(x_1)
        \widetilde{\chi}_{Q_{\alpha_2}^{k_2}}(x_2)\chi_{\Omega_0^c}(x_1,x_2) \big|^2\Big\}^{\frac{p_1}{2}}\, d\mu_1(x_1)\, d\mu_2(x_2)
        \\
    &\geq  C \int_{ X_1\times X_2 } \bigg[\Big\{\sum_{R=Q_{\alpha_1}^{k_1}\times Q_{\alpha_2}^{k_2}\in {\mathcal R}_0} \\
    & \hspace{3.5cm}
        \Big(M_s \big(\langle \psi_{\alpha_1}^{k_1}\psi_{\alpha_2}^{k_2},f \rangle^q\mu(R)^{-q}
        \chi_{R\cap \Omega_0^c}\big)(x_1,x_2)\Big)^{\frac{2}{q}}\Big\}^{\frac{q}{2}}\bigg]^{\frac{p_1}{q}}\, d\mu_1(x_1)\, d\mu_2(x_2)\\
    &\geq  C \int_{X_1\times X_2 } \Big\{ \sum_{R=Q_{\alpha_1}^{k_1}\times Q_{\alpha_2}^{k_2}\in {\mathcal R}_0}\Big| \langle \psi_{\alpha_1}^{k_1}\psi_{\alpha_2}^{k_2},f \rangle
        \widetilde{\chi}_{Q_{\alpha_1}^{k_1}}(x_1)
        \widetilde{\chi}_{Q_{\alpha_2}^{k_2}}(x_2)\Big|^2\Big\}^{\frac{p_1}{2}}\, d\mu_1(x_1)\, d\mu_2(x_2),
\end{align*}
where in the last inequality we have used the fact that
$\mu(\Omega_0^c\cap R)\geq \frac{1}{2}\mu(R)$ for $R\in
{\mathcal R}_0$, and thus $\chi_R(x_1,x_2)\leq 2^{\frac{1}{q}}
M_s(\chi_{R\cap\Omega_0^c})^{\frac{1}{q}}(x_1,x_2),$ and in the
second to last inequality we have used the vector-valued
Fefferman--Stein inequality for strong maximal functions:
$$
    \Big\| \Big\{\sum\limits_{k=1}^\infty
        M_s(f_k)^r \Big\}^{\frac{1}{r}} \Big\|_{L^p(X_1\times X_2)}
    \le C \Big\| \Big\{\sum\limits_{k=1}^\infty
        |f_k|^r \Big\}^{\frac{1}{r}} \Big\|_{L^p(X_1\times X_2)},
$$
with the exponents $r = 2/q > 1$ and $p = p_1/q > 1$. Thus the
claim follows.

Let $\widetilde{\Omega_\ell}$ be the enlargement of the set
$\Omega_\ell$ given by $\widetilde{\Omega_\ell} :=
\{(x_1,x_2)\in X_1\times X_2: M_s(\chi_{\Omega_\ell})
> (2A_0)^{-1}\}$.

\medskip
\noindent {\bf Claim 2:} For $p_2\le 1$,
\begin{eqnarray*}
    \Big\|\sum_{R=Q_{\alpha_1}^{k_1}\times Q_{\alpha_2}^{k_2}\in {\mathcal R}_\ell}
        \langle f,\psi_{\alpha_1}^{k_1}\psi_{\alpha_2}^{k_2} \rangle
        \psi_{\alpha_1}^{k_1}(x_1)\psi_{\alpha_2}^{k_2}(x_2)  \Big\|^{p_2}_{H^{p_2}(X_1 \times X_2)}
        \le C(2^{\ell}\alpha)^{p_2} \mu(\widetilde{\Omega}_{\ell-1}).
\end{eqnarray*}
Claim 2 implies that
\begin{eqnarray*}
    ||b||^{p_2}_{H^{p_2}(X_1 \times X_2)}&\le & \sum_{\ell\ge 1}(2^{\ell}\alpha)^{p_2}\mu(\widetilde{\Omega}_{\ell-1})
    \le C \sum_{\ell\ge 1}(2^{\ell}\alpha)^{p_2}\mu(\Omega_{\ell-1})\\
    &\le & C\int_{\widetilde{S}(f)(x_1,x_2)>\alpha}\widetilde{S}(f)(x_1,x_2)^{p_2} \, d\mu_1(x_1)\, d\mu_2(x_2) \\
    &\le & C\alpha^{p_2-p}\int_{\widetilde{S}(f)(x_1,x_2)>\alpha}\widetilde{S}(f)(x_1,x_2)^{p} \, d\mu_1(x_1)\, d\mu_2(x_2)\\
    &\le& C\alpha^{p_2-p}\|f\|^p_{H^p(X_1 \times X_2)}.
\end{eqnarray*}

To show Claim 2, note that
\begin{eqnarray*}
    \lefteqn{\Big\|\sum_{R=Q_{\alpha_1}^{k_1}\times Q_{\alpha_2}^{k_2}\in {\mathcal R}_\ell}
        \langle f,\psi_{\alpha_1}^{k_1}\psi_{\alpha_2}^{k_2} \rangle
        \psi_{\alpha_1}^{k_1}(x_1)\psi_{\alpha_2}^{k_2}(x_2)  \Big\|^{p_2}_{H^{p_2}(X_1 \times X_2)}} \hspace{1cm}\\
    &\le &  C \Big\|\Big\{ \sum_{R=Q_{\alpha_1}^{k_1}\times Q_{\alpha_2}^{k_2}\in {\mathcal R}_\ell}\Big|
        \langle \psi_{\alpha_1}^{k_1}\psi_{\alpha_2}^{k_2},f \rangle
        \widetilde{\chi}_{Q_{\alpha_1}^{k_1}}(x_1)
        \widetilde{\chi}_{Q_{\alpha_2}^{k_2}}(x_2) \Big|^2\Big\}^{1/2}\Big\|_{L^{p_2}(X_1 \times X_2)}.
\end{eqnarray*}
Then
\begin{eqnarray*}
    \lefteqn{\sum_{\ell=1}^\infty (2^\ell
        \alpha)^{p_2}\mu(\widetilde{\Omega}_{\ell-1})} \\
    &\ge &
        \int_{\widetilde{\Omega}_{\ell-1}\backslash
        \Omega_{\ell}}\widetilde{S}(f)^{p_2}(x_1,x_2)\, d\mu_1(x_1)\, d\mu_2(x_2)\\
    & = &
        \int_{\widetilde{\Omega}_{\ell-1}\backslash \Omega_{\ell}}\Big\{
        \sum_{k_1}\sum_{\alpha_1\in\mathscr{Y}^{k_1}}\sum_{k_1}\sum_{\alpha_2\in\mathscr{Y}^{k_2}}
        \Big| \langle \psi_{\alpha_1}^{k_1}\psi_{\alpha_2}^{k_2},f \rangle
        \widetilde{\chi}_{Q_{\alpha_1}^{k_1}}(x_1)
        \widetilde{\chi}_{Q_{\alpha_2}^{k_2}}(x_2) \Big|^2 \Big\}^{\frac{p_2}{2}}\, d\mu_1(x_1)\, d\mu_2(x_2)\\
    &\geq & \int_{X_1\times X_2} \Big\{ \sum_{R=Q_{\alpha_1}^{k_1}\times Q_{\alpha_2}^{k_2}\in {\mathcal R}_0} \\
    && \hspace{2.5cm} \big| \langle \psi_{\alpha_1}^{k_1}\psi_{\alpha_2}^{k_2},f \rangle
        \widetilde{\chi}_{Q_{\alpha_1}^{k_1}}(x_1)
        \widetilde{\chi}_{Q_{\alpha_2}^{k_2}}(x_2)\chi_{\widetilde{\Omega}_{\ell-1}\backslash \Omega_{\ell}}(x_1,x_2) \big|^2 \Big\}^{\frac{p_2}{2}}\, d\mu_1(x_1)\, d\mu_2(x_2)\\
    &\geq & C \int_{X_1\times X_2 } \Big\{ \sum_{R=Q_{\alpha_1}^{k_1}\times Q_{\alpha_2}^{k_2}\in {\mathcal R}_0}\Big| \langle \psi_{\alpha_1}^{k_1}\psi_{\alpha_2}^{k_2},f \rangle
        \widetilde{\chi}_{Q_{\alpha_1}^{k_1}}(x_1)
        \widetilde{\chi}_{Q_{\alpha_2}^{k_2}}(x_2)\Big|^2\Big\}^{\frac{p_2}{2}}\, d\mu_1(x_1)\, d\mu_2(x_2),
\end{eqnarray*}
where the last inequality follows from the fact that if $R\in
{\mathcal R}_\ell$ then $R\subset \widetilde{\Omega}_{\ell-1}$,
and therefore $\mu\big(R\cap
(\widetilde{\Omega}_{\ell-1}\backslash
\Omega_\ell)\big)>\frac{1}{2}\mu(R)$. This establishes Claim~2.
Hence, the proof of Theorem~\ref{theorem C-Z decomposition for
Hp} is complete.
\end{proof}

We end our paper by proving the interpolation theorem on Hardy
spaces~$H^p(X_1\times X_2)$.

\begin{proof}[Proof of Theorem~\ref{theorem interpolation Hp}]
(a) Suppose that $T$  is bounded from $H^{p_2}(X_1\times X_2)$
to $L^{p_2}(X_1\times X_2)$ and  from $H^{p_1}(X_1\times X_2)$
to $L^{p_1}(X_1\times X_2)$. For each given $\lambda > 0$ and
$f\in H^p(X_1\times X_2)$, by the Calder\'on--Zygmund
decomposition we may write
\[
    f(x_1,x_2)
    = g(x_1,x_2) + b(x_1,x_2)
\]
with
\[
    \|g\|^{p_1}_{H^{p_1}(X_1 \times X_2)}
    \le C\lambda^{p_1-p}\|f\|_{H^p(X_1 \times X_2)}^p
    \,\,\,\ {\rm and\ }\,\,
    \|b\|_{H^{p_2}(X_1 \times X_2)}^{p_2}
    \le C\lambda^{p_2-p}\|f\|_{H^p(X_1 \times X_2)}^p.
\]

Moreover, we have proved the estimates
$$
\|g\|^{p_1}_{H^{p_1}(X_1 \times X_2)}\le C\int_{\widetilde{S}(f)(x_1,x_2)
\le \alpha}\widetilde{S}(f)^{p_1}(x_1,x_2)\, d\mu_1(x_1)\, d\mu_2(x_2)
$$
and
$$
\|b\|^{p_2}_{H^{p_2}(X_1 \times X_2)}\le C\int_{\widetilde{S}(f)(x_1,x_2)
> \alpha}\widetilde{S}(f)^{p_2}(x_1,x_2)\, d\mu_1(x_1)\, d\mu_2(x_2),
$$
which imply that
\begin{eqnarray*}
    \|Tf\|^p_{L^p(X_1 \times X_2)}
    &= &  p\int_0^\infty \alpha^{p-1} \mu(\{(x_1,x_2): |Tf(x_1,x_2)|>\alpha\}) \, d\alpha\\
    &\le & p\int_0^\infty \alpha^{p-1}\mu(\{(x_1,x_2): |Tg(x_1,x_2)|>\alpha/2\}) \, d\alpha\\
    && {}+ p\int_0^\infty \alpha^{p-1}\mu(\{(x_1,x_2): |Tb(x_1,x_2)|>\alpha/2\}) \, d\alpha\\
    &\le & p\int_0^\infty \alpha^{p-p_1-1}\int_{\widetilde{S}(f)(x_1,x_2)\le \alpha}\widetilde{S}(f)^{p_1}(x_1,x_2)\, d\mu_1(x_1)\, d\mu_2(x_2)  \, d\alpha\\
    && {}+ p\int_0^\infty \alpha^{p-p_2-1}\int_{\widetilde{S}(f)(x_1,x_2)>\alpha}\widetilde{S}(f)^{p_2}(x_1,x_2)\, d\mu_1(x_1)\, d\mu_2(x_2)  \, d\alpha\\
    &\le & C\|f\|^p_{H^p(X_1 \times X_2)}
\end{eqnarray*}
for all $p$ with $p_2 < p < p_1$. Hence $T$ is bounded from
$H^p(X_1 \times X_2)$ to $L^p(X_1 \times X_2)$, as required.

\noindent (b) We turn to the second assertion. For each given
$\lambda > 0$ and $f\in H^p(X_1 \times X_2)$, by the
Calder\'on--Zygmund decomposition again we have
\begin{eqnarray*}
    &&\hskip-1cm \mu(\{(x_1,x_2): |\widetilde{S}(Tf)(x_1,x_2)|>\alpha\})\\
    &\le & \mu(\{(x_1,x_2): |\widetilde{S}(Tg)(x_1,x_2)|>\frac{\alpha}{2}\})+\mu(\{(x_1,x_2): |\widetilde{S}(Tb)(x_1,x_2)|>\frac{\alpha}{2}\})\\
    &\le & C\alpha^{-p_1}\|Tg\|_{H^{p_1}(X_1 \times X_2)}^{p_1}+ C\alpha^{-p_2}\|Tb\|_{H^{p_2}(X_1 \times X_2)}^{p_2}\\
    &\le &  C\alpha^{-p_1}\|g\|^{p_1}_{H^{p_1}(X_1 \times X_2)}+ C\alpha^{-p_2}\|b\|^{p_2}_{H^{p_2}(X_1 \times X_2)}\\
    &\le  &  C \alpha^{-p_1}\int_{\widetilde{S}(f)(x_1,x_2)\le \alpha}\widetilde{S}(f)^{p_1}(x_1,x_2)\, d\mu_1(x_1)\, d\mu_2(x_2)\\
    && {}+ C\alpha^{-p_2}\int_{\widetilde{S}(f)(x_1,x_2)>
    \alpha}\widetilde{S}(f)^{p_2}(x_1,x_2)\, d\mu_1(x_1)\, d\mu_2(x_2).
\end{eqnarray*}
Therefore $\| \widetilde{S}(Tf)\|_{L^p(X_1\times X_2)}\le
C\|\widetilde{S}(f)\|_{H^p(X_1 \times X_2)}$. Hence $\|
Tf\|_{H^p(X_1 \times X_2)}\le C\|f\|_{H^p(X_1 \times X_2)}$ for
all $p$ with $p_2 < p < p_1$, as required.
\end{proof}



\bigskip

\noindent Department of Mathematics, Auburn University, AL
36849-5310, USA.

\noindent {\it E-mail address}: \texttt{hanyong@auburn.edu}

\medskip

\noindent School of Information Technology and Mathematical
Sciences, University of South Australia, Mawson Lakes SA 5095,
Australia,

\smallskip
\noindent and
\smallskip

\noindent Department of Mathematics, Macquarie University, NSW
2019, Australia.

\noindent {\it E-mail address}: \texttt{ji.li@mq.edu.au}

\medskip

\noindent School of Information Technology and Mathematical
Sciences, University of South Australia, Mawson Lakes SA 5095,
Australia.

\noindent {\it E-mail address}:
\texttt{Lesley.Ward@unisa.edu.au}

\end{document}